\documentclass[review,12pt]{elsarticle}
\pdfoutput=1
\makeatletter
\@addtoreset{equation}{section}
\makeatother

\usepackage{verbatim}
\usepackage{amssymb,amsmath,color}
\usepackage{graphicx,amsthm}
\usepackage{subfigure,txfonts}
\usepackage[none]{hyphenat}
\usepackage{latexsym,bm}
\usepackage{booktabs}
\usepackage{cases}
\usepackage{lineno}%paper line number

\biboptions{numbers,sort&compress}
\date{\today}
\journal{}
\begin{document}
\newtheorem{The}{Theorem}[section]
\newtheorem{lem}[The]{Lemma}
\newtheorem{rem}{Remark}[section]
\newtheorem{ass}{Assumption}[section]
\newtheorem{prop}{Proposition}

%\maketitle
\begin{frontmatter}

\title{Coupled iterative analysis for  the stationary thermally coupled  inductionless MHD system based on charge-conservative finite element method}

\author{Shitian Dong}
\address{College of Mathematics and System Sciences,
Xinjiang University, Urumqi 830046, P.R. China}
\author{Haiyan Su \corref{correspondingauthor}}
\address{College of Mathematics and System Sciences,
Xinjiang University, Urumqi 830046, P.R. China}
\cortext[correspondingauthor]{Corresponding author: shymath@xju.edu.cn (H. Su)}

\ead{shymath@xju.edu.cn (H. Su), cmd@stu.xju.edu.cn (S. Dong)}

\begin{abstract}
This paper mainly considers three iterations based on charge-conservative finite element approx- imation in Lipschitz domain for the stationary thermally coupled  inductionless MHD equations. Based on the hybrid finite element method, the unknowns of hydrodynamic are discretized by the stable velocity-pressure finite element pair, and the current density along with electric potential are  similarly discretized by the comforming finite element pair in $\bm{H}(\mbox{div}, \Omega)\times L^2(\Omega)$.  And on account of the strong nonlinearity of the equations, we present three coupled iterative methods, namely, the Stokes, Newton and Oseen iteration and the convergence and stability under different uniqueness conditions are analyzed strictly. It is proved especially that the error estimates of velocity, current density, temperature and pressure do not depend on potential. The theoretical analysis is validated by the given numerical results, and for the proposed methods, the applicability and effectiveness are demonstrated.
\end{abstract}

\end{frontmatter}
\section{Introduction}
The inductionless magneto-hydrodynamics, referred to as IMHD,  is the theory that simulates the interaction between electromagnetic field and conductive fluid as the magnetic Reynolds number is small. 
IMHD model plays a very much important role in wide engineering applications, such as
 aluminum electrolysis, liquid metal magnetic pumps, and  fusion reactor blankets among others \cite{abdou2001exploration,moreau1990magnetohydrodynamics,20101On}.
 It is worth noting that the influence of temperature on the fluid cannot be ignored in practical applications. Thermal effects often plays a significant role in industrial MHD flows (i.e. metal hardening, semiconductor manufacturing, etc) \cite{meir1995thermally,ravindran2018decoupled,hughes1966electromagnetodynamics,davidson2002introduction,lifshits2012magnetohydrodynamics,gerbeau2006mathematical}. Thus, it is necessary to study the thermally coupled stationary IMHD problem.

The divergence-free condition for current density, namely $\nabla$$\cdot$$\bm{J}$$=$$0$, means conservation of charge for the IMHD equations. It is an important physical law of electromagnetics which takes an greatly important part in maintaining the computational accuracy of simulation of magnetohydrodynamic system. How to design a numerical method that can guarantee $\nabla$$\cdot$$\bm{J}_h$$=$$0$ has attracted much attention \cite{layton1994two,2021Coupled,2018A,hu2017stable,greif2010mixed}.
The charge-conservative and consistent finite volume schemes for IMHD equations with unstructured and structured meshes based on post- processing have been proposed in \cite{ni2007current,ni2007current1}.
 The charge-conservative played a key role in forming closed loops for the streamlines from the conservative scheme was shown in \cite{ni2012consistent}.
In addition, numerical experiments conducted therein showed that streamlines with non-conserved charge schemes fail to form closed loops at corners. A charge-conservative finite element method that yields
an exactly divergence-free current density directly was proposed in \cite{0a} to solve the IMHD equations. Recently the charge-conservative finite element approximation and three classic coupled iterations for stationary IMHD equations was proposed in \cite{2021Coupled}.

 Compared with the conservation scheme of IMHD, the work on conservation scheme for thermally coupled IMHD models is relatively few. A new family of recursive block LU precondi- tioners was designed and tested for solving the thermally coupled IMHD equations in \cite{2014Block}. For the  time- dependent IMHD model coupled thermal problem, a charge conservative fully discrete finite element scheme was proposed and analyzed in \cite{long2022convergence}.

Different from \cite{long2022convergence}, we mainly consider the iterative methods based on charge-conservative for the stationary thermally coupled IMHD equations in the Lipschitz domain in this paper. The core concept of the paper is to design effective iterative methods which can deal with the conservation of charge and the strong nonlinear characteristics of the thermally coupled IMHD equations. Our specific ideas are as follows: First, for the spatial discretization, we use $\bm{H}(\mbox{div}, \Omega)$-conforming face element and $L^2(\Omega)$-conforming volume element to discrete separately current density $\bm{J}$ and potential $\phi$, so as to generate directly the exact divergence-free current density. Second, due to the strong nonlinearity of the equations, three iterative methods are proposed and the strong uniqueness conditions: $0<\sigma<\frac{1}{4}$ (Stokes iterative method),  $0<\sigma<\frac{1}{3}$ (Newton iterative method), the uniqueness condition: $0<\sigma<1$ (Oseen iterative method), respectively, are obtained, where $\sigma:=\frac{\lambda_2\|L\|_*}{C_{\min}^2}+\frac{\lambda_2\lambda_q\lambda_{\psi}}{C_{\min}^2}
+\frac{\lambda_2\lambda_q\lambda_{\psi}}{C_{\min}}+\frac{\lambda_2\|L\|_*}{C_{\min}}<1$. Third,
%under the Assumptions 3.1 and 3.2,
the optimal error estimates with respect to iteration number $n$ and mesh size $h$ for the three methods are presented.
Finally, numerical examples verify our theoretical results and show the uniqueness range of the three iterative methods.

The structure of this paper is as below: In the second section, some symbols of Sobolev space are introduced, and a weak form is proposed, and then the well posedness of the continuous problem is further proved. In the third section, using the hybrid finite element method, we give the well posed and the optimal error estimates for the discrete problem. In the fourth section, the three iteration methods of the equations (\ref{eq:3.2}) are introduced, and then stabilities and convergence for them are further analyzed theoretically below the conditions differently. In the fifth section, through some numerical experiments we test and verify the performance of the numerical scheme.

\section{Preliminaries}
In the paper, we consider  the stationary thermally coupled inductionless MHD equations as follows:
\begin{equation}\label{eq:2.1}
\begin{split}
Pr\Delta \bm{u} +(\bm{u}\cdot\nabla)\bm{u} + Pr\nabla p -\kappa \bm{J}\times \bm{B} -PrRa\theta\bm{i} & =\bm{f},  \quad \mbox{in} \ \Omega, \\
\bm{J}+\nabla \phi-\bm{u}\times \bm{B}& =\bm{g}, \quad \mbox{in} \ \Omega,\\
-\Delta \theta + \bm{u}\cdot \nabla \theta & = \varphi, \quad \mbox{in} \ \Omega,\\
\nabla \cdot \bm{u}=0, \ \nabla \cdot \bm{J} & =0, \quad \mbox{in} \ \Omega,\\
\end{split}
\end{equation}
where $\Omega$ is a bounded domain in $R^d,$ $ d=2, 3$ whose boundary $\partial\Omega$ is Lipschitz-continuous, $\bm{u}$  the fluid velocity, $p$  the hydrodynamic pressure, $\bm{J}$  the current density, $\phi$  the electric potential, $\theta$  the temperature, the Rayleigh number $Ra$, $\bm{i}$ the unit basis vector, $\kappa$ the coupling number, $\bm{f}$ and $\bm{g}$ are given external force terms, $\varphi$ a given a given heat source, $\bm{B}$ is a hypothetical given applied magnetic field.
The system of the equations is supplemented by homogeneous  boundary conditions:
\begin{equation}\label{eq:2.2}
\begin{split}
\bm{u}=\bm{0}, \quad \bm{J}\cdot\bm{n}=0, \quad \theta=0, \quad \mbox{on} \ \partial\Omega.
\end{split}
\end{equation}
We rewrite those 2D variables using 3D fashion for the sake of studying 2D and 3D model simultaneously. Hence, the 2D variables $\bm{\vartheta}$  are described by $\bm{\vartheta}=(\vartheta_1(x), \vartheta_2(x), 0),\ \bm{\vartheta}=\bm{u}, \bm{J}, \bm{f}, \bm{g}$.
On the contrary, the applied magnetic field is $\bm{B}=(0, 0, B_{3}(x))^{T}.$

We introduce some Sobolev spaces used in this article. Let $L^p(\Omega)$ be $p$-th integrable function space, whose normal is capable of being defined by $\|\cdot\|_p$. Inner product and norm of $L^2(\Omega)$ can be described as :
$$(u,v):=\int_{\Omega}{uv}dx, \quad \|v\|_0=\|v\|_2:=(v,v)^{\frac{1}{2}}.$$
The standard Sobolev space $H^{m}$  has its corresponding seminorm and norm defined by $|\cdot|_{m,\Omega}$ and $\|\cdot\|_{m,\Omega}$, respectively.
For convenience's sake, throughout the paper, vector is represented by  bold-type letter and scalar by roman alphabet, while we introduce the following function spaces:
$$\bm{X}:= \bm{H}_0^1(\Omega), \quad E := L_0^2(\Omega), \quad
\bm{Y}:= \bm{H}_0(\mbox{div},\Omega), \quad S := L_0^2(\Omega), \quad Z:=H_0^1(\Omega),$$
and product spaces $\bm{X}\times\bm{Y}$ and $E\times S$ with the energy norms denoted as
$$\|(\bm{w},\bm{K})\|_1=(\|\nabla\bm{w}\|_0^2,\|\bm{K}\|_{div}^2)^{\frac{1}{2}},\quad
\|(q,\psi)\|_0=(\|q\|_0^2+\|\psi\|_0^2)^\frac{1}{2}.$$
Here, the div-norm is defined by $\|K\|_{div}=(\|K\|_0^2+\|\nabla \cdot K\|_2)^{\frac{1}{2}}.$
Further, we define the norm of the source term by
$$\|\bm{f}\|_{-1}=\mathop{\sup}\limits_{\bm{v}\neq\bm{0}\in \bm{X}}\frac{\langle \bm{f},\bm{v}\rangle}{\|\nabla \bm{v}\|_0},\quad \|L\|_* = (\|\bm{f}\|_{-1}^2+\kappa^2\|\bm{g}\|^2)^\frac{1}{2},\quad \lambda_{\varphi}=\|\varphi\|_{-1}=\mathop{\sup}\limits_{r\neq 0 \in Z}\frac{\langle \varphi,r \rangle}{\|\nabla r\|_0}. $$
In order to make the weak form of equations (\ref{eq:2.1})-(\ref{eq:2.2}) clearer,
we define the following bilinear forms: for $\forall \bm{w}, \bm{v}, \bm{u}\in \bm{X}, q\in E, \bm{F}, \bm{K}\in \bm{Y}, \psi\in S$ and $\theta, r\in Z$
\begin{equation*}
\begin{split}
&a_s(\bm{v},\bm{w})=R_e^{-1}(\nabla\bm{v},\nabla\bm{w}),\quad a_m(\bm{F},\bm{K})=\kappa (\bm{F},\bm{K}),\quad b_s(q,\bm{v})=-Pr(q,\nabla\cdot v),\\
&e(\theta,r)=(\nabla\theta,\nabla r),\quad d(\bm{K},\bm{v})=-\kappa(\bm{K}\times\bm{B},\bm{v}),\quad b_m(\psi,\bm{F})=-\kappa(\psi,\nabla\cdot\bm{F}),
\end{split}
\end{equation*}
and trilinear forms:
\begin{equation*}
\begin{split}
&c(\bm{u},\bm{v},\bm{w})=\frac{1}{2}((\bm{u}\cdot\nabla)\bm{v},\bm{w})-\frac{1}{2}((\bm{u}\cdot\nabla)\bm{w},\bm{v}),\quad h(\bm{u},\theta,r)=\frac{1}{2}((\bm{u}\cdot\nabla)\theta,r)-\frac{1}{2}((\bm{u}\cdot\nabla)r,\theta).
\end{split}
\end{equation*}
For the convenience of our analysis and proof below, we use a more compact way to express the weak form of (\ref{eq:2.1}) as follows:
find $(\bm{u},\bm{J},p,\phi,\theta)\in \bm{X}\times\bm{Y}\times E\times S \times Z$ such that
\begin{equation}\label{eq:2.3}
\left\{\begin{array}{c}
\begin{aligned}
&\mathcal{A}_0((\bm{u},\bm{J}),(\bm{v},\bm{K}))+\mathcal{A}_1((\bm{u},\bm{J}),(\bm{u},\bm{J}),(\bm{v},\bm{K}))\\
& +\mathcal{B}((p,\psi),(\bm{v},\bm{K}))+\mathcal{Q}(\theta,(\bm{v},\bm{K})) =\langle \bm{L},(\bm{v},\bm{K})\rangle, \\
&\mathcal{B}((q,\psi),(\bm{u},\bm{J})) =0,\\
&e(\theta,r)+\mathcal{H}((\bm{u},\bm{J}),\theta,r) =\Psi(r),
\end{aligned}
\end{array} \right.
\end{equation}
for $\forall (\bm{v},\bm{K},p,\psi,r)\in \bm{X}\times\bm{Y}\times E\times S \times Z$, where
\begin{equation*}
\begin{aligned}
&\mathcal{A}_0((\bm{u},\bm{J}),(\bm{v},\bm{K}))=a_s(\bm{u},\bm{v})+a_m(\bm{J},\bm{K})+d(\bm{J},\bm{v})-d(\bm{K},\bm{u}),\quad
\mathcal{Q}(\theta,(\bm{v},\bm{K}))=-PrRa\bm{i}(\theta,\bm{v}),\\
&\mathcal{A}_1((\bm{w},\bm{M}),(\bm{u},\bm{J}),(\bm{v},\bm{K}))=c_0(\bm{w},\bm{u},\bm{v}), \quad
\mathcal{B}((q,\psi),(\bm{v},\bm{K})) =b_s(q,\bm{v})+b_m(\phi,\bm{K}),\\
&\mathcal{H}((\bm{u},\bm{J}),\theta,r)=h_0(\bm{u},\theta,r), \quad
\langle \bm{L},(\bm{v},\bm{K})\rangle=\langle \bm{f},\bm{v}\rangle+\kappa(\bm{g},\bm{K}), \quad
\Psi(r)=\langle \varphi,r \rangle.
\end{aligned}
\end{equation*}

Furthermore, the kernel space is defined
\begin{equation*}
\begin{aligned}
\Upsilon=\{ (\bm{v},\bm{M})\in \bm{X}\times\bm{Y}:\mathcal{B}((q,\psi),(\bm{v},\bm{M}))=0,\ \forall (q,\psi)\in E\times S\}=\bm{V}\times\bm{U},
\end{aligned}
\end{equation*}
where
\begin{equation*}
\begin{aligned}
\bm{V}=\{\bm{v}\in \bm{X} : b_s(q,\bm{v})=0,\ \forall q\in E\},\quad \bm{U}=\{\bm{K}\in \bm{Y} : b_m(\psi,\bm{K})=0,\ \forall \psi \in S\},
\end{aligned}
\end{equation*}
and the following inequalities frequently used in our proof are valid in general Lipschitz polyhedra,
\begin{equation*}
\begin{aligned}
\|\bm{v}\|_{\bm{L}^6(\Omega)}\le \tilde{\lambda}_1\|\nabla \bm{v}\|_0, \quad \|\bm{v}\|_{\bm{L}^4(\Omega)}\le \tilde{\lambda}_2\|\nabla \bm{v}\|_0^2,\quad \|\bm{v}\|_0 \le \tilde{C}_{\Omega}\|\nabla \bm{v}\|_0\quad \forall \bm{v}\in \bm{X},
\end{aligned}
\end{equation*}
\begin{equation*}
\begin{aligned}
\|r\|_{L^6(\Omega)}\le \hat{\lambda}_1\|\nabla r\|_0, \quad \|r\|_{L^4(\Omega)}\le \hat{\lambda_2}\|\nabla r\|_0^2,\quad \|r\|_0 \le \hat{C}_{\Omega}\|\nabla r\|_0, 
 \quad \forall r\in Z,
\end{aligned}
\end{equation*}
where $\lambda_1$$=\max(\tilde{\lambda}_1, \hat{\lambda}_1)$, $\lambda_2$$=\max(\tilde{\lambda}_2, \hat{\lambda}_2)$ and $C_{\Omega}$$=\max(\tilde{C}_{\Omega}, \hat{C}_{\Omega})$, the positive constants $\tilde{\lambda}_1,$ $\tilde{\lambda}_2,$ $\hat{\lambda}_1,$ $\hat{\lambda}_2,$ $\tilde{C}_{\Omega},$ $\hat{C}_{\Omega},$ only depend on $\Omega$.

\begin{lem}
 The following estimates for bilinear and trilinear terms hold:\\
$(1).$ The bilinear form $\mathcal{A}_0((\cdot,\cdot),(\cdot,\cdot))$ is bounded on $(\bm{X}\times\bm{Y})\times(\bm{X}\times\bm{Y})$ and coercive on $(\bm{X}\times\bm{U})\times(\bm{X}\times\bm{U}),$ namely,
\begin{equation}\label{eq:2.4}
\begin{aligned}
|\mathcal{A}_0((\bm{u},\bm{J}),(\bm{v},\bm{K}))|\le C_{\max}\|(\bm{u},\bm{J})\|_1\|(\bm{v},\bm{K})\|_1,
\end{aligned}
\end{equation}
\begin{equation}\label{eq:2.5}
\begin{aligned}
\mathcal{A}_0((\bm{u},\bm{J}),(\bm{u},\bm{J}))\ge C_{\min}\|(\bm{u},\bm{J})\|_1^2,
\end{aligned}
\end{equation}
where \ $C_{\max}=2\max\{Pr, \kappa, \kappa\lambda_1\|\bm{B}\|_{\bm{L}^3(\Omega)}\},$ \ $C_{\min}=\min\{Pr, \kappa\}.$\\
(2). The continuous property of the bilinear form $\mathcal{B}((\cdot,\cdot),(\cdot,\cdot)),$ nambly, for any $(\bm{u}, q, \bm{J}, \psi)\in (\bm{X}\times E)\times(\bm{Y}\times S),$
\begin{equation}\label{eq:2.6}
\begin{aligned}
|\mathcal{B}((q,\psi),(\bm{v},\bm{K}))|\le \max\{Pr,\kappa\}\|(\bm{v},\bm{K})\|_1\|(q,\psi)\|_0.
\end{aligned}
\end{equation}
(3). The continuous property of the bilinear form $\mathcal{Q}(\cdot,(\cdot,\cdot))$, for any $(\theta,\bm{v},\bm{K})\in Z\times \bm{X}\times\bm{Y},$
\begin{equation}\label{eq:2.7}
\begin{aligned}
|\mathcal{Q}(\theta,(\bm{v},\bm{K}))| \le \lambda_q\|(\bm{v},\bm{K})\|_1\|\nabla \theta\|_0,
\end{aligned}
\end{equation}
where $\lambda_q=PrRaC_\Omega^2.$ \\
(4). The continuous property of the bilinear form $e(\cdot,\cdot)$ is also bounded and coercive on $Z$,
\begin{equation}\label{eq:2.8}
\begin{aligned}
e(\theta,r)\le \|\nabla \theta\|_0\|\nabla r\|_0,
\end{aligned}
\end{equation}
\begin{equation}\label{eq:2.9}
e(\theta,\theta)\ge \|\nabla \theta\|_0^2.
\end{equation}
(5). The trilinear form $\mathcal{A}_1((\cdot,\cdot),(\cdot,\cdot),(\cdot,\cdot))$ and $\mathcal{H}((\cdot,\cdot),(\cdot,\cdot),(\cdot,\cdot))$ are skew-symmetric with respect to their last two arguments and bounded on $(\bm{X}\times\bm{Y})\times(\bm{X}\times\bm{Y})\times(\bm{X}\times\bm{Y})$  and  $(\bm{X}\times\bm{Y})\times Z\times Z$   respectively,
\begin{equation}\label{eq:2.10}
\mathcal{A}_1((\bm{w},\bm{M}),(\bm{u},\bm{J}),(\bm{u},\bm{J}))=0,
\end{equation}
\begin{equation}\label{eq:2.11}
\begin{aligned}
|\mathcal{A}_1((\bm{w},\bm{M}),(\bm{u},\bm{J}),(\bm{v},\bm{K}))| \le \lambda_2\|(\bm{w},\bm{M})\|_1\|(\bm{u},\bm{v})\|_1\|(\bm{v},\bm{K})\|_1,
\end{aligned}
\end{equation}
\begin{equation}\label{eq:2.12}
\begin{aligned}
\mathcal{H}((\bm{u},\bm{J}),\theta,\theta)=0,
\end{aligned}
\end{equation}
\begin{equation}\label{eq:2.13}
\begin{aligned}
|\mathcal{H}((\bm{u},\bm{J}),\theta,r)|\le \lambda_2\|(\bm{u},\bm{J})\|_1\|\nabla \theta\|_0\|\nabla r\|_0.
\end{aligned}
\end{equation}
\end{lem}
\begin{proof}
The estimates can be easily derived by Holder inequalities,  we will not prove them here.
\end{proof}
\begin{lem}
There exists a constant $\beta^*>0$ only depending on $\Omega$ such that
\begin{equation}\label{eq:2.14}
\mathop{\sup}\limits_{(\bm{0},\bm{0})\not=(\bm{v},\bm{K})\in \bm{X}\times \bm{Y}}\frac{\mathcal{B}((q,\psi),(\bm{v},\bm{K}))}{\|(\bm{v},\bm{K})\|_1} \ge \beta^*\|(q,\psi)\|_0,
\quad \quad \forall (q,\psi ) \in E\times S.
\end{equation}
\end{lem}
\begin{proof}
For the proof of this lemma, please refer to \cite{2021Coupled}.
\end{proof}
Furthermore, we prove the existence and uniqueness of the solution of problem (\ref{eq:2.3}).
\begin{The}
For $\bm{f}\in \bm{H}^{-1}(\Omega), \bm{g}\in\bm{L}^2(\Omega), \varphi \in H^{-1}(\Omega),$ assume that
\begin{equation}\label{eq:2.15}
\sigma=\frac{\lambda_2\|L\|_*}{C_{\min}^2}+\frac{\lambda_2\lambda_q\lambda_{\psi}}{C_{\min}^2}
+\frac{\lambda_2\lambda_q\lambda_{\psi}}{C_{\min}}+\frac{\lambda_2\|L\|_*}{C_{\min}}<1,
\end{equation}
then (\ref{eq:2.3}) is well-posed and  the following estimates hold,
\begin{equation}\label{eq:2.16}
\|(\bm{u},\bm{v})\|_1\le \frac{\|L\|_*+\lambda_q\lambda_{\psi}}{C_{\min}},\quad \|\nabla\theta\|_0\le \lambda_{\psi}.
\end{equation}
\end{The}
\begin{proof}
 The train of thought in this proof is similar to that  in \cite{2021Coupled} but has its characteristics. Here we propose it completely. Firstly, the existence of solutions will be proved by the Banach fixed point theorem.

Given $(\bm{w},\bm{M})\in \Upsilon,$ we consider the following equation:
\begin{equation}\label{eq:2.18}
\begin{aligned}
e(\theta,r)+\mathcal{H}((\bm{w},\bm{M}),\theta,r)& =\Psi(r).
\end{aligned}
\end{equation}
Then, by (\ref{eq:2.8}), (\ref{eq:2.9}), (\ref{eq:2.12}), (\ref{eq:2.13}) and applying Lax-Milgram Theory \cite{Evans1964Partial}, the problem (\ref{eq:2.18}) has unique solution $\theta \in Z$.  With the result that the map $\mathcal{F}$ can be defined by $\mathcal{F}(\bm{w},\bm{M})=\theta,$ where $(\bm{w},\bm{M})\in \bm{X}\times\bm{Y}$ and $\theta \in Z.$ Setting $r=\theta$, by (\ref{eq:2.8}), (\ref{eq:2.9}) and (\ref{eq:2.12}), we have
\begin{equation}\label{eq:2.19}
||\nabla(\mathcal{F}(\bm{w},\bm{M})||_0=||\nabla \theta||_0 \le \lambda_{\psi}.
\end{equation}
Next, we consider the saddle point problem as follow: find $(\bm{u},p,\bm{J},\phi,\theta)\in \bm{X}\times E \times \bm{Y}\times S \times Z$
such that for $\forall (\bm{v},q,\bm{K},\psi,r)\in \bm{X}\times E \times \bm{Y}\times S \times Z,$
\begin{equation}\label{eq:2.20}
\left\{\begin{array}{c}
\begin{aligned}
&\mathcal{A}_0((\bm{u},\bm{J}),(\bm{v},\bm{K}))+\mathcal{A}_1((\bm{w},\bm{M}),(\bm{u},\bm{J}),(\bm{v},\bm{K}))\\
&+\mathcal{B}((p,\phi),(\bm{v},\bm{K})) =\langle \bm{L},(\bm{v},\bm{K})\rangle -\mathcal{Q}(\mathcal{F}(\bm{w},\bm{M}),(\bm{v},\bm{K})), \\
&\mathcal{B}((q,\psi),(\bm{u},\bm{J})) =0.\\
\end{aligned}
\end{array} \right.
\end{equation}
 Applying the saddle theory in \cite{1986Finite}, the problem (\ref{eq:2.20}) has a unique solution $(\bm{u}, p, \bm{J}, \phi)\in \bm{X}\times E \times \bm{Y}\times S.$ Taking $(\bm{v}, q, \bm{K}, \psi)=(\bm{u}, p, \bm{J}, \phi),$ by (\ref{eq:2.10}), (\ref{eq:2.5}), (\ref{eq:2.7}) and (\ref{eq:2.9}) leads to
\begin{equation}\label{eq:2.21}
\begin{aligned}
\|(\bm{u},\bm{J})\|_1\le \frac{\|L\|_*+\lambda_q\lambda_{\psi}}{C_{\min}}.
\end{aligned}
\end{equation}
Inspired by the above discussion, setting $B_R=\{(\bm{v},\bm{K})\in \Upsilon : \|(\bm{v},\bm{K})\|_1 \le C_0 \}, C_0=(\|L\|_*+\lambda_q\lambda_{\psi})/ C_{\min},$ we establish the mapping: $$\mathcal{G}:B_R \to B_R, \ (\bm{w},\bm{M}) \mapsto (\bm{u},\bm{J}).$$
Let $(\bm{w}_j,\bm{M}_j)\in B_R, j=1, 2,$ $\mathcal{G}(\bm{w}_j,\bm{M}_j)=(\bm{u}_j,\bm{J}_j).$ By definition, there exists $(p_j, \phi_j)\in M\times N $ such that $(\bm{u}_j, \bm{J}_j, p_j, \phi_j)$  satisfy the following equations: for $\forall  (\bm{v}, q, \bm{K}, \psi) \in \bm{X}\times E \times \bm{Y}\times S$
\begin{equation}\label{eq:2.22}
\left\{\begin{array}{c}
\begin{aligned}
&\mathcal{A}_0((\bm{u}_j,\bm{J}_j),(\bm{v},\bm{K}))+\mathcal{A}_1((\bm{w}_j,\bm{M}_j),(\bm{u}_j,\bm{J}_j),(\bm{v},\bm{K}))\\
&+\mathcal{B}((p_j,\phi_j),(\bm{v},\bm{K}))=\langle \bm{L},(\bm{v},\bm{K})\rangle -\mathcal{Q}(\mathcal{F}(\bm{w}_j,\bm{M}_j),(\bm{v},\bm{K})), \\
&\mathcal{B}((q,\psi),(\bm{u}_j,\bm{J}_j)) =0.\\
\end{aligned}
\end{array} \right.
\end{equation}
Similarly, through the definition of mapping $\mathcal{F}$, we can also conclude that
\begin{equation}\label{eq:2.23}
\begin{aligned}
e(\mathcal{F}(\bm{w}_j,\bm{M}_j),r)+\mathcal{H}((\bm{w}_j,\bm{M}_j),\mathcal{F}(\bm{w}_j,\bm{M}_j),r)& =\Psi(r).
\end{aligned}
\end{equation}
Subtracting (\ref{eq:2.22}) as $j=2$ from (\ref{eq:2.22}) as $j=1$, set $\bm{v}=\bm{u}_1-\bm{u}_2, \ K=\bm{M}_1-\bm{M}_2,\ q=p_1-p_2,\ \psi=\phi_1-\phi_2 $ and (\ref{eq:2.10}), we arrive at
\begin{equation*}
\begin{aligned}
&\mathcal{A}_0((\bm{u}_1-\bm{u}_2,\bm{J}_1-\bm{J}_2),(\bm{u}_1-\bm{u}_2,\bm{J}_1-\bm{J}_2))
=-\mathcal{A}_1((\bm{w}_1-\bm{w}_2,\bm{M}_1-\bm{M}_2),(\bm{u}_2,\bm{J}_2),(\bm{u}_1-\bm{u}_2,\bm{J}_1-\bm{J}_2))\\
&-\mathcal{Q}(\mathcal{F}(\bm{w}_1,\bm{M}_1)-\mathcal{F}(\bm{w}_2,\bm{M}_2),(\bm{u}_1-\bm{u}_2,\bm{J}_1-\bm{J}_2)).
\end{aligned}
\end{equation*}
Combining (\ref{eq:2.5}),
(\ref{eq:2.7}) and (\ref{eq:2.11}), we derive further
\begin{equation} \label{eq:2.24}
\begin{aligned}
\|(\bm{u}_1-\bm{u}_2,\bm{J}_1-\bm{J}_2)\|_1 \le \frac{\lambda_2\|(\bm{w}_1-\bm{w}_2,\bm{M}_1-\bm{M}_2)\|_1\|(\bm{u}_2,\bm{J}_2)||_1
+\lambda_q\|\nabla(\mathcal{F}(\bm{w}_1,\bm{M}_1)-\mathcal{F}(\bm{w}_2,\bm{M}_2))\|_0}{C_{\min}}.
\end{aligned}
\end{equation}
Analogously subtract (\ref{eq:2.23}) as $j=2$ from (\ref{eq:2.23}) as $j=1,$ take $r=\mathcal{F}(\bm{w}_1,\bm{M}_1)-\mathcal{F}(\bm{w}_2,\bm{M}_2)$ and by (\ref{eq:2.9}), (\ref{eq:2.12}), (\ref{eq:2.13}) and (\ref{eq:2.19}), it is further concluded that:
\begin{equation} \label{eq:2.25}
\begin{aligned}
\|\nabla(\mathcal{F}(\bm{w}_1,\bm{M}_1)-\mathcal{F}(\bm{w}_2,\bm{M}_2))\|_0 \le \lambda_2\lambda_{\psi}\|(\bm{w}_1-\bm{w}_2,\bm{M}_1-\bm{M}_2)\|_1.
\end{aligned}
\end{equation}
Substitute (\ref{eq:2.25}) into (\ref{eq:2.24}) and apply (\ref{eq:2.21}), we can obtain
\begin{equation} \label{eq:2.26}
\begin{aligned}
\|(\bm{u}_1-\bm{u}_2,\bm{J}_1-\bm{J}_2)\|_1  & \le  \frac{\lambda_2\|L\|_*+\lambda_2\lambda_q\lambda_{\psi}
+C_{\min}\lambda_2\lambda_q\lambda_{\psi}}{C_{\min}^2}
\|(\bm{w}_1-\bm{w}_2,\bm{M}_1-\bm{M}_2)\|_1.
\end{aligned}
\end{equation}
Then we have,
\begin{equation*}
\begin{aligned}
&\|\mathcal{G}(\bm{u}_1,\bm{J}_1)-\mathcal{G}(\bm{u}_2,\bm{J}_2)\|_1=\|(\bm{u}_1-\bm{u}_2,\bm{J}_1-\bm{J}_2)\|_1\\
 &\le  \sigma \|(\bm{w}_1-\bm{w}_2,\bm{M}_1-\bm{M}_2)\|_1=\sigma\|(\bm{w}_1,\bm{M}_1)-(\bm{w}_2,\bm{M}_2)\|_1.
\end{aligned}
\end{equation*}
By virtue of (\ref{eq:2.15}), $\mathcal{G}$ is a contraction mapping on $B_R \to B_R.$ Hence applying the Banach fixed point theorem, it can be deduced that $\mathcal{G}$ has a fixed point in $B_R$, which is the solution of equations (\ref{eq:2.20}). Due to the definition of mapping, the problem (\ref{eq:2.3}) exists solution on $\bm{X}\times E \times \bm{Y} \times S \times Z.$

Next, the uniqueness of the solutions will be proved in the following work. Suppose $(\bm{u}_i, p_i, \bm{J}_i, \phi_i, \theta_i)\in\bm{X}\times E \times \bm{Y} \times S \times Z,i=1, 2$ satisfy the equations as follow: for $\forall (\bm{v},q,\bm{K},\psi,r)\in\bm{X}\times E \times \bm{Y} \times S \times Z$
\begin{equation}\label{eq:2.27}
\left\{\begin{array}{c}
\begin{aligned}
&\mathcal{A}_0((\bm{u}_i,\bm{J}_i),(\bm{v},\bm{K}))+\mathcal{A}_1((\bm{u}_i,\bm{J}_i),(\bm{u}_i,\bm{J}_i),(\bm{v},\bm{K}))\\
&+\mathcal{B}((p_i,\phi_i),(\bm{v},\bm{K}))+\mathcal{Q}(\theta_i,(\bm{v},\bm{K})) =\langle \bm{L},(\bm{v},\bm{K})\rangle,  \\
&\mathcal{B}((q,\psi),(\bm{u}_i,\bm{J}_i)) =0,\\
&e(\theta_i,r)+\mathcal{H}((\bm{u}_i,\bm{J}_i),\theta_i,r) =\Psi(r).
\end{aligned}
\end{array} \right.
\end{equation}
Subtract (\ref{eq:2.27}) as $i=2$ from (\ref{eq:2.27}) as $i=1$, we arrive further at
\begin{equation}\label{eq:2.28}
\left\{\begin{array}{c}
\begin{aligned}
&\mathcal{A}_0((\bm{u}_1-\bm{u}_2,\bm{J}_1-\bm{J}_2),(\bm{v},\bm{K}))
+\mathcal{A}_1((\bm{u}_1,\bm{J}_1),(\bm{u}_1,\bm{J}_1),(\bm{v},\bm{K}))\\
&-\mathcal{A}_1((\bm{u}_2,\bm{J}_2),(\bm{u}_2,\bm{J}_2),(\bm{v},\bm{K}))
+\mathcal{B}((p_1-p_2,\phi_1-\phi_2),(\bm{v},\bm{K}))\\
&+\mathcal{Q}(\theta_1-\theta_2,(\bm{v},\bm{K}))=0,\\
&\mathcal{B}((q,\psi),(\bm{u}_1-\bm{u}_2,\bm{J}_1-\bm{J}_2))=0,\\
&e(\theta_1-\theta_2,r)+\mathcal{H}((\bm{u}_1,\bm{J}_1),\theta_1,r)
-\mathcal{H}((\bm{u}_2,\bm{J}_2),\theta_2,r)=0.
\end{aligned}
\end{array} \right.
\end{equation}
Simplify, setting $(\bm{v},\bm{K})=(\bm{u}_1-\bm{u}_2,\bm{J}_1-\bm{J}_2),\ (q,\psi)=(p_1-p_2,\phi_1-\phi_2),\ r=\theta_1-\theta_2$, we can get
\begin{equation}\label{eq:2.29}
\left\{\begin{array}{c}
\begin{aligned}
 &\mathcal{A}_0((\bm{u}_1-\bm{u}_2,\bm{J}_1-\bm{J}_2),(\bm{u}_1-\bm{u}_2,\bm{J}_1-\bm{J}_2))
 =-\mathcal{Q}(\theta_1-\theta_2,(\bm{u}_1-\bm{u}_2,\bm{J}_1-\bm{J}_2))\\
&-\mathcal{A}_1((\bm{u}_1-\bm{u}_2,\bm{J}_1-\bm{J}_2),(\bm{u}_2,\bm{J}_2),(\bm{u}_1-\bm{u}_2,\bm{J}_1-\bm{J}_2)),\\
&e(\theta_1-\theta_2,\theta_1-\theta_2)=-\mathcal{H}((\bm{u}_1-\bm{u}_2,\bm{J}_1-\bm{J}_2),\theta_2,\theta_1-\theta_2).
\end{aligned}
\end{array} \right.
\end{equation}
 Combining (\ref{eq:2.5}), (\ref{eq:2.7}), (\ref{eq:2.9}), (\ref{eq:2.11}) and (\ref{eq:2.13}), we can get the following two inequalities easily that
\begin{equation} \label{eq:2.30}
\begin{aligned}
\|\nabla(\theta_1-\theta_2))\|_0 \le \lambda_2\lambda_{\psi}\|(\bm{u}_1-\bm{u}_2,\bm{J}_1-\bm{J}_2)\|_1,
\end{aligned}
\end{equation}
and
\begin{equation} \label{eq:2.31}
\begin{aligned}
\|(\bm{u}_1-\bm{u}_2,\bm{J}_1-\bm{J}_2)\|_1^2  & \le  \sigma||(\bm{u}_1-\bm{u}_2,\bm{J}_1-\bm{J}_2)||_1^2.
\end{aligned}
\end{equation}
Because of (\ref{eq:2.15}) and (\ref{eq:2.31}), we directly see
\begin{equation} \label{eq:2.32}
\begin{aligned}
\bm{u}_1=\bm{u}_2,\quad\bm{J}_1=\bm{J}_2,\quad\theta_1=\theta_2.
\end{aligned}
\end{equation}
Substituting (\ref{eq:2.32}) into (\ref{eq:2.28}), and combining the inf-sup condition (\ref{eq:2.14}) lead to $p_1=p_2,\ \phi_1=\phi_2.$

This completes the proof.
\end{proof}
\section{Finite element approximation}
The finite element approximation of  the stationary thermally coupled inductionless MHD, in this section, is considered.
 Set $\mathcal{T}_h$ is a shape-regular and quasi-uniform triangular for $d = 2$ or tetrahedral for $d = 3$ mesh of $\Omega$. We introduce generally the local mesh size $h_e= diam (e)$ and the global mesh size $h:=\mathop{\max}\limits_{e\in \mathcal{T}_h}h_e.$
$P_k(I)$ be the space of polynomials of degree $k$ on element $e$ and define $\bm{P}_k(e)=P_k(e)^d$
where $k$ is any integer and $k\ge 0.$

Here we choose conforming finite element spaces $\bm{X}_h\subseteq \bm{X},E_h\subseteq E,\bm{Y}_h\subseteq \bm{Y},S_h\subseteq S,Z_h\subseteq Z$ to discrete velocity $\bm{u}$, pressure $p$, current density $\bm{J}$, electric potential $\phi$ and temperature $\theta$ , respectively. Moreover, we need the assumptions as follow.
\begin{ass}(\textbf{inf-sup condition})
There are respectively the following conclusions on $\bm{X}_h\times E_h$ and $\bm{Y}_h\times S_h,$
\begin{equation}\label{eq:3.1}
\mathop{\inf}\limits_{0\not=p_h\in E_h }\mathop{\sup}\limits_{\bm{0}\not=\bm{w}_h\in \bm{X}_h }\frac{b_s((p_h,\bm{w}_h))}{\|(\nabla\bm{w}_h,p_h)\|_1} \ge \beta_{s,h}, \quad
\mathop{\inf}\limits_{0\not=\phi_h\in S_h }\mathop{\sup}\limits_{\bm{0}\not=\bm{J}_h\in \bm{Y}_h }\frac{b_m((\phi_h,\bm{J}_h))}{\|(\nabla\bm{J}_h,\phi_h)\|_1} \ge \beta_{m,h}.
\end{equation}
where $\beta_{s,h}$ and $\beta_{m,h}$ only depend on $\Omega.$
\end{ass}
\begin{ass}\textbf{(Approximation property) }
There exist two integers $k$ and $l$ such that the following standard approximation properties hold for $\gamma$ and $\tau,$
\begin{equation*}
\begin{aligned}
&\mathop{\inf}\limits_{\bm{w}_h\in \bm{X}_h}\|\bm{u}-\bm{w}_h\|_{1,\Omega}\le Ch^{\min\{\gamma,l\}}
\|\bm{u}\|_{1+\gamma,\Omega},\quad
\mathop{\inf}\limits_{q_h\in E_h}\|p-q_h\|_{1,\Omega}\le Ch^{\min\{\gamma,l\}}
\|p\|_{\gamma,\Omega},\\
&\mathop{\inf}\limits_{r_h\in Z_h}\|\theta-r_h\|_{1,\Omega}\le Ch^{\min\{\gamma,l\}}
\|\theta\|_{1+\gamma,\Omega},\quad
\mathop{\inf}\limits_{\psi_h\in S_h}\|\phi-\psi_h\|_{1,\Omega}\le Ch^{\min\{\tau,k\}}
\|\phi\|_{\tau,\Omega},\\
&\mathop{\inf}\limits_{\bm{K}_h\in \bm{Y}_h}\|\bm{J}-\bm{K}_h\|_{1,\Omega}\le Ch^{\min\{\tau,k\}}
(\|\bm{J}\|_{\tau,\Omega}+\|\mbox{div}\bm{J}\|_{\tau,\Omega}).
\end{aligned}
\end{equation*}
\end{ass}
There are several stable finite element pairs which satisfy these assumptions \cite{1991Mixed,1986Finite}.
With the discrete spaces, the finite element approximation of problem (\ref{eq:2.3})
reads as follows: find $(\bm{u}_h,\bm{J}_h,p_h,\phi_h,\theta_h)\in \bm{X}_h\times\bm{Y}_h\times E_h\times S_h \times Z_h$ such that $\forall (\bm{v}_h,\bm{K}_h,p_h,\psi_h,r_h)\in \bm{X}_h\times\bm{Y}_h\times E_h\times S_h \times Z_h,$
\begin{equation}\label{eq:3.2}
\left\{\begin{array}{c}
\begin{aligned}
&\mathcal{A}_0((\bm{u}_h,\bm{J}_h),(\bm{v}_h,\bm{K}_h))+\mathcal{A}_1((\bm{u}_h,\bm{J}_h),(\bm{u}_h,\bm{J}_h),(\bm{v}_h,\bm{K}_h))\\
& +\mathcal{B}((p_h,\psi_h),(\bm{v}_h,\bm{K}_h))+\mathcal{Q}(\theta_h,(\bm{v}_h,\bm{K}_h)) =\langle \bm{L},(\bm{v}_h,\bm{K}_h)\rangle, \\
&\mathcal{B}((q_h,\psi_h),(\bm{u}_h,\bm{J}_h)) =0,\\
&e(\theta_h,r_h)+\mathcal{H}((\bm{u}_h,\bm{J}_h),\theta_h,r_h) =\Psi(r_h).
\end{aligned}
\end{array} \right.
\end{equation}
\begin{prop}
 Suppose $(\bm{u}_h,\bm{J}_h,p_h,\phi_h,\theta_h)$ is the solution of (\ref{eq:3.2}), with the result that the scheme is charge-conservative, namely, $\nabla \cdot \bm{J}=0.$
 \end{prop}
\begin{proof}
It note that $\forall \psi \in S_h, (\psi_h,\nabla \cdot \bm{J}_h)=0$ and $\nabla\cdot J_h\in S_h.$  Set $\psi_h=\nabla\cdot J_h,$ we further have $\nabla\cdot J_h=0.$
\end{proof}
On the basis of continuous space, we define discrete kernel space
\begin{equation*}
\begin{aligned}
\Upsilon_h=\{ (\bm{v}_h,\bm{J}_h)\in \bm{X}_h\times\bm{Y}_h:\mathcal{B}((q_h,\psi_h),(\bm{v}_h,\bm{J}_h))=0,\forall (q,\psi)\in E_h\times S_h\}=\bm{V}_h\times\bm{U}_h,
\end{aligned}
\end{equation*}
where
\begin{equation*}
\begin{aligned}
&\bm{V}_h=\{\bm{v}_h\in \bm{X}_h : b_s(q_h,\bm{v}_h)=0,\forall q_h\in E_h\},\\
&\bm{U}_h=\{\bm{K}_h\in \bm{Y}_h : b_m(\psi_h,\bm{K}_h)=0,\forall \psi_h \in S_h\}=\{\bm{K}_h\in \bm{X}:\mbox{div}\bm{K}_h=0\}.
\end{aligned}
\end{equation*}
What deserves our special attention is: $\bm{V}_h \not \subset \bm{V}$ but $\bm{U}_h  \subset \bm{U}.$
\begin{lem}
 The following estimates for bilinear and trilinear terms hold:\\
$(1).$ The bilinear form $\mathcal{A}_0((\cdot,\cdot),(\cdot,\cdot))$ is bounded on $(\bm{X}_h\times\bm{Y}_h)\times(\bm{X}_h\times\bm{Y}_h)$ and coercive on $(\bm{X}_h\times\bm{U}_h)\times(\bm{X}_h\times\bm{U}_h),$ namely,
\begin{equation}\label{eq:3.3}
\begin{aligned}
|\mathcal{A}_0((\bm{u}_h,\bm{J}_h),(\bm{v}_h,\bm{K}_h))|\le C_{\max}\|(\bm{u}_h,\bm{J}_h)\|_1\|(\bm{v}_h,\bm{K}_h)\|_1,
\end{aligned}
\end{equation}
\begin{equation}\label{eq:3.4}
\begin{aligned}
\mathcal{A}_0((\bm{u}_h,\bm{J}_h),(\bm{u}_h,\bm{J}_h))\ge C_{\min}\|(\bm{u}_h,\bm{J}_h)\|_1^2,
\end{aligned}
\end{equation}
where $C_{\max}=2max\{Pr,\kappa,\kappa \lambda_1||\bm{B}||_{\bm{L}^3(\Omega)}\},$ $C_{\min}=min\{Pr,\kappa\}.$\\
(2). The continuous property of the bilinear form $\mathcal{B}((\cdot,\cdot),(\cdot,\cdot)),$ namely, for $\forall (\bm{u}_h,q_h,\bm{J}_h,\psi_h)\in (\bm{X}_h\times E_h)\times(\bm{Y}_h\times S_h),$
\begin{equation}\label{eq:3.5}
\begin{aligned}
|\mathcal{B}((q_h,\psi_h),(\bm{v}_h,\bm{K}_h))|\le \max\{Pr,\kappa\}\|(\bm{v},\bm{K})\|_1\|(q,\psi)\|_0,
\end{aligned}
\end{equation}
(3). The continuous property of the bilinear form $\mathcal{Q}(\cdot,(\cdot,\cdot)),$ for  $(\forall \theta,\bm{v}_h,\bm{K}_h)\in Z_h\times \bm{X}_h\times\bm{Y}_h,$
\begin{equation}\label{eq:3.6}
\begin{aligned}
|\mathcal{Q}(\theta_h,(\bm{v}_h,\bm{K}_h))| \le \lambda_q\|(\bm{v}_h,\bm{K}_h)\|_1\|\nabla \theta_h\|_0,
\end{aligned}
\end{equation}
where $\lambda_q=PrRaC^2_\Omega.$ \\
(4). The continuous property of the bilinear form $e(\cdot,\cdot)$ is also bounded and coercive on $Z$,
\begin{equation}\label{eq:3.7}
\begin{aligned}
e(\theta_h,r_h)\le \|\nabla \theta_h\|_0\|\nabla r_h\|_0,
\end{aligned}
\end{equation}
\begin{equation}\label{eq:3.8}
e(\theta_h,\theta_h)\ge \|\nabla \theta_h\|_0^2.
\end{equation}
(5). The trilinear form $\mathcal{A}_1((\cdot,\cdot),(\cdot,\cdot),(\cdot,\cdot))$ and $\mathcal{H}((\cdot,\cdot),(\cdot,\cdot),(\cdot,\cdot))$ are skew-symmetric with respect to their last two arguments and bounded on $(\bm{X}_h\times\bm{Y}_h)\times(\bm{X}_h\times\bm{Y}_h)\times(\bm{X}_h\times\bm{Y}_h)$ and  $(\bm{X}_h\times\bm{Y}_h)\times Z_h\times Z_h$   respectively,
\begin{equation}\label{eq:3.9}
\mathcal{A}_1((\bm{w}_h,\bm{M}_h),(\bm{u}_h,\bm{J}_h),(\bm{u}_h,\bm{J}_h))=0,
\end{equation}
\begin{equation}\label{eq:3.10}
\begin{aligned}
|\mathcal{A}_1((\bm{w}_h,\bm{M}_h),(\bm{u}_h,\bm{J}_h),(\bm{v}_h,\bm{K}_h))| \le \lambda_2\|(\bm{w}_h,\bm{M}_h)\|_1\|(\bm{u}_h,\bm{v}_h)\|_1\|(\bm{v}_h,\bm{K}_h)\|_1,
\end{aligned}
\end{equation}
\begin{equation}\label{eq:3.11}
\begin{aligned}
\mathcal{H}((\bm{u}_h,\bm{J}_h),\theta_h,\theta_h)=0,
\end{aligned}
\end{equation}
\begin{equation}\label{eq:3.12}
\begin{aligned}
|\mathcal{H}((\bm{u}_h,\bm{J}_h),\theta_h,r_h)|\le \lambda_2\|(\bm{u}_h,\bm{J}_h)\|_1\|\nabla \theta_h\|_0\|\nabla r_h\|_0.
\end{aligned}
\end{equation}
\end{lem}
\begin{lem}
There exists a constant $\beta^*>0$ only depending on $\Omega$ such that
\begin{equation}\label{eq:3.13}
\mathop{\sup}\limits_{(\bm{0},\bm{0})\not=(\bm{v}_h,\bm{K}_h)\in \bm{X}_h\times \bm{Y}_h}\frac{\mathcal{B}((q_h,\psi_h),(\bm{v}_h,\bm{K}_h))}{\|(\bm{v}_h,\bm{K}_h)\|_1} \ge \beta^*\|(q_h,\psi_h)\|_0,
\quad \quad \forall (q_h,\psi_h) \in E_h\times S_h.
\end{equation}
\end{lem}
\begin{The}
For $\bm{f}\in \bm{H}^{-1}(\Omega),\bm{g}\in\bm{L}^2(\Omega),\varphi \in H^{-1}(\Omega),$ assume
\begin{equation}\label{eq:3.14}
\sigma=\frac{\lambda_2\|L\|_*}{C_{min}^2}+\frac{\lambda_2\lambda_q\lambda_{\psi}}{C_{min}^2}
+\frac{\lambda_2\lambda_q\lambda_{\psi}}{C_{min}}+\frac{\lambda_2\|L\|_*}{C_{min}}<1,
\end{equation}
then problem (\ref{eq:3.2}) is well-posed and the estimates hold as follow,
\begin{equation}\label{eq:3.15}
\|(\bm{u}_h,\bm{v}_h)\|_1\le \frac{\|L\|_*+\lambda_q\lambda_{\psi}}{ C_{min}}, \quad \|\nabla\theta_h\|_0\le \lambda_{\psi}.
\end{equation}
\end{The}
\begin{The}
 Assume the conditions of Theorems 2.3, 3.3 hold, the solutions of continuous problem $(\ref{eq:2.3})$ and discrete problem (\ref{eq:3.2}) are respectively $(\bm{u},\bm{J},p,\phi,\theta)$ and $(\bm{u}_h,\bm{J}_h,p_h,\phi_h,\theta_h).$ Suppose that the exact solution $(\bm{u},\bm{J},p,\phi,\theta)$ satisfy
\begin{equation}\label{eq:3.17}
\bm{u}\in \bm{H}^{\gamma+1}(\Omega),\quad p\in H^{\gamma}(\Omega),\quad \bm{J}\in \bm{H}^{\tau}(\Omega),\quad
\phi\in H^{\tau}(\Omega),\quad \theta\in H^{\gamma+1}(\Omega),
\end{equation}
for a regularity exponent $\gamma,\tau >1/2,$ then we can get the following estimates:
\begin{equation*}
\begin{aligned}
&\|(\bm{u}-\bm{u}_h,\bm{J}-\bm{J}_h)\|_1 +\|(\theta-\theta_h)\|_1+\|p-p_h\|_0\\
&\le C(\mathop{\inf}\limits_{(\bm{v}_h,\bm{K}_h)\in (\bm{X}_h\times \bm{Y}_h)}\|(\bm{u}-\bm{v}_h,\bm{J}-\bm{K}_h)\|_1+\mathop{\inf}\limits_{q_h\in E_h}\|p-q_h\|_0+\mathop{\inf}\limits_{s_h\in Z_h}\|\theta-s_h\|_1)  \\
&\le Ch^{\min(\gamma,l,k,\tau)}(\|\bm{u}\|_{1+\gamma,\Omega}+\|\bm{J}_{\tau,\Omega}\|+\|p\|_{\gamma,\Omega}
+\|\theta\|_{1+\gamma,\Omega}),\\
&\|\phi-\phi_h\|_0 \le  C(\mathop{\inf}\limits_{(\bm{v}_h,\bm{K}_h)\in (\bm{X}_h\times \bm{Y}_h)}\|(\bm{u}-\bm{v}_h,\bm{J}-\bm{K}_h)\|_1+\mathop{\inf}\limits_{q_h\in E_h}\|p-q_h\|_0+\mathop{\inf}\limits_{\psi_h\in S_h}\|\phi-\psi_h\|_0)\\
&\le Ch^{\min(\gamma,l,k,\tau)}(\|\bm{u}\|_{1+\gamma,\Omega}+\|\bm{J}_{\tau,\Omega}\|+\|p\|_{\gamma,\Omega}
+\|\phi\|_{\tau,\Omega}).\\
\end{aligned}
\end{equation*}
\end{The}
\begin{proof}
We begin to handle the error in $\theta$.
Setting $(\bm{v},q,\bm{K},\psi,\theta)=(\bm{v}_h,q_h,\bm{K}_h,\psi_h,\theta_h)$, subtracting (\ref{eq:3.2}) from (\ref{eq:2.3}), we have
\begin{equation}\label{eq:3.18}
\left\{\begin{array}{c}
\begin{aligned}
&\mathcal{A}_0((\bm{u}-\bm{u}_h,\bm{J}-\bm{J}_h),(\bm{v}_h,\bm{K}_h))
+\mathcal{A}_1((\bm{u}-\bm{u}_h,\bm{J}-\bm{J}_h),(\bm{u},\bm{J}),(\bm{v}_h,\bm{K}_h))\\
&+\mathcal{A}_1((\bm{u}_h,\bm{J}_h),(\bm{u}-\bm{u}_h,\bm{J}-\bm{J}_h),(\bm{v}_h,\bm{K}_h))
+\mathcal{B}((p-p_h,\phi-\phi_h),(\bm{v}_h,\bm{K}_h))\\
&+\mathcal{Q}(\theta-\theta_h,(\bm{v}_h,\bm{K}_h))=0,\\
&\mathcal{B}((q_h,\psi_h),(\bm{u}-\bm{u}_h,\bm{J}-\bm{J}_h))=0,\\
&e(\theta-\theta_h,r_h)+\mathcal{H}((\bm{u}-\bm{u}_h,\bm{J}-\bm{J}_h),\theta,r_h)
+\mathcal{H}((\bm{u}_h,\bm{J}_h),\theta-\theta_h,r_h)=0.
\end{aligned}
\end{array} \right.
\end{equation}
For $(\bm{w}_h,\bm{M}_h)\in \Upsilon_h,(q_h,\psi_h)\in E_h\times S_h,s_h \in Z_h,$ setting
$$\bm{u}-\bm{u}_h=(\bm{u}-\bm{w}_h)+(\bm{w}_h-\bm{u}_h),\quad p-p_h=(p-q_h)+(q_h-p_h),\quad \theta-\theta_h=(\theta-s_h)+(s_h-\theta_h),$$
$$\bm{J}-\bm{J}_h=(\bm{J}-\bm{M}_h)+(\bm{M}_h-\bm{J}_h),\quad \phi-\phi_h=(\phi-\psi_h)+(\psi_h-\phi_h).$$
Substituting the above formula into (\ref{eq:3.18}) and taking $r_h=\theta_h-s_h$, we can further obtain
\begin{equation}\label{eq:3.19}
\begin{aligned}
&e(\theta_h-s_h,\theta_h-s_h)+\mathcal{H}((\bm{u}_h,\bm{J}_h),\theta_h-s_h,\theta_h-s_h)
=\mathcal{H}((\bm{u}-\bm{w}_h,\bm{J}-\bm{M}_h),\theta,\theta_h-s_h)\\
&+e(\theta-s_h,\theta_h-s_h)+\mathcal{H}((\bm{u}_h,\bm{J}_h),\theta-s_h,\theta_h-s_h)-\mathcal{H}((\bm{u}_h-\bm{w}_h,\bm{J}_h-\bm{M}_h),\theta,\theta_h-s_h).
\end{aligned}
\end{equation}
By applying (\ref{eq:3.12}), (\ref{eq:2.8}), (\ref{eq:2.13}), (\ref{eq:2.11}), the second inequality of (\ref{eq:2.16}) and (\ref{eq:3.15}), we arrive at
\begin{equation}\label{eq:3.20}
\begin{aligned}
\|\nabla(\theta-\theta_h)\|_0 & \le \frac{2C_{min}+\lambda_2\mu\|L\|_*+\lambda_2\lambda_q\lambda_{\psi}}{C_{min}}\|\nabla(\theta-s_h)\|_0 +\lambda_2\lambda_{\psi}\|(\bm{u}-\bm{w}_h,\bm{J}-\bm{M}_h)\|_1 \\
 &+\lambda_2\lambda_{\psi}\|(\bm{u}_h-\bm{w}_h,\bm{J}_h-\bm{M}_h)\|_1.
\end{aligned}
\end{equation}
Since (\ref{eq:3.18}), the following equation holds:
\begin{equation}\label{eq:3.21}
\begin{aligned}
&\mathcal{A}_0((\bm{u}_h-\bm{w}_h,\bm{J}_h-\bm{M}_h),(\bm{v}_h,\bm{K}_h))
+\mathcal{A}_1((\bm{u}_h-\bm{w}_h,\bm{J}_h-\bm{M}_h),(\bm{u},\bm{J}),(\bm{v}_h,\bm{K}_h))\\
&+\mathcal{A}_1((\bm{u}_h,\bm{J}_h),(\bm{u}_h-\bm{w}_h,\bm{J}_h-\bm{M}_h),(\bm{v}_h,\bm{K}_h))
+\mathcal{B}((p_h-q_h,\phi_h-\psi_h),(\bm{v}_h,\bm{K}_h))\\
&=\mathcal{A}_0((\bm{u}-\bm{w}_h,\bm{J}-\bm{M}_h),(\bm{v}_h,\bm{K}_h))
+\mathcal{B}((p-q_h,\phi-\psi_h),(\bm{v}_h,\bm{K}_h))-\mathcal{Q}(\theta-\theta_h,(\bm{v}_h,\bm{K}_h))\\
&+\mathcal{A}_1((\bm{u}_h,\bm{J}_h),(\bm{u}-\bm{w}_h,\bm{J}-\bm{M}_h),(\bm{v}_h,\bm{K}_h))
+\mathcal{A}_1((\bm{u}-\bm{w}_h,\bm{J}-\bm{M}_h),(\bm{u},\bm{J}),(\bm{v}_h,\bm{K}_h)).\\
\end{aligned}
\end{equation}
Choosing $(\bm{v}_h,\bm{K}_h)=(\bm{u}_h-\bm{w}_h,\bm{J}_h-\bm{M}_h) \in \Upsilon_h,$ and combining (\ref{eq:3.9}), (\ref{eq:3.4}), (\ref{eq:2.11}) and (\ref{eq:2.16}), we obtain
\begin{equation}\label{eq:3.22}
\begin{aligned}
L.H.S
& \ge (C_{min}-\frac{\lambda_2\|L\|_*+\lambda_2\lambda_q\lambda_{\psi}}{ C_{min}})\|(\bm{u}_h-\bm{w}_h,\bm{J}_h-\bm{M}_h)\|_1^2.
\end{aligned}
\end{equation}
Analogously by (\ref{eq:2.4}), (\ref{eq:2.6}), (\ref{eq:2.11}) and (\ref{eq:2.7}), we can further obtain
\begin{equation}\label{eq:3.23}
\begin{aligned}
R.H.S \le\|(\bm{u}_h-\bm{w}_h,\bm{J}_h-\bm{M}_h)\|_1(C(\|(\bm{u}-\bm{w}_h,\bm{J}-\bm{M}_h)\|_1
+\|(p-q_h)\|_0)+\lambda_q\|\nabla(\theta-\theta_h)\|_0),
\end{aligned}
\end{equation}
which combines (\ref{eq:3.22}) to lead to
\begin{equation}\label{eq:3.24}
\begin{aligned}
&(C_{min}-\frac{\lambda_2\|L\|_*
+\lambda_2\lambda_q\lambda_{\psi}}{ C_{min}})\|(\bm{u}_h-\bm{w}_h,\bm{J}_h-\bm{M}_h)\|_1\\
&\le C(\|(\bm{u}-\bm{w}_h,\bm{J}-\bm{M}_h)\|_1+\|(p-q_h)\|_0)+\lambda_q\|\nabla(\theta-\theta_h)\|_0.
\end{aligned}
\end{equation}
Substituting (\ref{eq:3.20}) into (\ref{eq:3.24}), the inequality holds as follow:
\begin{equation*}
\begin{aligned}
&C_{min}(1-\sigma)\|(\bm{u}_h-\bm{w}_h,\bm{J}_h-\bm{M}_h)\|_1\\
&\le C(\|(\bm{u}-\bm{w}_h,\bm{J}-\bm{M}_h)\|_1+\|(p-q_h)\|_0+\|\nabla(\theta-s_h)\|_0 ).\\
\end{aligned}
\end{equation*}
On account of (\ref{eq:3.14}), we can easily conclude that $1-\sigma \ge 0$.
Consequently, we have
\begin{equation*}
\begin{aligned}
\|(\bm{u}_h-\bm{w}_h,\bm{J}_h-\bm{M}_h)\|_1 \le C(\|(\bm{u}-\bm{w}_h,\bm{J}-\bm{M}_h)\|_1+\|(p-q_h)\|_0+\|\nabla(\theta-s_h)\|_0 ).\\
\end{aligned}
\end{equation*}
Therefore, by the triangle inequality and taking infimum, we get
$$\|(\bm{u}-\bm{u}_h,\bm{J}-\bm{J}_h)\|_1
\le C(\mathop{inf}\limits_{(\bm{w}_h,\bm{M}_h)\in \Upsilon_h}\|(\bm{u}-\bm{w}_h,\bm{J}-\bm{M}_h)\|_1+\mathop{inf}\limits_{q_h\in E_h}\|p-q_h\|_0+\mathop{inf}\limits_{s_h\in Z_h}\|\theta-s_h\|_1),  $$
\begin{equation*}
\begin{aligned}
\|\nabla(\theta-\theta_h)\|_0 & \le C(\|\nabla(\theta-s_h)\|_0 +\|(\bm{u}-\bm{w}_h,\bm{J}-\bm{M}_h)\|_1 +||p-q_h||_0).
\end{aligned}
\end{equation*}
Applying the inf-sup condition (\ref{eq:2.14}), the discrete inf-sup condition (\ref{eq:3.13}) and Lemma1.1 in Part II of \cite{1986Finite}, we get
\begin{equation*}
\begin{aligned}
\mathop{inf}\limits_{(\bm{w}_h,\bm{M}_h)\in \Upsilon_h}\|(\bm{u}-\bm{w}_h,\bm{J}-\bm{M}_h)\|_1 \le C^* \mathop{inf}\limits_{(\bm{v}_h,\bm{K}_h)\in (\bm{X}_h\times \bm{Y}_h)}||(\bm{u}-\bm{v}_h,\bm{J}-\bm{K}_h)||_1,
\end{aligned}
\end{equation*}
where $C^*=1+max\{Pr,\kappa\}/\beta^*.$
Therefore, we can further obtain the following  estimates
$$\|(\bm{u}-\bm{u}_h,\bm{J}-\bm{J}_h)\|_1
\le C(\mathop{inf}\limits_{(\bm{v}_h,\bm{K}_h)\in (\bm{X}_h\times \bm{Y}_h)}\|(\bm{u}-\bm{v}_h,\bm{J}-\bm{K}_h)\|_1+\mathop{inf}\limits_{q_h\in E_h}\|p-q_h\|_0+\mathop{inf}\limits_{s_h\in Z_h}\|\theta-s_h\|_1),  $$
$$\|\theta-\theta_h)\|_1\le C(\mathop{inf}\limits_{(\bm{v}_h,\bm{K}_h)\in (\bm{X}_h\times \bm{Y}_h)}\|(\bm{u}-\bm{v}_h,\bm{J}-\bm{K}_h)\|_1+\mathop{inf}\limits_{q_h\in E_h}\|p-q_h\|_0+\mathop{inf}\limits_{s_h\in Z_h}\|\theta-s_h\|_1).  $$
Next, let's continue to arrive at the error estimate of $(p,\phi)$.
For all $(q_h,\psi_h)\in E_h\times S_h, (\bm{v}_h,\bm{K}_h)\in \bm{X}_h\times \bm{Y}_h,$ we have
\begin{equation*}%\label{eq:3.26}
\begin{aligned}
\mathcal{B}((p_h-q_h,\phi_h-\psi_h),(\bm{v}_h,\bm{K}_h)=\mathcal{A}_0((\bm{u}-\bm{u}_h,\bm{J}-\bm{J}_h),(\bm{v}_h,\bm{K}_h))
+\mathcal{A}_1((\bm{u}-\bm{u}_h,\bm{J}-\bm{J}_h),(\bm{u},\bm{J}),(\bm{v}_h,\bm{K}_h))\\
+\mathcal{A}_1((\bm{u}_h,\bm{J}_h),(\bm{u}-\bm{u}_h,\bm{J}-\bm{J}_h),(\bm{v}_h,\bm{K}_h))
+\mathcal{B}((p-q_h,\phi-\psi_h),(\bm{v}_h,\bm{K}_h))
+\mathcal{Q}(\theta-\theta_h,(\bm{v}_h,\bm{K}_h)).
\end{aligned}
\end{equation*}
To get more detailed results, we rewrite the above compact form to the original form, as follows:
\begin{equation*}
\begin{aligned}
b_s(p_h-q_h,\bm{v}_h)&=a_s((\bm{u}-\bm{u}_h),\bm{v}_h)+d(\bm{J}-\bm{J}_h,\bm{v}_h)
+c_0((\bm{u}-\bm{u}_h),\bm{u},\bm{v}_h)\\
&+c_0(\bm{u}_h,(\bm{u}-\bm{u}_h),\bm{v}_h)+b_s(p-q_h,\bm{v}_h)+Q(\theta-\theta_h,(\bm{v}_h,\bm{K}_h))
\end{aligned}
\end{equation*}
and
\begin{equation*}
\begin{aligned}
b_m(\phi_h-\psi_h,\bm{K}_h)&=a_m((\bm{J}-\bm{J}_h),\bm{K}_h)-d(\bm{K}_h,\bm{u}-\bm{u}_h)
+b_m(\phi-\psi_h,\bm{K}_h).
\end{aligned}
\end{equation*}
Therefore, using (\ref{eq:2.2}), (\ref{eq:2.6}), (\ref{eq:2.7}), we deduce further that
\begin{equation*}
\begin{aligned}
b_s(p_h-q_h,\bm{v}_h) &\le \|\nabla (\bm{u}-\bm{u}_h)\|_0\|\nabla \bm{v}_h\|_0+\lambda_1 \kappa \|\bm{B}\|_{L^3}\| (\bm{J}-\bm{J}_h)\|_{div}\|\nabla \bm{v}_h\|_0
+\lambda_2\|\nabla (\bm{u}-\bm{u}_h)\|_0\|\nabla \bm{u}\|_0\|\nabla \bm{v}_h\|_0\\
&+\lambda_2\|\nabla (\bm{u}-\bm{u}_h)\|_0\|\nabla \bm{u}_h\|_0\|\nabla \bm{v}_h\|_0
+Pr\|(p-q_h)\|_0\|\nabla \bm{v}\|_0+\lambda_q\|\nabla (\theta-\theta_h)\|_0\|\nabla \bm{v}_h\|_0
\end{aligned}
\end{equation*}
and
\begin{equation*}%\label{eq:3.28}
\begin{aligned}
b_m(\phi_h-\psi_h,\bm{K}_h) &\le \kappa \| (\bm{J}-\bm{J}_h)\|_{div}\|\bm{K}_h\|_{div}+\lambda_1 \kappa \|\bm{B}\|_{L^3}\|\nabla(\bm{u}-\bm{u}_h)\|_0\| \bm{K}_h\|_{div}
+\kappa\|(\phi-\psi_h)\|_0\|\bm{K}_h\|_{div}\\
&\le C(\|(\bm{u}-\bm{u}_h,\bm{J}-\bm{J}_h)\|_1+\|\phi-\psi_h\|_0)\|\bm{K}_h\|_{div}.
\end{aligned}
\end{equation*}
With the help of the discrete inf-sup condition (\ref{eq:3.13}), the triangle inequality and taking infimum over $E_h$ and $S_h$, respectively, we get
\begin{equation*}
\begin{aligned}
\|p-p_h\|_0 &\le C(\|(\bm{u}-\bm{u}_h,\bm{J}-\bm{J}_h)\|_1+\|\nabla (\theta-\theta_h)\|_0+\|p-q_h\|_0)\\
&\le C(\mathop{inf}\limits_{(\bm{v}_h,\bm{K}_h)\in (\bm{X}_h\times \bm{Y}_h)}\|(\bm{u}-\bm{v}_h,\bm{J}-\bm{K}_h)||_1+\mathop{inf}\limits_{q_h\in E_h}\|p-q_h\|_0+\mathop{inf}\limits_{s_h\in Z_h}\|\theta-s_h\|_1),
\end{aligned}
\end{equation*}
and
\begin{equation*}
\begin{aligned}
\|\phi-\psi_h\|_0 &\le C(\mathop{inf}\limits_{\psi_h\in S_h}\|(\phi-\psi_h)\|_0+\|(\bm{u}-\bm{u}_h,\bm{J}-\bm{J}_h)\|_1)\\
&\le C(\mathop{inf}\limits_{\psi_h\in S_h}\|(\phi-\psi_h)\|_0+\mathop{inf}\limits_{(\bm{v}_h,\bm{K}_h)\in (\bm{X}_h\times \bm{Y}_h)}\|(\bm{u}-\bm{v}_h,\bm{J}-\bm{K}_h)\|_1+\mathop{inf}\limits_{q_h\in E_h}\|p-q_h\|_0).
\end{aligned}
\end{equation*}
This completes the proof.
\end{proof}
\section{Iterative methods}
In this section, we introduce three classical iterative methods, namely Stokes, Newton and Oseen iterative methods, for hybrid finite element problem (\ref{eq:3.2}). For simplicity, we consider
\begin{equation*}
\begin{aligned}
\bm{D} :=\bm{X}\times\bm{Y},\quad G :=M\times N,
\end{aligned}
\end{equation*}
and for any $(\bm{u},\bm{J}),(\bm{v},\bm{K})\in \bm{D},$ and $(p,\phi), (q,\psi)\in G,$
\begin{equation*}
\begin{aligned}
\bm{\eta}:=(\bm{u},\bm{J}),\quad \xi :=(p,\phi), \quad \bm{\Phi} :=(\bm{v},\bm{K}),\quad \chi :=(q,\psi),\quad
z_{\zeta}^n=\zeta_h-\zeta_h^n,\quad\zeta=\bm{\eta}, \xi, \theta.
\end{aligned}
\end{equation*}
The initial value $\bm{\eta}_h^0\in \Upsilon_h,\theta^0\in Z_h$  of three iterations is obtained by the following equations:
\begin{equation}\label{eq:4.1}
\left\{\begin{array}{c}
\begin{aligned}
&\mathcal{A}_0(\bm{\eta}_h^0,\bm{\Phi}_h)+\mathcal{B}(\xi_h^0,\bm{\Phi}_h)+\mathcal{Q}(\theta_h^0,\bm{\Phi}_h)=\langle L,\bm{\Phi}_h \rangle,\\
&\mathcal{B}(\chi_h,\bm{\eta}_h^0)=0,\\
&e(\theta_h^0,r_h)=\Psi(r_h),
\end{aligned}
\end{array} \right.
\end{equation}
for $\forall (\bm{\Phi}_h,\chi_h,r_h)\in \bm{D}_h\times G_h\times Z_h.$\\
\textbf{Method I (Stokes Iteration)}\\
Given $\bm{\eta}_h^{n-1} \in \Upsilon_h$ and $\theta_h^{n-1}\in Z_h$, find $(\bm{\eta}_h^n,\xi_h^n,\theta_h^n) \in \bm{D}_h\times G_h \times Z_h$ satisfies the equations that for $\forall (\bm{\Phi}_h,\chi_h,r_h)\in \bm{D}_h\times G_h \times Z_h$
\begin{equation} \label{eq:4.2}
\left\{\begin{array}{c}
\begin{aligned}
&\mathcal{A}_0(\bm{\eta}_h^n,\bm{\Phi}_h)+\mathcal{A}_1(\bm{\eta}_h^{n-1},\bm{\eta}_h^{n-1},\bm{\Phi}_h)+\mathcal{B}(\xi_h^n,\bm{\Phi}_h)
+\mathcal{Q}(\theta_h^n,\bm{\Phi}_h)=\langle L,\bm{\Phi}_h \rangle,\\
&\mathcal{B}(\chi_h,\bm{\eta}_h^n)=0,\\
&e(\theta_h^n,r_h)+\mathcal{H}(\bm{\eta}_h^{n-1},\theta_h^{n-1},r_h)=\Psi(r_h).
\end{aligned}
\end{array} \right.
\end{equation}
\textbf{Method II (Newton Iteration)}\\
Given $\bm{\eta}_h^{n-1} \in \Upsilon_h$ and $\theta_h^{n-1}\in Z_h$, find $(\bm{\eta}_h^n,\xi_h^n,\theta_h^n) \in \bm{D}_h\times G_h \times Z_h$ satisfies the equations that for $\forall (\bm{\Phi}_h,\chi_h,r_h)\in \bm{D}_h\times G_h \times Z_h$
\begin{equation} \label{eq:4.3}
\left\{\begin{array}{c}
\begin{aligned}
&\mathcal{A}_0(\bm{\eta}_h^n,\bm{\Phi}_h)+\mathcal{A}_1(\bm{\eta}_h^{n-1},\bm{\eta}_h^{n},\bm{\Phi}_h)
+\mathcal{A}_1(\bm{\eta}_h^{n},\bm{\eta}_h^{n-1},\bm{\Phi}_h)
-\mathcal{A}_1(\bm{\eta}_h^{n-1},\bm{\eta}_h^{n-1},\bm{\Phi}_h)\\
&+\mathcal{B}(\xi_h^n,\bm{\Phi}_h)+\mathcal{Q}(\theta_h^n,\bm{\Phi}_h)=\langle L,\bm{\Phi}_h \rangle,\\
&\mathcal{B}(\chi_h,\bm{\eta}_h^n)=0,\\
&e(\theta_h^n,r_h)+\mathcal{H}(\bm{\eta}_h^{n-1},\theta_h^{n},r_h)+\mathcal{H}(\bm{\eta}_h^{n},\theta_h^{n-1},r_h)
-\mathcal{H}(\bm{\eta}_h^{n-1},\theta_h^{n-1},r_h)=\Psi(r_h).
\end{aligned}
\end{array} \right.
\end{equation}
\textbf{Method III (Oseen Iteration)}\\
Given $\bm{\eta}_h^{n-1} \in \Upsilon_h$, find $(\bm{\eta}_h^n,\xi_h^n,\theta_h^n) \in \bm{D}_h\times G_h \times Z_h$ satisfies the equations that for $\forall (\bm{\Phi}_h,\chi_h,r_h)\in \bm{D}_h\times G_h \times Z_h$
\begin{equation} \label{eq:4.4}
\left\{\begin{array}{c}
\begin{aligned}
&\mathcal{A}_0(\bm{\eta}_h^n,\bm{\Phi}_h)+\mathcal{A}_1(\bm{\eta}_h^{n-1},\bm{\eta}_h^{n},\bm{\Phi}_h)
+\mathcal{B}(\xi_h^n,\bm{\Phi}_h)+\mathcal{Q}(\theta_h^n,\bm{\Phi}_h)=\langle L,\bm{\Phi}_h \rangle,\\
&\mathcal{B}(\chi_h,\bm{\eta}_h^n)=0,\\
&e(\theta_h^n,r_h)+\mathcal{H}(\bm{\eta}_h^{n-1},\theta_h^{n},r_h)=\Psi(r_h).
\end{aligned}
\end{array} \right.
\end{equation}
\begin{lem}
Under the conditions of Theorem 2.3, Assumptions 3.1 and 3.2 , the
initial values are well-defined as well as the following estimates hold:
\begin{equation}\label{eq:4.5}
\begin{aligned}
\|\nabla\theta_h^0\|_0 \le \lambda_{\psi},\quad
\|\bm{\eta}_h^0\|_1 \le \frac{\|L\|_*+\lambda_q\lambda_{\psi}}{ C_{\min}},\quad
\beta_1^*\|\xi_h^0\|
&\le (\|L\|_*+\lambda_q\lambda_{\psi})(1+\frac{C_{\max}}{C_{\min}}),
\end{aligned}
\end{equation}
and also satisfy the following boundedness
\begin{equation} \label{eq:4.8}
\begin{aligned}
\|\nabla z_{\theta}^0\|_0 \le\sigma \lambda_{\psi},\quad
\|z_{\bm{\eta}}^0\|_1 &\le \frac{\|L\|_*+\lambda_q\lambda_{\psi}}{ C_{\min}}\sigma,\quad
\beta^*\|z_{\xi}^0\|_0
&\le(C_{\max}+C_{\min})(\frac{\|L\|_*+\lambda_q\lambda_{\psi}}{C_{\min}})\sigma.
\end{aligned}
\end{equation}
\end{lem}
\begin{proof}
 Setting $(\bm{\Phi}_h,\chi_h,r_h)=(\bm{\eta}_h^0,\xi_h^0,\theta_h^0),$
\begin{equation*}
\left\{\begin{array}{c}
\begin{aligned}
&\mathcal{A}_0(\bm{\eta}_h^0,\bm{\eta}_h^0)+\mathcal{B}(\xi_h^0,\bm{\eta}_h^0)
+\mathcal{Q}(\theta_h^0,\bm{\eta}_h^0)=\langle L,\bm{\eta}_h^0 \rangle,\\
&\mathcal{B}(\xi_h^0,\bm{\eta}_h^0)=0,\\
&e(\theta_h^0,\theta_h^0)=\Psi(\theta_h^0).
\end{aligned}
\end{array} \right.
\end{equation*}
Combining (\ref{eq:3.8}), (\ref{eq:3.4}) and (\ref{eq:3.6}), we can have
\begin{equation}\label{eq:4.11}
\begin{aligned}
\|\nabla\theta_h^0\|_0 \le \lambda_{\psi},\quad
\|\bm{\eta}_h^0\|_1 \le \frac{\|L\|_*+\lambda_q\lambda_{\psi}}{ C_{\min}}.
\end{aligned}
\end{equation}
Now, let's consider the following equation:
\begin{equation*}
\begin{aligned}
\mathcal{B}(\xi_h^0,\bm{\Phi}_h)&=\langle L,\bm{\Phi}_h \rangle-\mathcal{A}_0(\bm{\eta}_h^0,\bm{\Phi}_h)-\mathcal{Q}(\theta_h^0,\bm{\Phi}_h).
\end{aligned}
\end{equation*}
Using the discrete inf-sup condition (\ref{eq:3.13}), (\ref{eq:3.3}) and (\ref{eq:3.6}), we can further obtain
\begin{equation} \label{eq:4.13}
\begin{aligned}
\beta^*\|\xi_h^0\|
\le (\|L\|_*+\lambda_q\lambda_{\psi})(1+\frac{C_{\max}}{C_{\min}}).
\end{aligned}
\end{equation}
Subtracting (\ref{eq:4.1}) from (\ref{eq:3.2}) and setting $(\bm{\Phi}_h,\chi_h,r_h)=(z_{\bm{\eta}}^0,z_\xi^0,z_{\theta}^0)$ yield
\begin{equation*}
\left\{\begin{array}{c}
\begin{aligned}
&\mathcal{A}_0(z_{\bm{\eta}}^0,z_{\bm{\eta}}^0)+\mathcal{A}_1(\bm{\eta}_h,\bm{\eta}_h,z_{\bm{\eta}}^0)
+\mathcal{B}(z_\xi^0,z_{\bm{\eta}}^0)
+\mathcal{Q}(z_{\theta}^0,z_{\bm{\eta}}^0)=0,\\
&\mathcal{B}(z_\xi^0,z_{\bm{\eta}}^0)=0,\\
&e(z_{\theta}^0,z_{\theta}^0)+\mathcal{H}(\bm{\eta}_h,\theta_h,z_{\theta}^0)=0.
\end{aligned}
\end{array} \right.
\end{equation*}
By (\ref{eq:3.8}), (\ref{eq:3.12}) and (\ref{eq:3.15}), we have
\begin{equation} \label{eq:4.14}
\begin{aligned}
\|\nabla z_{\theta}^0\|_0
\le \sigma\lambda_{\psi}.
\end{aligned}
\end{equation}
Applying (\ref{eq:4.14}), (\ref{eq:3.4}), (\ref{eq:3.9}) and (\ref{eq:3.6}), the following inequality is drawn:
\begin{equation} \label{eq:4.15}
\begin{aligned}
\|z_{\bm{\eta}}^0\|_1 &\le \frac{\|L\|_*+\lambda_q\lambda_{\psi}}{ C_{\min}}\sigma.
\end{aligned}
\end{equation}
Using the discrete inf-sup condition (\ref{eq:3.13}), (\ref{eq:4.15}), (\ref{eq:3.3}), (\ref{eq:3.10}) and (\ref{eq:3.6}) gives that
\begin{equation} \label{eq:4.16}
\begin{aligned}
\beta^*\|z_{\xi}^0\|_0
&\le(C_{\max}+C_{\min})(\frac{\|L\|_*+\lambda_q\lambda_{\psi}}{C_{\min}})\sigma.
\end{aligned}
\end{equation}
The proof is completed.
\end{proof}

\begin{The}\textbf{(Stokes Iteration)} Under the conditions of Theorem 2.3, Assumptions 3.1 and 3.2, provided
\begin{equation} \label{eq:4.24}
\begin{aligned}
0<\sigma<\frac{1}{4},
\end{aligned}
\end{equation}
for $\forall n\ge0$. The Methold I is stable, the numerical solution $ (\bm{\eta}_h^n,\xi_h^n,\theta_h^n)$ defined by equations (\ref{eq:4.2}) satisfies:
\begin{equation} \label{eq:4.25}
\begin{aligned}
\|\bm{\eta}_h^n\|_1\le2\frac{\|L\|_*+\lambda_q\lambda_{\psi}}{ C_{\min}},\quad
\|\theta_h^n\|_1\le 2\lambda_{\psi}.
\end{aligned}
\end{equation}
Moreover, the iterative errors $z_{\bm{\eta}}^n, z_{\xi}^n$ and $z_{\theta}^n$ satisfy
\begin{equation} \label{eq:4.27}
\begin{aligned}
\|z_{\bm{\eta}}^n\|_1 \le (3\sigma)^n(\frac{\|L\|_*+\lambda_q\lambda_{\psi}}{C_{\min}}),\quad
\|\nabla z_{\theta}^n\|_0 \le (3\sigma)^n\lambda_{\psi},\quad
\|z_{\xi}^n\|_0 \le (\tilde{C})(3\sigma)^n(\frac{\|L\|_*+\lambda_q\lambda_{\psi}}{ C_{\min}})
\end{aligned}
\end{equation}
where $\tilde{C}=C_{\max}+2C_{\min}.$
\end{The}
\begin{proof}
We will use mathematical induction to obtain the desired results.
As can be seen from Lemma 4.1, (\ref{eq:4.25}) and (\ref{eq:4.27})  hold for $n = 0$.
If equations (\ref{eq:4.25}), and (\ref{eq:4.27}) are valid for $n-1$, we will show that they also hold for $n$.
Taking $(\bm{\Phi}_h,\chi_h,r_h)=(\bm{\eta}_h^n,\xi_h^n,\theta_h^n)$, we have
\begin{equation*}
\left\{\begin{array}{c}
\begin{aligned}
&\mathcal{A}_0(\bm{\eta}_h^n,\bm{\eta}_h^n)+\mathcal{A}_1(\bm{\eta}_h^{n-1},\bm{\eta}_h^{n-1},\bm{\eta}_h^n)
+\mathcal{B}(\xi_h^n,\bm{\eta}_h^n)
+\mathcal{Q}(\theta_h^n,\bm{\eta}_h^n)=\langle L,\bm{\eta}_h^n \rangle,\\
&\mathcal{B}(\xi_h,\bm{\eta}_h^n)=0,\\
&e(\theta_h^n,\theta_h^n)+\mathcal{H}(\bm{\eta}_h^{n-1},\theta_h^{n-1},\theta_h^n)=\Psi(\theta_h^n).
\end{aligned}
\end{array} \right.
\end{equation*}
Applying (\ref{eq:3.7}) and (\ref{eq:3.12}), we further obtain
\begin{equation*}
\begin{aligned}
\|\nabla\theta_h^n\|_0  \le \lambda_{\psi}(1+4\lambda_2(\frac{\|L\|_*
+\lambda_q\lambda_{\psi}}{ C_{\min}}))
 \le 2\lambda_{\psi}.
\end{aligned}
\end{equation*}
And using (\ref{eq:3.3}), (\ref{eq:3.10}) and (\ref{eq:3.6}), it holds
\begin{equation*}
\begin{aligned}
\|\bm{\eta}_h^n\|_1 & \le \frac{\|L\|_*}{C_{\min}}+\frac{\lambda_2}{C_{\min}}\|\bm{\eta}_h^{n-1}\|_1^2+\frac{\lambda_q}{C_{\min}}\|\nabla \theta_h^n\|_0\\
& \le (\frac{\|L\|_*+\lambda_q\lambda_{\psi}}{ C_{\min}})+
4(\frac{\|L\|_*+\lambda_q\lambda_{\psi}}{ C_{\min}})(\frac{\lambda_2\|L\|_*}{C_{\min}^2}+\frac{\lambda_2\lambda_q\lambda_{\psi}}{ C_{\min}^2}+\frac{\lambda_2\lambda_q\lambda_{\psi}}{C_{\min}})\\
& \le 2\frac{\|L\|_*+\lambda_q\lambda_{\psi}}{ C_{\min}}.
\end{aligned}
\end{equation*}
As a consequence, we prove that inequalities (\ref{eq:4.25}) hold for $n\ge 0.$
Subtracting (\ref{eq:4.2}) from (\ref{eq:3.2}), the following result holds:
\begin{equation} \label{eq:4.28}
\left\{\begin{array}{c}
\begin{aligned}
&\mathcal{A}_0(z_{\bm{\eta}}^n,\bm{\Phi}_h)+\mathcal{A}_1(z_{\bm{\eta}}^{n-1},\bm{\eta}_h,\bm{\Phi}_h)
+\mathcal{A}_1(\bm{\eta}_h^{n-1},z_{\bm{\eta}}^{n-1},\bm{\Phi}_h)
+\mathcal{B}(z_{\xi}^n,\bm{\Phi}_h)+\mathcal{Q}(z_{\theta}^n,\bm{\Phi}_h)=0,\\
&\mathcal{B}(\chi_h,z_{\bm{\eta}}^n)=0,\\
&e(z_{\theta}^n,r_h)+\mathcal{H}(z_{\bm{\eta}}^{n-1},\theta_h,r_h)+\mathcal{H}(\bm{\eta}_h^{n-1},z_{\theta}^{n-1},r_h)=0.
\end{aligned}
\end{array} \right.
\end{equation}
Setting $r_h=z_{\theta}^n$, we have easily
\begin{equation*}
\begin{aligned}
e(z_{\theta}^n,z_{\theta}^n)+\mathcal{H}(z_{\bm{\eta}}^{n-1},\theta_h,z_{\theta}^n)
+\mathcal{H}(\bm{\eta}_h^{n-1},z_{\theta}^{n-1},z_{\theta}^n)=0.
\end{aligned}
\end{equation*}
By (\ref{eq:3.7}), (\ref{eq:3.12}), (\ref{eq:4.25}) and the second inequality of (\ref{eq:3.15}), we can get
\begin{equation} \label{eq:4.29}
\begin{aligned}
\|\nabla z_{\theta}^n\|_0
&\le \lambda_2\|z_{\bm{\eta}}^{n-1}\|_1\|\nabla \theta_h\|_0+\lambda_2\|\bm{\eta}_h^{n-1}\|_0\|\nabla z_{\theta}^{n-1}\|_0\\
& \le \lambda_{\psi} (3\sigma)^{n-1}(\frac{\lambda_2\|L\|_*}{ C_{\min}}+\frac{\lambda_2\lambda_q\lambda_{\psi}}{C_{\min}})
+2\lambda_{\psi}(3\sigma)^{n-1}(\frac{\lambda_2\|L\|_*}{ C_{\min}}+\frac{\lambda_2\lambda_q\lambda_{\psi}}{C_{\min}})\\
& \le \lambda_{\psi}(3\sigma)^n.
\end{aligned}
\end{equation}
Similarly choosing $(\bm{\Phi}_h,\chi_h)=(z_{\bm{\eta}}^n,z_{\xi}^n),$ we have
\begin{equation*}
\left\{\begin{array}{c}
\begin{aligned}
&\mathcal{A}_0(z_{\bm{\eta}}^n,z_{\bm{\eta}}^n)+\mathcal{A}_1(z_{\bm{\eta}}^{n-1},\bm{\eta}_h,z_{\bm{\eta}}^n)
+\mathcal{A}_1(\bm{\eta}_h^{n-1},z_{\bm{\eta}}^{n-1},z_{\bm{\eta}}^n)
+\mathcal{B}(z_{\xi}^n,z_{\bm{\eta}}^n)+Q(z_{\theta}^n,z_{\bm{\eta}}^n)=0,\\
&\mathcal{B}(z_{\xi}^n,z_{\bm{\eta}}^n)=0.\\
\end{aligned}
\end{array} \right.
\end{equation*}
Using (\ref{eq:3.4}), (\ref{eq:3.10}), (\ref{eq:3.6}), (\ref{eq:3.15}) and (\ref{eq:4.25}) leads to
\begin{equation} \label{eq:4.30}
\begin{aligned}
\|z_{\bm{\eta}}^n\|_1 & \le \frac{\lambda_2}{C_{\min}}\|z_{\bm{\eta}}^{n-1}\|\bm{\eta}_h\|_1
+\frac{\lambda_2}{C_{\min}}\|\bm{\eta}_h^{n-1}\|_1\|z_{\bm{\eta}}^{n-1}\|_1
+\frac{\lambda_q}{C_{\min}}\|\nabla z_{\theta}^n\|_0\\
& \le \frac{\lambda_2}{C_{\min}}(\frac{\|L\|_*+\lambda_q\lambda_{\psi}}{ C_{\min}})\|z_{\bm{\eta}}^{n-1}\|_1
+\frac{\lambda_2}{C_{\min}}(2\frac{\|L\|_*+\lambda_q\lambda_{\psi}}{ C_{\min}})\|z_{\bm{\eta}}^{n-1}\|_1\\
& \le (3\sigma)^{n}(\frac{\|L\|_*+\lambda_q\lambda_{\psi}}{ C_{\min}}).
\end{aligned}
\end{equation}
In term of (\ref{eq:4.28}), we arrive at
\begin{equation*}
\begin{aligned}
\mathcal{B}(z_{\xi}^n,\bm{\Phi}_h)=-\mathcal{A}_0(z_{\bm{\eta}}^n,\bm{\Phi}_h)
-\mathcal{A}_1(z_{\bm{\eta}}^{n-1},\bm{\eta}_h,\bm{\Phi}_h)-
\mathcal{A}_1(\bm{\eta}_h^{n-1},z_{\bm{\eta}}^{n-1},\bm{\Phi}_h)-\mathcal{Q}(z_{\theta}^n,\bm{\Phi}_h).\\
\end{aligned}
\end{equation*}
By combining the discrete inf-sup condition (\ref{eq:3.13}) with (\ref{eq:3.3}), (\ref{eq:3.10}) and (\ref{eq:3.7}), the following estimation can be obtained:
\begin{equation} \label{eq:4.31}
\begin{aligned}
\beta^*\|z_{\xi}^n\|_0 &\le C_{\max}\|z_{\bm{\eta}}^n\|_1+\lambda_2\|\bm{\eta}_h\|_1\|z_{\bm{\eta}}^{n-1}\|_1
+\lambda_2\|\bm{\eta}_h^{n-1}\|_1\|z_{\bm{\eta}}^{n-1}\|_1+\lambda_q\|\nabla z_{\theta}^n\|_0\\
&\le (\tilde{C})(3\sigma)^{n}(\frac{\|L\|_*+\lambda_q\lambda_{\psi}}{ C_{\min}}).
\end{aligned}
\end{equation}
As a consequence, we complete the proof.
\end{proof}
\begin{The}
\textbf{(Newton Iteration)} Under the conditions of Theorem 2.3, Assumptions 3.1 and 3.2. Provided
\begin{equation} \label{eq:4.32}
\begin{aligned}
0<\sigma<\frac{1}{3},
\end{aligned}
\end{equation}
for all $n\ge0$. The Methold II is stable, and the numerical solution $ (\bm{\eta}_h^n,\xi_h^n,\theta_h^n)$ defined by equation (\ref{eq:4.3}) satisfies:
\begin{equation} \label{eq:4.33}
\begin{aligned}
\|\bm{\eta}_h^n\|_1\le \frac{4}{3}\frac{\|L\|_*+\lambda_q\lambda_{\psi}}{ C_{\min}},\quad
\|\nabla\theta_h^n\|_0\le \frac{4}{3}(\lambda_{\psi}).
\end{aligned}
\end{equation}
Moreover, the iterative errors $z_{\bm{\eta}}^n, z_{\xi}^n$ and $z_{\theta}^n$ satisfy
\begin{equation}\label{eq:4.34}
\begin{aligned}
\|z_{\xi}^n\|_0 \le (\hat{C})(\frac{9}{5}\sigma)^{2^n-1}(\frac{\|L\|_*+\lambda_q\lambda_{\psi}}{ C_{\min}}),
\end{aligned}
\end{equation}
\begin{equation}\label{eq:4.35}
\begin{aligned}
\|z_{\xi}^n\|_0 \le (\hat{C})(\frac{9}{5}\sigma)^{2^n-1}(\frac{\|L\|_*+\lambda_q\lambda_{\psi}}{ C_{\min}}),
\end{aligned}
\end{equation}
where $\hat{C}=C_{\max}+\frac{22}{9}C_{\min}.$
\end{The}
\begin{proof}
 In the proof of this theorem, the mathematical thought we use is consistent with that of the previous theorem.
By Lemma 4.1, (\ref{eq:4.33})-(\ref{eq:4.35}) hold for $n = 0$.
Assume (\ref{eq:4.33})-(\ref{eq:4.35}) are valid for $n-1$, we should show that they also hold for $n.$
\begin{equation} \label{eq:4.36}
\left\{\begin{array}{c}
\begin{aligned}
&\mathcal{A}_0(z_{\bm{\eta}}^n,\bm{\Phi}_h)+\mathcal{A}_1(z_{\bm{\eta}}^{n},\bm{\eta}_h^{n-1},\bm{\Phi}_h)
+\mathcal{A}_1(\bm{\eta}_h^{n-1},z_{\bm{\eta}}^n,\bm{\Phi}_h)
+\mathcal{A}_1(z_{\bm{\eta}}^{n-1},z_{\bm{\eta}}^{n-1},\bm{\Phi}_h)\\
&+\mathcal{B}(z_{\xi}^n,\bm{\Phi}_h)+Q(z_{\theta}^n,\bm{\Phi}_h)=0,\\
&\mathcal{B}(\chi_h,z_{\bm{\eta}}^n)=0,\\
&e(z_{\theta}^n,r_h)+\mathcal{H}(z_{\bm{\eta}}^{n},\theta_h^{n-1},r_h)+\mathcal{H}(\bm{\eta}_h^{n-1},z_{\theta}^{n},r_h)
+\mathcal{H}(z_{\bm{\eta}}^{n-1},z_{\theta}^{n-1},r_h)=0.
\end{aligned}
\end{array} \right.
\end{equation}
Setting $(\bm{\Phi}_h,\chi_h,r_h)=(z_{\bm{\eta}}^n,z_{\xi}^n,z_{\theta}^n)$, we directly obtain
\begin{equation*}
\left\{\begin{array}{c}
\begin{aligned}
&\mathcal{A}_0(z_{\bm{\eta}}^n,z_{\bm{\eta}}^n)+\mathcal{A}_1(z_{\bm{\eta}}^{n},\bm{\eta}_h^{n-1},z_{\bm{\eta}}^n)
+\mathcal{A}_1(\bm{\eta}_h^{n-1},z_{\bm{\eta}}^n,z_{\bm{\eta}}^n)
+\mathcal{A}_1(z_{\bm{\eta}}^{n-1},z_{\bm{\eta}}^{n-1},z_{\bm{\eta}}^n)+\mathcal{B}(z_{\xi}^n,z_{\bm{\eta}}^n)
+\mathcal{Q}(z_{\theta}^n,z_{\bm{\eta}}^n)=0,\\
&\mathcal{B}(z_{\xi}^n,z_{\bm{\eta}}^n)=0,\\
&e(z_{\theta}^n,z_{\theta}^n)+\mathcal{H}(z_{\bm{\eta}}^{n},\theta_h^{n-1},z_{\theta}^n)
+\mathcal{H}(\bm{\eta}_h^{n-1},z_{\theta}^{n},z_{\theta}^n)+\mathcal{H}(z_{\bm{\eta}}^{n-1},z_{\theta}^{n-1},z_{\theta}^n)=0.
\end{aligned}
\end{array} \right.
\end{equation*}
From (\ref{eq:3.8}) and (\ref{eq:3.12}), we deduce that
\begin{equation} \label{eq:4.37}
\begin{aligned}
\|\nabla z_{\theta}^n\|_0
& \le \frac{4\lambda_2\lambda_{\psi}}{3}\|z_{\bm{\eta}}^n\|_1+\lambda_2\|z_{\bm{\eta}}^{n-1}\|_1\|\nabla z_{\theta}^{n-1}\|_0.\\
\end{aligned}
\end{equation}
Analogously, applying (\ref{eq:3.3}), (\ref{eq:3.10}) and (\ref{eq:3.6}), we have
\begin{equation} \label{eq:4.38}
\begin{aligned}
\|z_{\bm{\eta}}^n\|_1 &\le (\frac{9}{5}\sigma)^{2^{n}-1}(\frac{\|L\|_*+\lambda_q\lambda_{\psi}}{ C_{\min}}).\\
\end{aligned}
\end{equation}
Then,
\begin{equation} \label{eq:4.39}
\begin{aligned}
\|\nabla z_{\theta}^n\|_0
& \le \frac{4\lambda_2\lambda_{\psi}}{3}\|z_{\bm{\eta}}^n\|_1+\lambda_2\|z_{\bm{\eta}}^{n-1}\|_1\|\nabla z_{\theta_h}^{n-1}\|_0\\
& \le
\frac{4\lambda_2\lambda_{\psi}}{3}(\frac{9}{5}\sigma)^{2^{n}-1}
(\frac{\|L\|_*+\lambda_q\lambda_{\psi}}{ C_{\min}})+
\lambda_2(\frac{9}{5}\sigma)^{2^{n-1}-1}
(\frac{\|L\|_*+\lambda_q\lambda_{\psi}}{ C_{\min}})(\frac{9}{5}\sigma)^{2^{n-1}-1}\lambda_{\psi}\\
& \le (\frac{9}{5}\sigma)^{2^{n}-1}\lambda_{\psi}.
\end{aligned}
\end{equation}
With the help of (\ref{eq:4.36}), we arrive at
\begin{equation*}
\begin{aligned}
\mathcal{B}(z_{\xi}^n,\bm{\Phi}_h)=-\mathcal{A}_0(z_{\bm{\eta}}^n,\bm{\Phi}_h)
-\mathcal{A}_1(z_{\bm{\eta}}^{n},\bm{\eta}_h^{n-1},\bm{\Phi}_h)
-\mathcal{A}_1(\bm{\eta}_h^{n-1},z_{\bm{\eta}}^n,\bm{\Phi}_h)-\mathcal{A}_1(z_{\bm{\eta}}^{n-1},z_{\bm{\eta}}^{n-1},\bm{\Phi}_h)
-\mathcal{Q}(z_{\theta}^n,\bm{\Phi}_h).\\
\end{aligned}
\end{equation*}
Then, by (\ref{eq:3.3}), (\ref{eq:3.10}) and ({\ref{eq:3.6}}), the discrete inf-sup condition (\ref{eq:3.13}), we derive
\begin{equation} \label{eq:4.40}
\begin{aligned}
\beta^*\|z_{\xi}^n\|_0
& \le C_{\max}\|z_{\bm{\eta}}^n\|_1+\lambda_q\|\nabla z_{\theta}^n\|_0+2\lambda_2\|z_{\bm{\eta}}^n\|_1\|\bm{\eta}_h^{n-1}\|_1+\lambda_2\|z_{\bm{\eta}}^{n-1}\|_1^2\\
& \le C_{\max}(\frac{9}{5}\sigma)^{2^{n}-1}(\frac{\|L\|_*+\lambda_q\lambda_{\psi}}{ C_{\min}})
+\lambda_q(\frac{9}{5}\sigma)^{2^{n}-1} \lambda_{\psi}\\
&+2\lambda_2(\frac{9}{5}\sigma)^{2^{n}-1}(\frac{\|L\|_*+\lambda_q\lambda_{\psi}}{ C_{\min}})\frac{4}{3}(\frac{\|L\|_*+\lambda_q\lambda_{\psi}}{C_{\min}})
 +\lambda_2(\frac{9}{5}\sigma)^{2^{n}-2}(\frac{\|L\|_*+\lambda_q\lambda_{\psi}}{ C_{\min}})^2\\
&\le (\hat{C})(\frac{9}{5}\sigma)^{2^{n}-1}(\frac{ \|L\|_*+\lambda_q\lambda_{\psi}}{ C_{\min}}).
\end{aligned}
\end{equation}
Next, it follows from (\ref{eq:4.3}) and (\ref{eq:4.1}) with $n = 1$ that
\begin{equation} \label{eq:4.41}
\left\{\begin{array}{c}
\begin{aligned}
&\mathcal{A}_0(\bm{\eta}_h^1-\bm{\eta}_h^0,\bm{\Phi}_h)+\mathcal{A}_1(\bm{\eta}_h^{0},\bm{\eta}_h^{1},\bm{\Phi}_h)
+\mathcal{A}_1(\bm{\eta}_h^{1},\bm{\eta}_h^{0},\bm{\Phi}_h)
-\mathcal{A}_1(\bm{\eta}_h^{0},\bm{\eta}_h^{0},\bm{\Phi}_h),\\
&+\mathcal{B}(\xi_h^1-\xi_h^0,\bm{\Phi}_h)+\mathcal{Q}(\theta_h^1-\theta_h^0,\bm{\Phi}_h)=0,\\
&\mathcal{B}(\chi_h,\bm{\eta}_h^1-\bm{\eta}_h^0)=0,\\
&e(\theta_h^1-\theta_h^0,r_h)+\mathcal{H}(\bm{\eta}_h^{0},\theta_h^{1},r_h)+\mathcal{H}(\bm{\eta}_h^{1},\theta_h^{0},r_h)
-\mathcal{H}(\bm{\eta}_h^{0},\theta_h^{0},r_h)=0.
\end{aligned}
\end{array} \right.
\end{equation}
Setting $(\bm{\Phi}_h,\chi_h,r_h)=(\bm{\eta}^1-\bm{\eta}^0,\xi_h^1-\xi_h^0,\theta_h^1-\theta_h^0) $, and using (\ref{eq:3.8}), (\ref{eq:3.12}) and (\ref{eq:3.11}) lead to
 \begin{equation} \label{eq:4.42}
\begin{aligned}
\|\nabla(\theta_h^1-\theta_h^0)\|_0 \le \lambda_2\|\bm{\eta}_h^1\|_1\|\nabla\theta_h^0\|_0.
\end{aligned}
\end{equation}
Useing (\ref{eq:3.4}), (\ref{eq:3.9}), (\ref{eq:3.10}) and (\ref{eq:3.6}), the following inequality holds:
\begin{equation} \label{eq:4.43}
\begin{aligned}
\|\bm{\eta}_h^1-\bm{\eta}_h^0\|_1
& \le \frac{\lambda_q}{C_{\min}}\lambda_2\|\bm{\eta}_h^1\|_1\|\nabla\theta_h^0\|_0
+\frac{\lambda_2}{C_{\min}}\|\bm{\eta}_h^1\|_1 \|\bm{\eta}_h^0\|_1 \\
& \le \sigma\|\bm{\eta}_h^1\|_1.
\end{aligned}
\end{equation}
Taking $n = 1$ in equation (\ref{eq:4.3}) and setting $(\bm{\Phi}_h,\chi_h,r_h)=(\bm{\eta}_h^1,\xi_h^1,\theta_h^1)$ give that
\begin{equation*}
\left\{\begin{array}{c}
\begin{aligned}
&\mathcal{A}_0(\bm{\eta}_h^1,\bm{\eta}_h^1)+\mathcal{A}_1(\bm{\eta}_h^{0},\bm{\eta}_h^{1},\bm{\eta}_h^1)
+\mathcal{A}_1(\bm{\eta}_h^{1},\bm{\eta}_h^{0},\bm{\eta}_h^1)
-\mathcal{A}_1(\bm{\eta}_h^{0},\bm{\eta}_h^{0},\bm{\eta}_h^1)+\mathcal{B}(\xi_h^1,\bm{\eta}_h^1)
+\mathcal{Q}(\theta_h^1,\bm{\eta}_h^1)=\langle L,\bm{\eta}_h^1 \rangle,\\
&\mathcal{B}(\xi_h^1,\bm{\eta}_h^1)=0,\\
&e(\theta_h^1,\theta_h^1)+\mathcal{H}(\bm{\eta}_h^{0},\theta_h^{1},\theta_h^1)
+\mathcal{H}(\bm{\eta}_h^{1},\theta_h^{0},\theta_h^1)
-\mathcal{H}(\bm{\eta}_h^{0},\theta_h^{0},\theta_h^1)=\Psi(\theta_h^1).
\end{aligned}
\end{array} \right.
\end{equation*}
We can get
\begin{equation} \label{eq:4.44}
\begin{aligned}
\|\bm{\eta}_h^1\|_1
 \le \frac{9}{8}\frac{\|L\|_*+\lambda_q\lambda_{\psi}}{ C_{\min}},\quad
\|\nabla\theta_h^1\|_0
 \le \frac{9\lambda_{\psi}}{8}.
\end{aligned}
\end{equation}
Setting $(\bm{\Phi}_h,\chi_h,r_h)=(\bm{\eta}_h^n,\xi_h^n,\theta_h^n)$  in equations (\ref{eq:4.3}) with $n \ge 2$ leads to
\begin{equation*}
\left\{\begin{array}{c}
\begin{aligned}
&\mathcal{A}_0(\bm{\eta}_h^n,\bm{\eta}_h^n)+\mathcal{A}_1(\bm{\eta}_h^{n-1},\bm{\eta}_h^{n},\bm{\eta}_h^n)
+\mathcal{A}_1(\bm{\eta}_h^{n},\bm{\eta}_h^{n-1},\bm{\eta}_h^n)
-\mathcal{A}_1(\bm{\eta}_h^{n-1},\bm{\eta}_h^{n-1},\bm{\eta}_h^n)+\mathcal{B}(\xi_h^n,\bm{\eta}_h^n)
+\mathcal{Q}(\theta_h^n,\bm{\eta}_h^n)=\langle L,\bm{\eta}_h^n \rangle,\\
&\mathcal{B}(\xi_h^n,\bm{\eta}_h^n)=0,\\
&e(\theta_h^n,\theta_h^n)+\mathcal{H}(\bm{\eta}_h^{n-1},\theta_h^{n},\theta_h^n)+\mathcal{H}(\bm{\eta}_h^{n},\theta_h^{n-1},\theta_h^n)
-\mathcal{H}(\bm{\eta}_h^{n-1},\theta_h^{n-1},\theta_h^n)=\Psi(\theta_h^n).
\end{aligned}
\end{array} \right.
\end{equation*}
By (\ref{eq:3.8}), (\ref{eq:3.11}) and (\ref{eq:3.12}), we deduce
\begin{equation} \label{eq:4.45}
\begin{aligned}
\|\nabla\theta_h^n\|_0
 & \le \lambda_{\psi}+\lambda_2\|\bm{\eta}_h^n-\bm{\eta}_{h}^{n-1}\|_1\|\nabla(\theta_h^{n-1}
 -\theta_h^n)\|_0\\
& \le \lambda_{\psi}+\lambda_2\|z_{\bm{\eta}}^n-z_{\bm{\eta}}^{n-1}\|_1\|\nabla(z_{\theta}^{n-1}
-z_{\theta}^n)\|_0\\
& \le \frac{4}{3}\lambda_{\psi}.
\end{aligned}
\end{equation}
Applying (\ref{eq:3.4}), (\ref{eq:3.9}), (\ref{eq:3.10}) and (\ref{eq:3.6}), we have
\begin{equation} \label{eq:4.46}
\begin{aligned}
\|\bm{\eta}_h^n\|_1 & \le \frac{\|L\|_*}{C_{\min}}+\frac{\lambda_q}{C_{\min}}\|\nabla \theta_h^n\|_0+\frac{\lambda_2}{C_{\min}}\|\bm{\eta}_h^n-\bm{\eta}_h^{n-1}\|_1^2\\
& \le \frac{\|L\|_*+\lambda_q\lambda_{\psi}}{C_{\min}}(1+\frac{1}{3}((\frac{3}{5})^3+(\frac{3}{5}))^2)\\
& \le \frac{4}{3}\frac{\|L\|_*+\lambda_q\lambda_{\psi}}{C_{\min}}.
\end{aligned}
\end{equation}
The proof is completed.
\end{proof}

\begin{The}\textbf{(Oseen Iteration)}
Under the conditions of Theorem 2.3, Assumptions 3.1 and 3.2, the
Methold III is unconditional stable. and satisfies the following estimates
\begin{equation*}
\begin{aligned}
\|\nabla \theta_h^n\|_0\le \lambda_{\psi}, \quad
\|\bm{\eta}_h^n\|_1\le\frac{ \|L\|_*+\lambda_{\psi}\lambda_q}{C_{\min}},
\end{aligned}
\end{equation*}
and satisfy the following bounds
\begin{equation*}
\begin{aligned}
&\|z_{\bm{\eta}}^n\|_1 \le \sigma^{n+1}(\frac{\|L\|_*+\lambda_q\lambda_{\psi}}{ C_{\min}}), \quad
\|\nabla z_{\theta}^{n}\|_0 \le \sigma^{n+1}\lambda_{\psi}, \quad
 \beta^*\|z_{\xi}^n\|  \le (C_{\max}+3C_{\min})(\frac{\|L\|_*+\lambda_q\lambda_{\psi}}{ C_{min}}).
\end{aligned}
\end{equation*}
\end{The}
\begin{proof}
Considering model (\ref{eq:4.4}) and setting  $(\bm{\Phi}_h,\chi_h,r_h)=(\bm{\eta}_h^n,\xi_h^n,\theta_h^n),$ we can arrive at
\begin{equation*}
\left\{\begin{array}{c}
\begin{aligned}
&\mathcal{A}_0(\bm{\eta}_h^n,\bm{\eta}_h^n)+\mathcal{A}_1(\bm{\eta}_h^{n-1},\bm{\eta}_h^{n},\bm{\eta}_h^n)
+\mathcal{B}(\xi_h^n,\bm{\eta}_h^n)
+\mathcal{Q}(\theta_h^n,\bm{\eta}_h^n)=\langle L,\bm{\eta}_h^n \rangle,\\
&\mathcal{B}(\xi_h^n,\bm{\eta}_h^n)=0,\\
&e(\theta_h^n,\theta_h^n)+\mathcal{H}(\bm{\eta}_h^{n-1},\theta_h^{n},\theta_h^n)=\Psi(\theta_h^n).
\end{aligned}
\end{array} \right.
\end{equation*}
By (\ref{eq:3.8}) and (\ref{eq:3.11}), we obtain easily
\begin{equation} \label{eq:4.17}
\begin{aligned}
 \|\nabla\theta_h^n\|_0\le\lambda_{\psi}.
\end{aligned}
\end{equation}
 Similarly in terms of (\ref{eq:4.17}), while applying (\ref{eq:3.4}), (\ref{eq:3.9}) and (\ref{eq:3.6}), we can also get
\begin{equation} \label{eq:4.18}
\begin{aligned}
\|\bm{\eta}_h^n\|_1 \le \frac{\|L\|_*+\lambda_{\psi}\lambda_{q}}{C_{\min}}.
\end{aligned}
\end{equation}
In the next step, we give the estimation of the error bounds. Subtracting
(\ref{eq:4.4}) from (\ref{eq:3.2}), we deduce the following result
\begin{equation*}
\left\{\begin{array}{c}
\begin{aligned}
&\mathcal{A}_0(z_{\bm{\eta}}^n,\bm{\Phi}_h)+\mathcal{A}_1(z_{\bm{\eta}}^{n-1},\bm{\eta}_h,\bm{\Phi}_h)
+\mathcal{A}_1(\bm{\eta}_h^{n-1},z_{\bm{\eta}}^{n},\bm{\Phi}_h)
+\mathcal{B}(z_{\xi}^n,\bm{\Phi}_h)+\mathcal{Q}(z_{\theta}^n,\bm{\Phi}_h)=0,\\
&\mathcal{B}(\chi_h,z_{\bm{\eta}}^n)=0,\\
&e(z_{\theta}^n,r_h)+\mathcal{H}(z_{\bm{\eta}}^{n-1},\theta_h,r_h)+\mathcal{H}(\bm{\eta}_h^{n-1},z_{\theta}^n,r_h)=0.
\end{aligned}
\end{array} \right.
\end{equation*}
Setting $r_h=z_{\theta}^n$, it holds:
\begin{equation*}
\begin{aligned}
e(z_{\theta}^n,z_{\theta}^n)+\mathcal{H}(z_{\bm{\eta}}^{n-1},\theta_h,z_{\theta}^n)=0.
\end{aligned}
\end{equation*}
Using (\ref{eq:3.7}), (\ref{eq:3.12}) and the second inequality of (\ref{eq:2.16}) yields
\begin{equation} \label{eq:4.20}
\begin{aligned}
\|\nabla z_{\theta}^n\|_0
&\le\lambda_2\lambda_{\psi}\|\nabla z_{\bm{\eta}}^{n-1}\|_0.
\end{aligned}
\end{equation}
Choosing $(\bm{\Phi}_h,\chi_h)=(z_{\bm{\eta}}^n,z_{\xi}^n)$, we have
\begin{equation*}
\left\{\begin{array}{c}
\begin{aligned}
&\mathcal{A}_0(z_{\bm{\eta}}^n,z_{\bm{\eta}}^n)+\mathcal{A}_1(z_{\bm{\eta}}^{n-1},\bm{\eta}_h,z_{\bm{\eta}}^n)
+\mathcal{A}_1(\bm{\eta}_h^{n-1},z_{\bm{\eta}}^{n},z_{\bm{\eta}}^n)
+\mathcal{B}(z_{\xi}^n,z_{\bm{\eta}}^n)+\mathcal{Q}(z_{\theta}^n,z_{\bm{\eta}}^n)=0,\\
&\mathcal{B}(z_{\xi}^n,z_{\bm{\eta}}^n)=0.\\
\end{aligned}
\end{array} \right.
\end{equation*}
Combining (\ref{eq:3.4}), (\ref{eq:3.9}), (\ref{eq:3.10}) and (\ref{eq:3.6}), we arrive at
\begin{equation*}
\begin{aligned}
C_{\min}\|z_{\bm{\eta}}^n\|_1 \le \lambda_2\|z_{\bm{\eta}}^{n-1}\|_1\|\bm{\eta}_h\|_1+\lambda_{q}\|\nabla z_{\theta}^n\|_0.
\end{aligned}
\end{equation*}
By the first inequality of (\ref{eq:3.15}) and (\ref{eq:4.20}), we can obtain
\begin{equation} \label{eq:4.21}
\begin{aligned}
\|z_{\bm{\eta}}^n\|_1
&\le \sigma^n\|z_{\bm{\eta}}^0\|_1.
\end{aligned}
\end{equation}
Using (\ref{eq:4.8}), (\ref{eq:4.21}) and (\ref{eq:4.20}) can yield
\begin{equation} \label{eq:4.22}
\begin{aligned}
\|z_{\bm{\eta}}^n\|_1 \le \sigma^{n+1}(\frac{\|L\|_*+\lambda_q\lambda_{\psi}}{ C_{\min}}),\quad
\|\nabla z_{\theta}^{n}\|_0
\le \sigma^{n+1}\lambda_{\psi}.
\end{aligned}
\end{equation}
By the equation $$\mathcal{B}(z_{\xi}^n,\bm{\Phi}_h)=-\mathcal{A}_0(z_{\bm{\eta}}^n,\bm{\Phi}_h)-\mathcal{A}_1(z_{\bm{\eta}}^{n-1},\bm{\eta}_h,\bm{\Phi}_h)
-\mathcal{A}_1(\bm{\eta}_h^{n-1},z_{\bm{\eta}}^{n},\bm{\Phi}_h)-\mathcal{Q}(z_{\theta}^n,\bm{\Phi}_h),$$
and using (\ref{eq:3.3}), (\ref{eq:3.10}), (\ref{eq:3.6}) and (\ref{eq:3.13}), we further have
\begin{equation*}
\begin{aligned}
\beta^*\|z_{\xi}^n\| \le (C_{\max}+3C_{\min})\sigma^{n+1}(\frac{ \|L\|_*+\lambda_q\lambda_{\psi}}{C_{\min}}).
\end{aligned}
\end{equation*}
The proof is completed.
\end{proof}

\begin{rem}
Under the condition $0<\sigma<\frac{1}{4},$ Methods I-III are stable. Under the $0<\sigma<\frac{1}{3},$ Methods II and III are stable. And only Method III is unconditionally stable under the condition $0<\sigma<1.$ Hence, among our proposed methods, method III has the best stability.
\end{rem}

\begin{rem}
The convergence rates of Method I and Method III  respectively are : $\|z_{\eta}^n\|_1 \le (3\sigma)^n(\frac{\|L\|_*+\lambda_q\lambda_{\psi}}{C_{min}})$ and $\|z_{\eta}^n\|_1 \le \sigma^{n+1}(\frac{\|L\|_*+\lambda_q\lambda_{\psi}}{ C_{min}})$ which are linear convergent. Besides, the convergence rate of Method II is $\|z_{\eta}^n\|_1 \le (\frac{9}{5}\sigma)^{2^n-1}(\frac{\|L\|_*+\lambda_q\lambda_{\psi}}{ C_{min}})$  which is quadratic convergent. Therefore, among our proposed methods, Method II converges fastest.
\end{rem}

\section{Numerical experiments}
In this section, we evaluate the performance of the algorithms given in this paper through four numerical examples. In the first numerical example, we consider 2D/3D exact solution problem to verify the convergence rate of the iterative algorithms. The second one is the 2D singular solution problem to verify the convergence performance in a non-convex L-shaped domain. The thermal driven cavity problem is  presented in the third numerical example. In the last example, we consider a B$\grave{e}$nard convection problem.

For the spatial finite element discretizations , the velocity $\bm{u}$ and pressure $p$ are approximated by the Mini-elements, while we choose the lowest-order Raviart-Thomas element combined with the discontinuous and piecewise constant $P_0$ element to approximate current density $\bm{J}$ and electric potential $\phi$.
\subsection{Problems with smooth solutions}

The purpose of this example is to verify the convergence rate of the finite element solutions in $\Omega=(0,1)^d,d=2,3$. The parameters $Ra,Pr,\kappa$ are simply set to 1 and $\bm{B}=(0,0,1)^T$. The right-hand terms $\bm{f},\bm{g},\varphi$ can be given by the following exact solution : for $d=2$
\begin{equation*}
\left\{\begin{array}{c}
\begin{aligned}
&u_1=2\pi(\sin(\pi x))^2\cos(\pi y)\sin(\pi y), \quad u_2=-2\pi(\sin(\pi y))^2\cos(\pi x)\sin(\pi x),\\
&J_1=-1/20\pi \sin(\pi x)\cos(\pi y)\cos(\pi z),\quad J_2=1/10\pi \cos(\pi x)\sin(\pi y)\cos(\pi z),\\
&p=\cos(\pi y)\cos(\pi x),\quad \phi=x-1/2, \quad \theta=u_1+u_2.
\end{aligned}
\end{array} \right.
\end{equation*}
and for $d=3$
\begin{equation*}
\left\{\begin{array}{c}
\begin{aligned}
&u_1=-1/20\pi (\sin(\pi x))^2\sin(\pi y)\cos(\pi y)\sin(\pi z)\cos(\pi z), \\
&u_2=-1/10\pi \sin(\pi x)\cos(\pi x)(\sin(\pi y))^2\sin(\pi z)\cos(\pi z),\\
&u_3=-1/20\pi \sin(\pi x)\cos(\pi x)\sin(\pi y)\cos(\pi y)(\sin(\pi z))^2,\\
&p=1/10\cos(\pi x)\cos(\pi y)\cos(\pi z),\quad \phi =1/10\sin(\pi x)\sin(\pi y)\sin(\pi z),\\
&J_1=-1/20\pi \sin(\pi x)\cos(\pi y)\cos(\pi z),\quad J_2=1/10\pi \cos(\pi x)\sin(\pi y)\cos(\pi z)\\
&J_3=-1/20\pi \cos(\pi x)\cos(\pi y)\sin(\pi z),\quad \theta=u_1+u_2+u_3.
\end{aligned}
\end{array} \right.
\end{equation*}
Here the $j$-$th$ components of $\bm{u}$ and $\bm{J}$ are given by $u_j$ and $J_i$
, respectively. The numerical results are given in Tables 1-3 for 2D and Tables 4-6 for 3D. From the numerical result, we have the following points to explain:

First, we can further observe that the corresponding error of all variables is $O(h)$. This means that the optimal convergence rates of all variables reaches the optimum, which justify our theoretical analysis well.

Second, among the three iterative methods, Method I is obviously the fastest and Method II is the slowest and the errors of the three different iterative methods are the same almost.

Third, according to Proposition 7, the discrete current density has  exactly no divergence at all, but we can see that the approximate solution yields $\|\nabla \cdot J_h\|_0 $ in the order of $10^{-11}$ with almost no divergence from the tables.  The numerical integration error and rounding error cause $\|\nabla \cdot J_h\|_0 $ not to be exactly 0.

Finally, we can see from Figure 1 that Method I is applicable to small Rayleigh numbers. Moreover, Method II can handle some medium Rayleigh numbers and Method III can solve the steady thermally coupled  inductionless MHD equation with large Rayleigh number.

%\vspace{-3.0em}
\begin{table}
	\centering
	\caption{Numerical results in 2D for Method I}
    \resizebox{\textwidth}{!}{
	\scalebox{1}{
		\begin{tabular}{c c c c c c c c c}
			%%{|p{0.5cm}|p{1cm}|p{2cm}|p{1.5cm}|p{2cm}|p{1.5cm}|p{2cm}|p{1.5cm}|}
			%%{18cm}{@{\extracolsep{\fill}}|c|c|c|c|c|c|c|c|}
			\hline
			 h & $\|\nabla \bm{u}-\nabla \bm{u}_h\|_0$  & $\|p-p_h\|_0$  & $\|\bm{J}-\bm{J}_h\|_{div}$  & $\|\phi-\phi_h\|_0$ & $\|\nabla\theta-\nabla\theta_h\|_0$ & $\|\nabla\cdot\bm{J}\|_0$ & CPU(s) \\
			\hline
			 1/16 & 2.12(-  -) & 6.39e-1(-  -) & 6.92e-2(-  -)  & 1.62e-2(-  -)  &
1.39(-  -)& 2.04e-11 & 1.74\\
			 1/32 & 1.06(1.00) & 2.11e-1(1.60) &  3.47e-2(1.00) & 7.56e-3(1.10)  &
6.96e-1(1.00) & 2.04e-11 & 8.96\\
			 1/64 & 5.28e-1(1.00) & 7.27e-2(1.54) & 1.74e-2(1.00)   & 3.71e-3(1.03) &
3.48e-1(1.00) & 2.04e-11 & 39.71\\
			 1/128 & 2.64e-1(1.00) & 2.54e-2(1.52)  & 8.68e-3(1.00) & 1.84e-3(1.01) &
1.74e-1(1.00) & 2.13e-11 & 174.37\\
			\hline
	\end{tabular} }
}
\end{table}

\begin{table}
	\centering
	\caption{Numerical results in 2D for Method II}
\resizebox{\textwidth}{!}{
	\scalebox{1}{
		\begin{tabular}{c c c c c c c c c }
			%%{|p{0.5cm}|p{1cm}|p{2cm}|p{1.5cm}|p{2cm}|p{1.5cm}|p{2cm}|p{1.5cm}|}
			%%{18cm}{@{\extracolsep{\fill}}|c|c|c|c|c|c|c|c|}
			\hline
			 h & $\|\nabla \bm{u}-\nabla \bm{u}_h\|_0$  & $\|p-p_h\|_0$  & $\|\bm{J}-\bm{J}_h\|_{div}$  & $\|\phi-\phi_h\|_0$ & $\|\nabla\theta-\nabla\theta_h\|_0$ & $\|\nabla\cdot\bm{J}\|_0$ & CPU(s) \\
			\hline
			 1/16 & 2.12(-  -) & 6.39e-1(-  -) & 6.92e-2(-  -)  & 1.62e-2(-  -)  &
1.39(-  -)& 2.04e-11 & 0.94\\
			 1/32 & 1.06(1.00) & 2.11e-1(1.60) &  3.47e-2(1.00) & 7.56e-3(1.10)  &
6.96e-1(1.00) & 2.04e-11 & 4.40\\
			 1/64 & 5.28e-1(1.00) & 7.27e-2(1.54) & 1.74e-2(1.00)   & 3.71e-3(1.03) &
3.48e-1(1.00) & 2.04e-11 & 19.90\\
			 1/128 & 2.64e-1(1.00) & 2.54e-2(1.52)  & 8.68e-3(1.00) & 1.84e-3(1.01) &
1.74e-1(1.00) & 2.13e-11 & 86.31\\
			\hline
	\end{tabular} }
}
\end{table}

\begin{table}
	\centering
	\caption{Numerical results in 2D for Method III}
\resizebox{\textwidth}{!}{
	\scalebox{1}{
		\begin{tabular}{c c c c c c c c c}
			%%{|p{0.5cm}|p{1cm}|p{2cm}|p{1.5cm}|p{2cm}|p{1.5cm}|p{2cm}|p{1.5cm}|}
			%%{18cm}{@{\extracolsep{\fill}}|c|c|c|c|c|c|c|c|}
			\hline
			 h & $\|\nabla \bm{u}-\nabla \bm{u}_h\|_0$  & $\|p-p_h\|_0$  & $\|\bm{J}-\bm{J}_h\|_{div}$  & $\|\phi-\phi_h\|_0$ & $\|\nabla\theta-\nabla\theta_h\|_0$ & $\|\nabla\cdot\bm{J}\|_0$ & CPU(s) \\
			\hline
			 1/16 & 2.12(-  -) & 6.39e-1(-  -) & 6.92e-2(-  -)  & 1.62e-2(-  -)  &
1.39(-  -)& 2.04e-11 & 0.86\\
			 1/32 & 1.06(1.00) & 2.11e-1(1.60) &  3.47e-2(1.00) & 7.56e-3(1.10)  &
6.96e-1(1.00) & 2.04e-11 & 4.24\\
			 1/64 & 5.28e-1(1.00) & 7.27e-2(1.54) & 1.74e-2(1.00)   & 3.71e-3(1.03) &
3.48e-1(1.00) & 2.04e-11 & 19.87\\
			 1/128 & 2.64e-1(1.00) & 2.54e-2(1.52)  & 8.68e-3(1.00) & 1.84e-3(1.01) &
1.74e-1(1.00) & 2.14e-11 & 85.52\\
			\hline
	\end{tabular} }
}
\end{table}

%Table 6
\begin{table}
	\centering
	\caption{Numerical results in 3D for Method I}
\resizebox{\textwidth}{!}{
	\scalebox{1}{
		\begin{tabular}{c c c c c c c c c c }
			%%{|p{0.5cm}|p{1cm}|p{2cm}|p{1.5cm}|p{2cm}|p{1.5cm}|p{2cm}|p{1.5cm}|}
			%%{18cm}{@{\extracolsep{\fill}}|c|c|c|c|c|c|c|c|}
            \hline
			h & $\|\nabla \bm{u}-\nabla \bm{u}_h\|_0$  & $\|p-p_h\|_0$  & $\|\bm{J}-\bm{J}_h\|_{div}$  & $\|\phi-\phi_h\|_0$  & $\|\nabla\theta-\nabla\theta_h\|_0$ & $\|\nabla\cdot\bm{J}\|_0$ & CPU(s) \\
			\hline
			 1/8 & 1.20e-1(-  -)  & 8.29e-2(-  -) & 2.19e-2(-  -)  & 4.90e-3(-  -)  &
1.07e-1(- -) & 1.37e-12 & 9.88\\
			 1/12 & 8.16e-2(0.94) & 4.30e-2(1.62) & 1.46e-2(1.00)  & 3.27e-3(1.00)  & 7.27e-2(0.95) & 1.38e-12 & 38.7\\
			 1/16 & 6.17e-2(0.97) & 2.73e-2(1.58) & 1.10e-2(1.00)  & 2.45e-3(1.00)  &  5.50e-2(0.97) & 1.39e-12 & 105.54\\
			 1/20 & 4.95e-2(0.98) & 1.92e-2(1.58) & 8.77e-3(1.00)  & 1.96e-3(1.00)  & 4.41e-2(0.98) & 1.39e-12 & 244.97\\
			\hline
	\end{tabular} }
}
\end{table}

%Table 5
\begin{table}
	\centering
	\caption{Numerical results in 3D for Method II}
\resizebox{\textwidth}{!}{
	\scalebox{1}{
		\begin{tabular}{c c c c c c c c c c }
			\hline
			 h & $\|\nabla \bm{u}-\nabla \bm{u}_h\|_0$  & $\|p-p_h\|_0$  & $\|\bm{J}-\bm{J}_h\|_{div}$  & $\|\phi-\phi_h\|_0$  & $\|\nabla\theta-\nabla\theta_h\|_0$ & $\|\nabla\cdot\bm{J}\|_0$ & CPU(s) \\
			\hline
			 1/8 & 1.20e-1(-  -)  & 8.29e-2(-  -)  & 2.19e-2(-  -)  & 4.90e-3(-  -) &
 1.07e-1(- -) & 1.37e-12 & 14.124\\
			 1/12 & 8.16e-2(0.94) & 4.30e-2(1.62)  &  1.46e-2(1.00) & 3.27e-3(1.00) & 7.27e-2(0.95) & 1.38e-12 & 56.59\\
			 1/16 & 6.17e-2(0.97) & 2.73e-2(1.58)  & 1.10e-2(1.00)  & 2.45e-3(1.00) & 5.50e-2(0.97) & 1.39e-12 & 148.86\\
			 1/20 & 4.95e-2(0.98) & 1.92e-2(1.58)  & 8.88e-3(0.94) & 1.96e-3(1.00)  & 4.41e-2(0.98) & 1.39e-12 & 281.99\\
			\hline
	\end{tabular} }
}
\end{table}

%Table 4
\begin{table}
	\centering
	\caption{Numerical results in 3D for Method III}
\resizebox{\textwidth}{!}{
	\scalebox{1}{
		\begin{tabular}{c c c c c c c c c c }
			\hline
			 h & $\|\nabla \bm{u}-\nabla \bm{u}_h\|_0$  & $\|p-p_h\|_0$  & $\|\bm{J}-\bm{J}_h\|_{div}$  & $\|\phi-\phi_h\|_0$ & $\|\nabla\theta-\nabla\theta_h\|_0$ & $\|\nabla\cdot\bm{J}\|_0$ & CPU(s) \\
			\hline
			 1/8  & 1.20e-1(-  -) & 8.29e-2(-  -) & 2.19e-2(-  -)  & 4.90e-3(-  -)  &
 1.07(- -) & 1.37e-12 & 8.74\\
			 1/12 & 8.16e-2(0.94) & 4.30e-2(1.62)  &  1.46e-2(1.00) & 3.27e-3(1.00) & 7.27e-2(0.95) & 1.38e-12 & 32.45\\
			 1/16 & 6.17e-2(0.97) & 2.73e-2(1.58)  & 1.10e-2(1.00)  & 2.45e-3(1.00) & 5.50e-2(0.97) & 1.39e-12 & 95.60\\
			 1/20 & 4.95e-2(0.98) & 1.92e-2(1.58)  & 8.77e-3(1.00)  & 1.96e-3(1.00) & 4.41e-2(0.98) & 1.39e-12 & 222.70\\
			\hline
	\end{tabular} }
}
\end{table}

\subsection{ Problems with singular solutions}
In order to verify the ability of this method to capture singularities, we consider the thermally coupled MHD problem in a non-convex L-shaped domain $\Omega =(-0.5,0.5)^2/([0,0.5)\times(-0.5,0])$.
As a result of the re-entrant corner presented in the domain, the exact solutions have strong singularities at this corner, namely the origin of coordinates. We set $\kappa=Pr=Ra=1$ and $\bm{B}=[0,0,1]^T.$ And the force term and boundary conditions are selected so that the analytical solution is as follows. Let $(r^*,\theta^*)$ be the polar coordinate and $\alpha=\frac{3\pi}{2}$.
 The minimum positive solution of this equation $\mu \sin(\alpha)+\sin(\mu\alpha)=0$ is given to the parameter $\mu$ as well as the exact solutions are of defined by
\begin{equation*}
\left\{\begin{array}{c}
\begin{aligned}
&\bm{R}(\theta^*)=\sin((1+\mu)\theta^*)\frac{\cos(\mu \alpha)}{1+\mu}\cos((1+\mu)\theta^*)-\sin((1-\mu)\theta^*)\frac{\cos(\mu\alpha)}{1-\mu}+\cos(\theta^*(1-\mu)\theta^*),\\
&u_1(r^*,\theta^*)=(r^*)^{\mu}((1+\mu)\sin(\theta^*)\bm{R}(\theta^*)+\cos(\theta^*)\bm{R^{'}}(\theta^*)),\
J(r^*,\theta^*)=\nabla((r^*)^{2/3}\sin(2/3\theta^*)),\\
&u_2(r^*,\theta^*)=(r^*)^{\mu}(-(1+\mu)\cos(\theta^*)\bm{R}(\theta^*)+\sin(\theta^*)\bm{R}^{'}(\theta^*)),
\quad \phi(r^*,\theta^*)\equiv 0,\\
&p(r^*,\theta^*)=-(r^*)^{\mu-1}((1+\mu)^2\bm{R}^{'}(\theta^*))+\bm{R}^{'''}(\theta^*)/(1-\mu),\quad
\quad \theta(r^*,\theta^*)=u_1+u_2.
\end{aligned}
\end{array} \right.
\end{equation*}
In fact, we can get $(\bm{u},p)\in\bm{H}^{1+\mu}(\Omega)\times H^{\mu}(\Omega)$ along with $\bm{J}\in \bm{H}^{2/3}(\Omega).$ The conductive boundary is considered in the example instead of the insulating boundary condition. This means the electric potential $\phi=0$ on $\partial\Omega.$

The numerical results of Methods I-III under several different grid sizes showed in Tables 7-9.  From the numerical results in Tables 7-9, it is shown that the optimal numerical convergence orders O(h$^{2/3}$) of these methods are consistent with that predicted by the theoretical analysis in the foregoing. Besides, different from the first example, Method II takes the least CPU time for the reason maybe that the fast convergence speed saves CPU time greatly.   Ultimately, the streamline of numerical solution velocity and current density and the contour of their components are shown respectively in Figures 2-3. Therefore, the methods proposed can effectively deal with the non convex region problems.

% Table7
\begin{table}
	\centering
	\caption{Numerical results with L-shaped domain in 2D for Method I}
\resizebox{\textwidth}{!}{
	\scalebox{1}{
		\begin{tabular}{c c c c c c c c c c c c }
			%%{|p{0.5cm}|p{1cm}|p{2cm}|p{1.5cm}|p{2cm}|p{1.5cm}|p{2cm}|p{1.5cm}|}
			%%{18cm}{@{\extracolsep{\fill}}|c|c|c|c|c|c|c|c|}
			\hline
			 h & $\|\nabla \bm{u}-\nabla \bm{u}_h\|_0$  & $\|p-p_h\|_0$  & $\|\bm{J}-\bm{J}_h\|_{div}$  & $\|\phi-\phi_h\|_0$   & $\|\nabla\theta-\nabla\theta_h\|_0$ & $\|\nabla\cdot\bm{J}\|_0$ & CPU(s) \\
			\hline
			 1/16 & 5.09e-1(-  -) & 5.99e-1(-  -) & 4.82e-2(-  -)  & 5.17e-4(-  -)  &
 5.53e-1(- -) & 5.27e-14 & 15.76 \\
			 1/32 & 3.49e-1(0.54) & 3.99e-1(0.59)  &  3.07e-2(0.65) & 2.09e-4(1.30) & 3.81e-1(0.54) & 7.26e-14 & 73.15 \\
			 1/64 & 2.39e-1(0.54) & 2.71e-1(0.56) & 1.95e-2(0.66)   & 8.63e-5(1.28)  & 2.62e-1(0.54) & 1.55e-13 & 267.18 \\
			 1/128 & 1.64e-1(0.54) & 1.85e-1(0.55)  & 1.23e-2(0.66) & 3.67e-5(1.24) & 1.80e-1(0.54) & 3.19e-13 & 1117.29 \\
			\hline
	\end{tabular} }
}
\end{table}

%Table 8
\begin{table}
	\centering
	\caption{Numerical results with L-shaped domain in 2D for Method II}
\resizebox{\textwidth}{!}{
	\scalebox{1}{
		\begin{tabular}{c c c c c c c c c c }
			%%{|p{0.5cm}|p{1cm}|p{2cm}|p{1.5cm}|p{2cm}|p{1.5cm}|p{2cm}|p{1.5cm}|}
			%%{18cm}{@{\extracolsep{\fill}}|c|c|c|c|c|c|c|c|}
			\hline
			 h & $\|\nabla \bm{u}-\nabla \bm{u}_h\|_0$  & $\|p-p_h\|_0$  & $\|\bm{J}-\bm{J}_h\|_{div}$  & $\|\phi-\phi_h\|_0$ & $\|\nabla\theta-\nabla\theta_h\|_0$ & $\|\nabla\cdot\bm{J}\|_0$ & CPU(s) \\
			\hline
			 1/16 & 5.09e-1(-  -) & 5.99e-1(-  -) & 4.82e-2(-  -)  & 5.17e-4(-  -)  &
 5.53e-1(- -) & 5.10e-14 & 7.89\\
			 1/32 & 3.49e-1(0.54) & 3.99e-1(0.59)  &  3.07e-2(0.65) & 2.09e-4(1.30) & 3.81e-1(0.54) & 7.00e-14 & 32.01\\
			 1/64 & 2.39e-1(0.54) & 2.71e-1(0.56) & 1.95e-2(0.66)   & 8.63e-5(1.28)  & 2.62e-1(0.54) & 1.59e-13 & 137.45\\
			 1/128 & 1.64e-1(0.54) & 1.85e-1(0.55)  & 1.23e-2(0.66) & 3.67e-5(1.24) & 1.80e-1(0.54) & 3.18e-13 & 566.89\\
			\hline
	\end{tabular} }
}
\end{table}

\begin{table}
	\centering
	\caption{Numerical results with L-shaped domain in 2D for Method III}
\resizebox{\textwidth}{!}{
	\scalebox{1}{
		\begin{tabular}{c c c c c c c c c c }
			%%{|p{0.5cm}|p{1cm}|p{2cm}|p{1.5cm}|p{2cm}|p{1.5cm}|p{2cm}|p{1.5cm}|}
			%%{18cm}{@{\extracolsep{\fill}}|c|c|c|c|c|c|c|c|}
			\hline
			 h & $\|\nabla \bm{u}-\nabla\bm{u}_h\|_0$  & $\|p-p_h\|_0$  & $\|\bm{J}-\bm{J}_h\|_{div}$  & $\|\phi-\phi_h\|_0 $  & $\|\nabla\theta-\nabla\theta_h\|_0$ & $\|\nabla\cdot\bm{J}\|_0$ & CPU(s) \\
			\hline
			 1/16 & 5.09e-1(-  -) & 5.99e-1(-  -) & 4.82e-2(-  -)  & 5.17e-4(-  -)  &
 5.53(- -) & 5.17e-14 & 7.94\\
			 1/32 & 3.49e-1(0.54) & 3.99e-1(0.59)  &  3.07e-2(0.65) & 2.09e-4(1.30) & 3.81(0.54) & 6.85e-14 & 33.43\\
			 1/64 & 2.39e-1(0.54) & 2.71e-1(0.56) & 1.95e-2(0.66)   & 8.63e-5(1.28)  & 2.62e-1(0.54) & 1.56e-13 & 149.73\\
			 1/128 & 1.64e-1(0.54) & 1.85e-1(0.55)  & 1.23e-2(0.66) & 3.67e-5(1.24) & 1.80e-1(0.54) & 3.22e-13 & 635.81\\
			\hline
	\end{tabular} }
}
\end{table}

\subsection{Thermal driven cavity problem}
Next, a classic benchmark test, namely the thermal driven cavity problem is considered. We consider  a square cavity with differential thermal vertical walls in $\Omega=[0,1]^2$, in which the left and right walls are separately kept at $\theta_l$ and $\theta_r$ with $\theta_l>\theta_r$. Setting $\theta_l=1$ and $\theta_r=0$ and the rest of the walls are insulated. In order to greatly verify the effectiveness of the iterative methods, the fluid in the cavity is treated as air in our model and we take Pr=0.71, $\kappa=1$, Ra varies within the range $[10^3$-$10^5]$, $\bm{f} = \bm{g}=\bm{0}$ and $\psi=0$. We use the Newton iterative method for momentum equation and temperature equation to improve appropriately the accuracy while consider no-slip boundary  for the velocity, namely $\bm{u}= \bm{0}$ on $\partial\Omega$.

We present the vertical/horizontal velocity at mid-height for various Ra in Figure 4. Then, we further study the change of the forms of various variables with $Ra=10^3,10^4,10^5$. Firstly, Figure 5 shows the streamlines of velocity, we can easily observe that the streamline of velocity changes from one large vortex to two large vortices, and then and moved to both sides. Secondly, the change in the streamlines pattern of the current density from curves to straight lines can be observed in Figure 6. Thirdly, Figure 7 shows that as the change of $Ra$ the proportion of large value of potential in the whole region is getting higher and higher.  Finally, we can see in Figure 8 that the temperature distribution becomes more disordered.

\subsection{ B$\acute{e}$nard convection problem}
Finally, we consider a B$\acute{e}$nard convection problem in domain $\Omega=[0,5]\times[0,1]$ to consider its effectiveness in more detail. Setting $\bm{u}=\bm{0}$ on $\partial\Omega$, the bottom and the top walls are enforced by $\theta$$=1$(or $\sin(x/5)$) and $\theta$$=0$, the right $(x$=$5)$ and left $(x$$=0)$ walls are adiabatic,  respectively. Here we only consider the Oseen iterative method with $Pr$$=1,\kappa$$=1$  while the source $\bm{f}$$=\bm{g}=$$\bm{0}$ and $\varphi$$=0$.

 Figures 9-12 report the velocity streamlines, current density streamlines, potential isolines along with isotherms for various kinds $Ra = 3\times10^3, 10^5$ with homogeneous heating $\theta= 1$ on the bottom wall. It can be clearly observed that the vortexes become more inclined and the value becomes larger with the increase of $Ra$  in Figure 9. And then the value of current density increases slowly, and the distribution radian of streamline on the upper and lower sides gradually increases, showed in Figure 10. Ultimately, as can be seen from Figures 11-12, with the increase of $Ra$, the value of the potential becomes larger and  performance of the temperature gets also more and more complicated.

The numerical results of non-uniform heating $\theta$$=\sin(x/5)$ on the bottom wall with $Ra=10^3,5\times 10^4$ are shown in the Figures 13-16.  In Figure 13, we can clearly observe that  evolved into a large right vortex. Furthermore, the changes in current density and the temperature field  are consistent with those under uniform heating conditions, we can see from Figures 14 and 16. Last, Figure 15 shows that the isoline distribution of the potential changes greatly, and the value also decreases. From the above discussion, we know that the used method  can simulate the B$\grave{e}$nard convection problem for large Rayleigh number very well.

%\begin{comment}
\begin{figure}[!htbp]
\begin{center}
\subfigure[]{\includegraphics[width=0.32\textwidth]{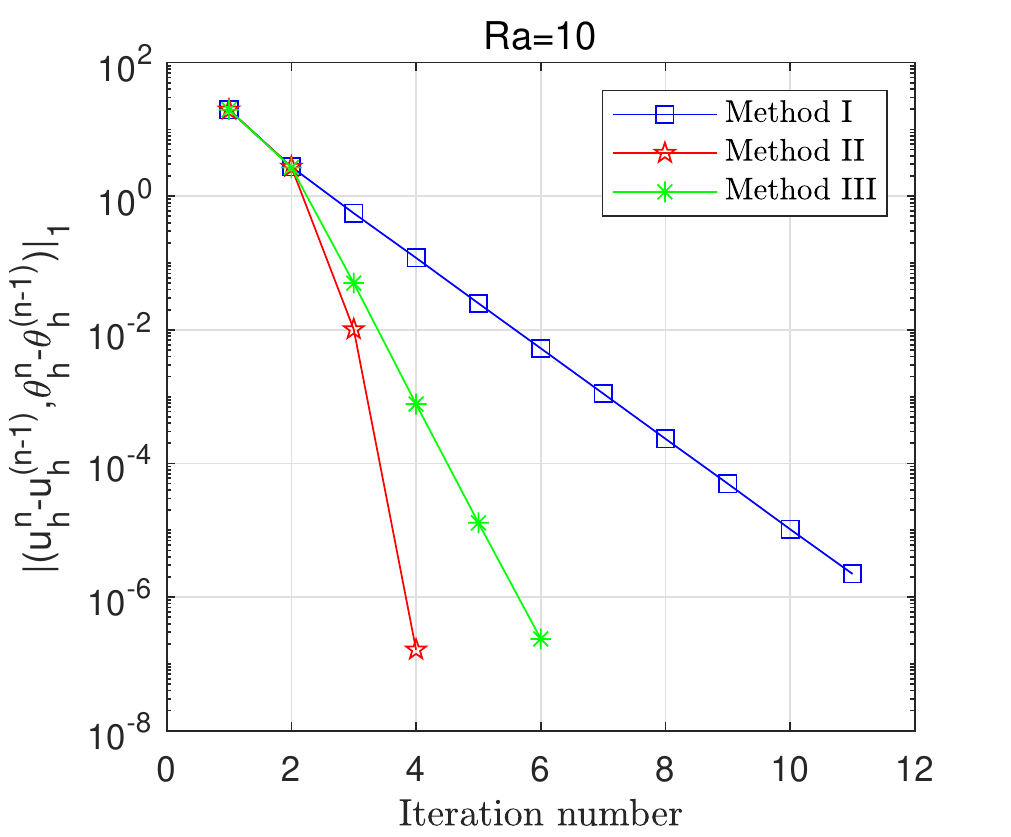}}
\subfigure[]{\includegraphics[width=0.32\textwidth]{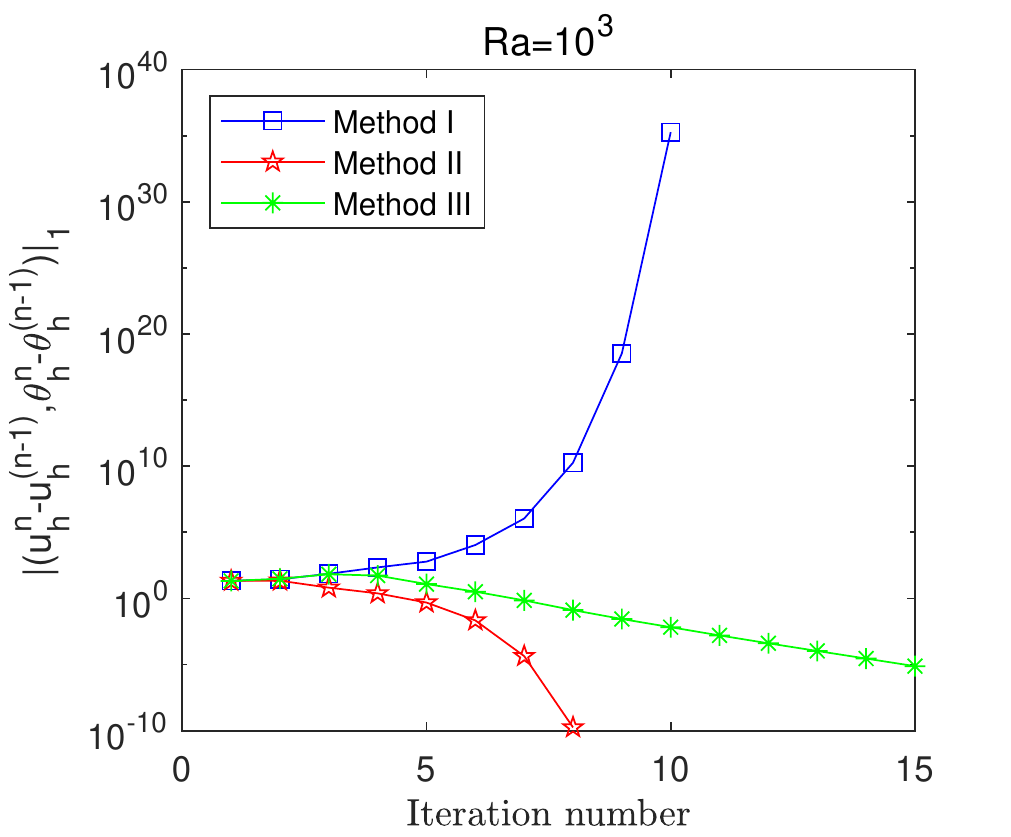}}
\subfigure[]{\includegraphics[width=0.32\textwidth]{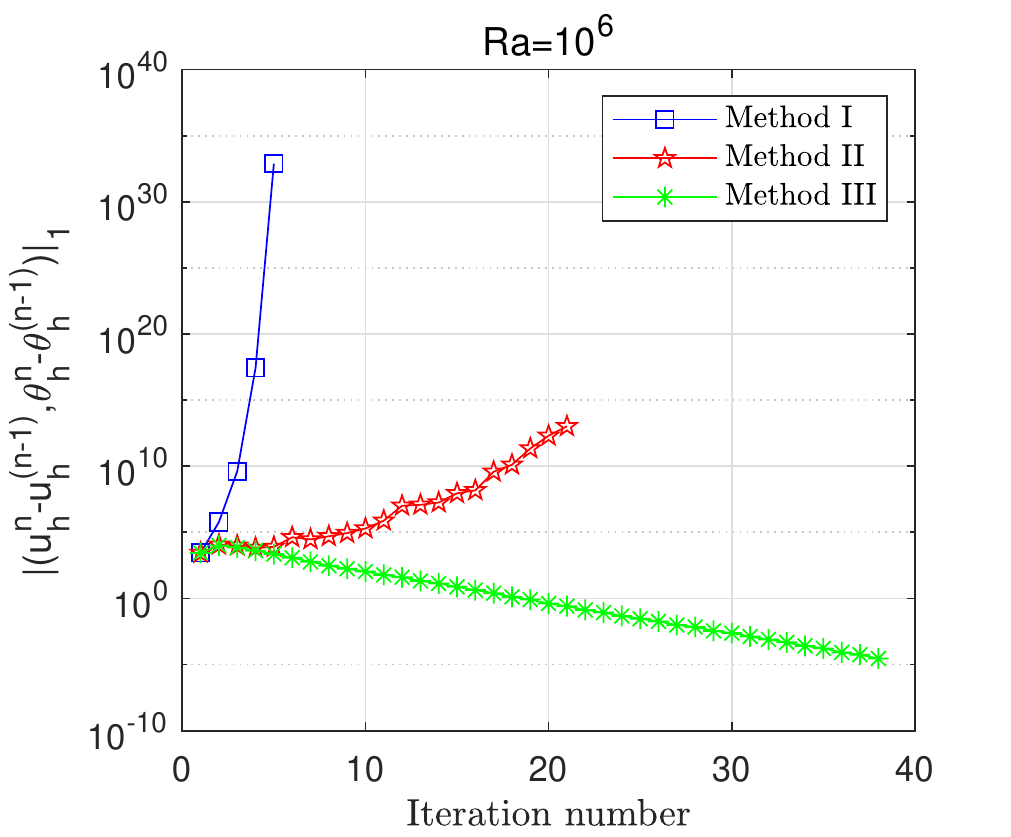}}
\caption{The iteration convergence errors with Ra$= 10,10^3$ and $10^6$.}
\end{center}
\end{figure}

%\begin{comment}
\begin{figure}[!htbp]
\begin{center}
\subfigure[]{\includegraphics[width=0.32\textwidth]{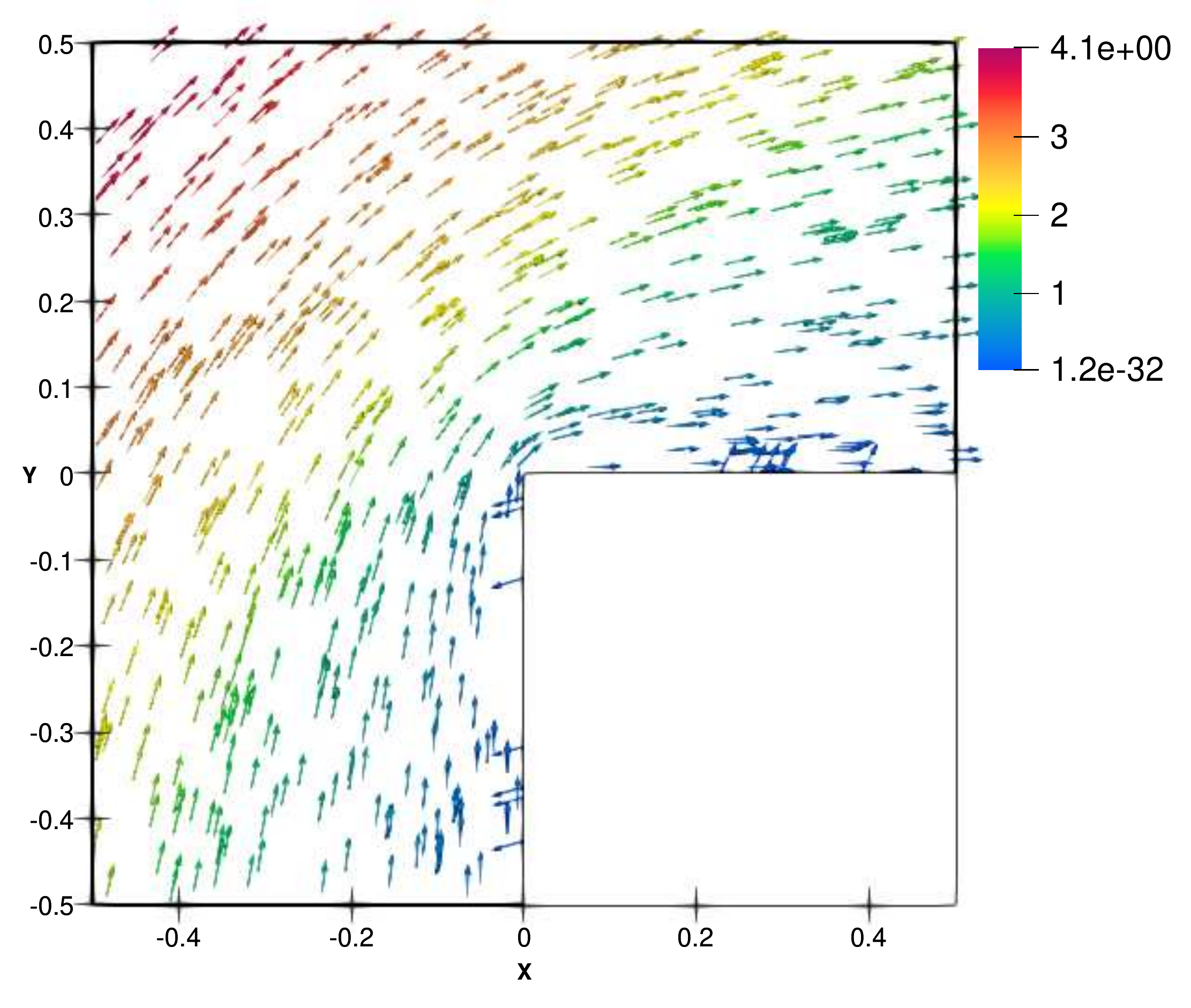}}
\subfigure[]{\includegraphics[width=0.32\textwidth]{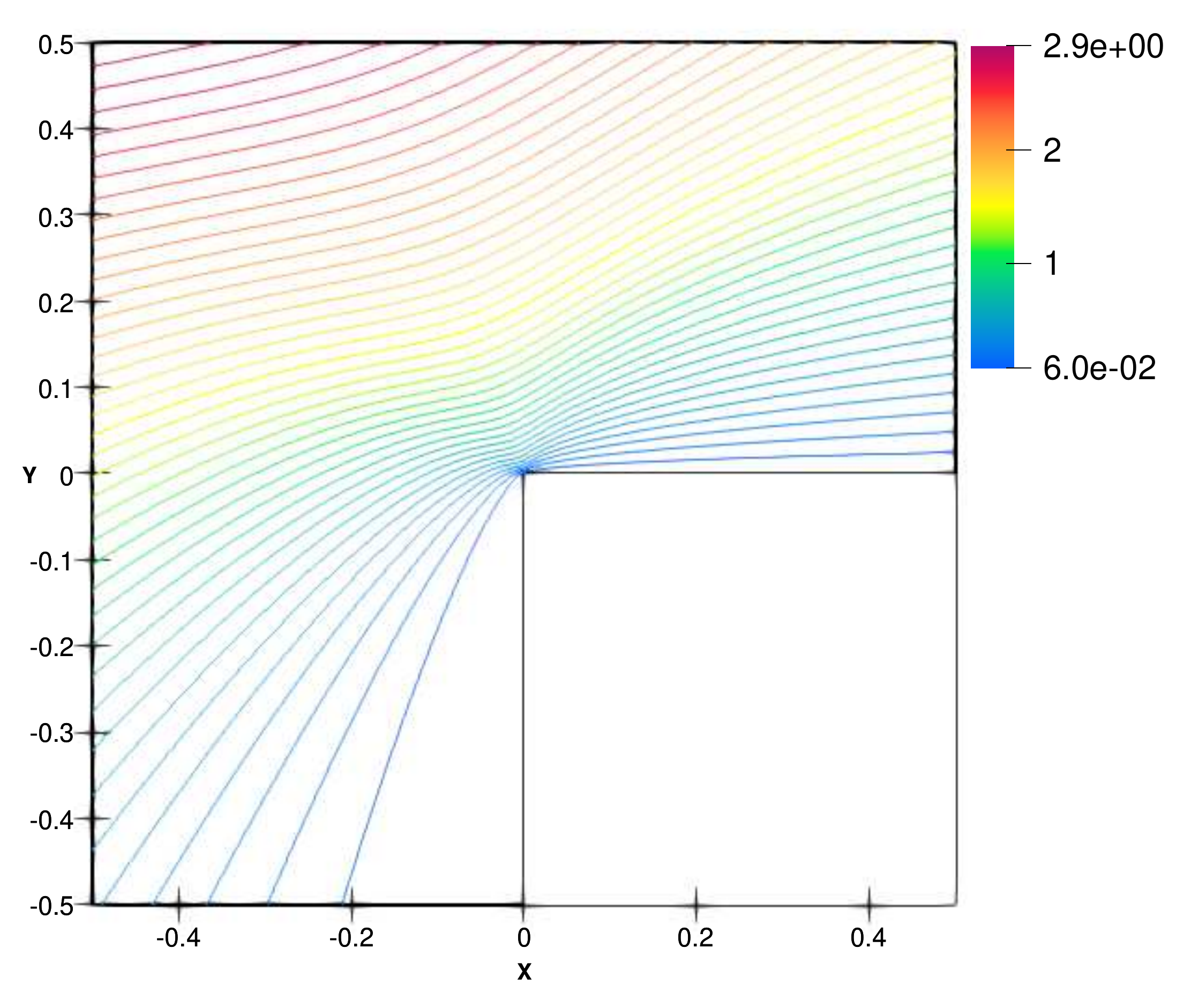}}
\subfigure[]{\includegraphics[width=0.32\textwidth]{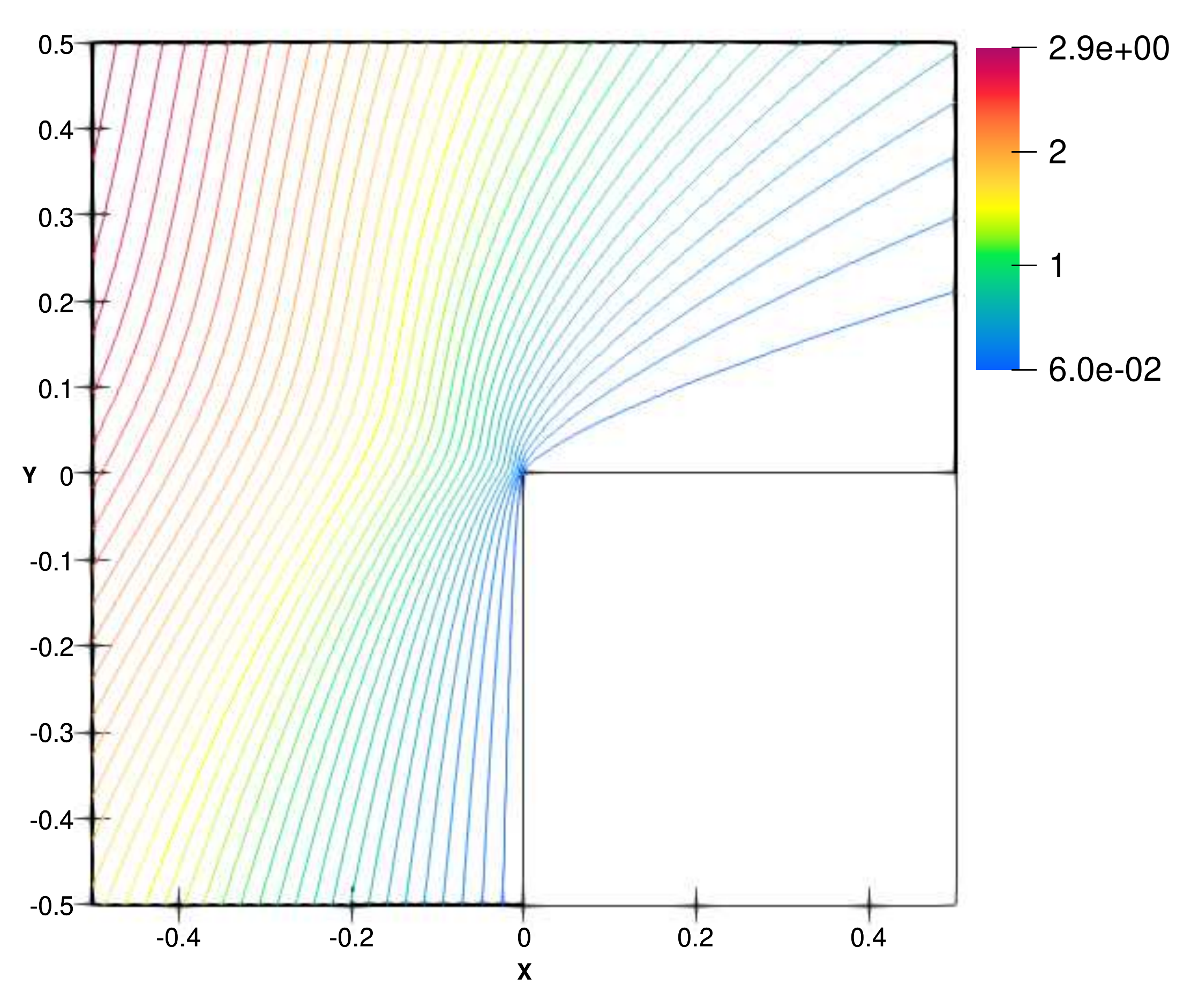}}
\caption{(a) Numerical approximations of $\bm{u}$, (b) contours of $u_1$, (c) contours of $u_2$.}
\end{center}
\end{figure}

\begin{figure}[!htbp]
\begin{center}
\subfigure[]{\includegraphics[width=0.32\textwidth]{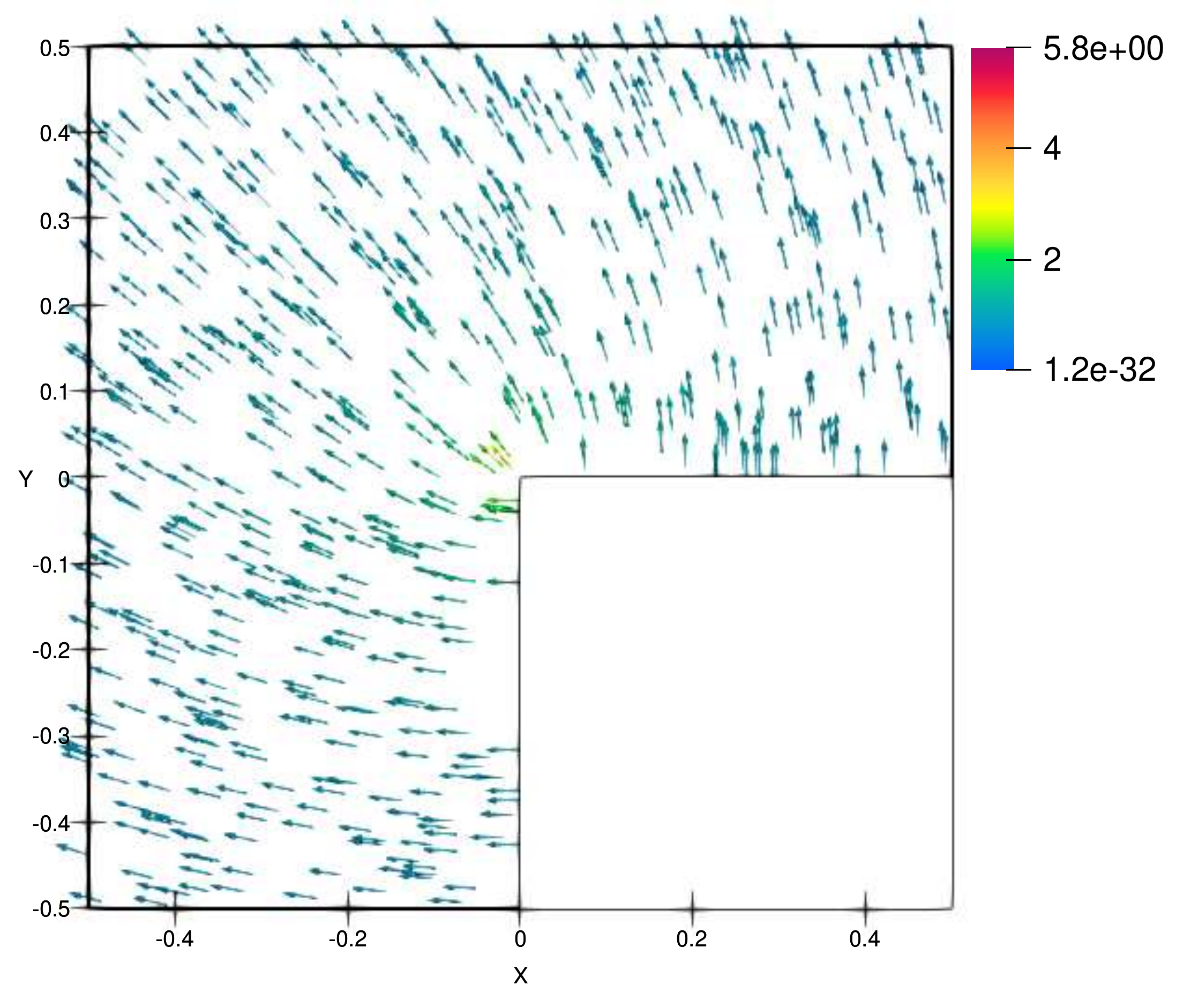}}
\subfigure[]{\includegraphics[width=0.32\textwidth]{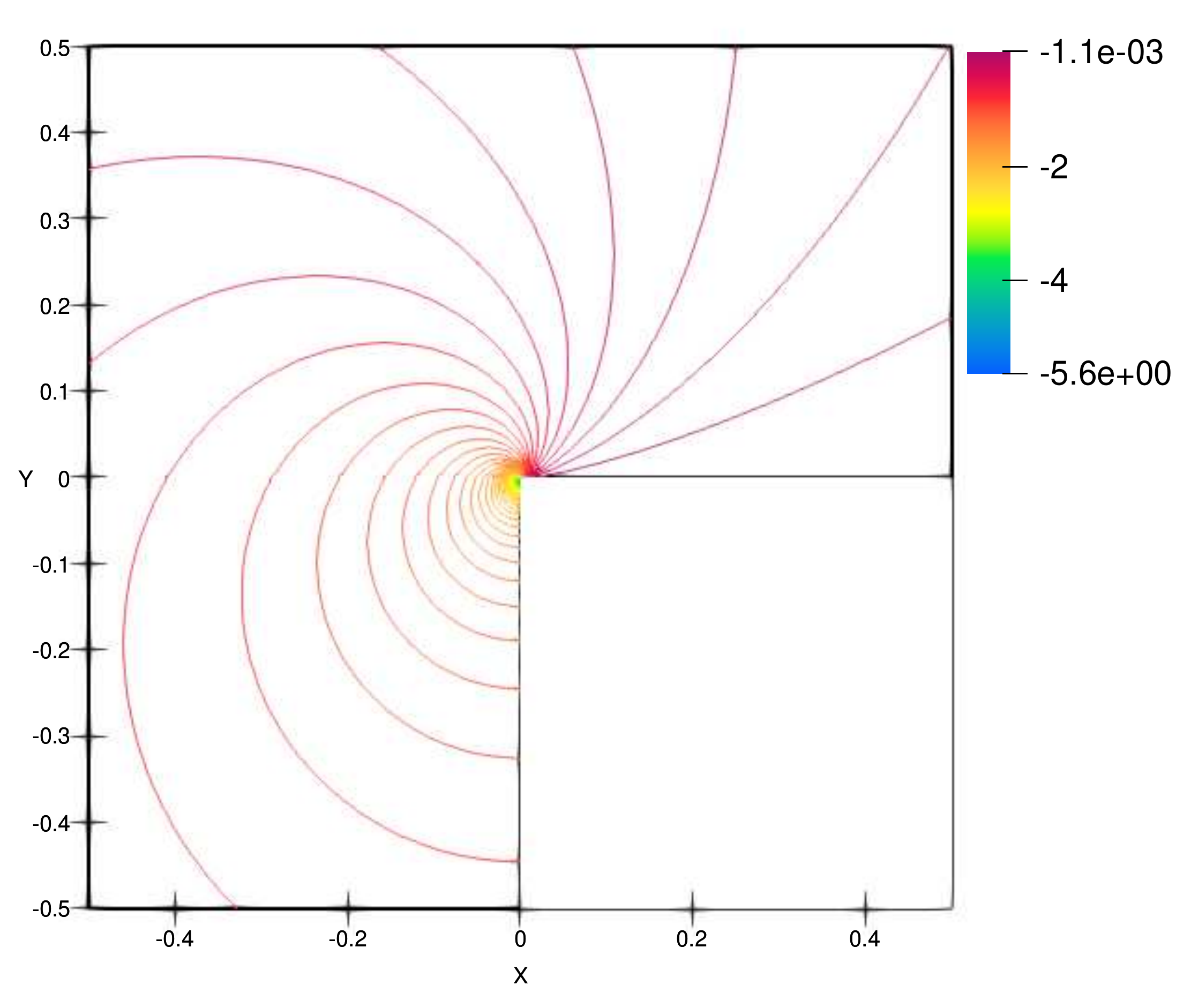}}
\subfigure[]{\includegraphics[width=0.32\textwidth]{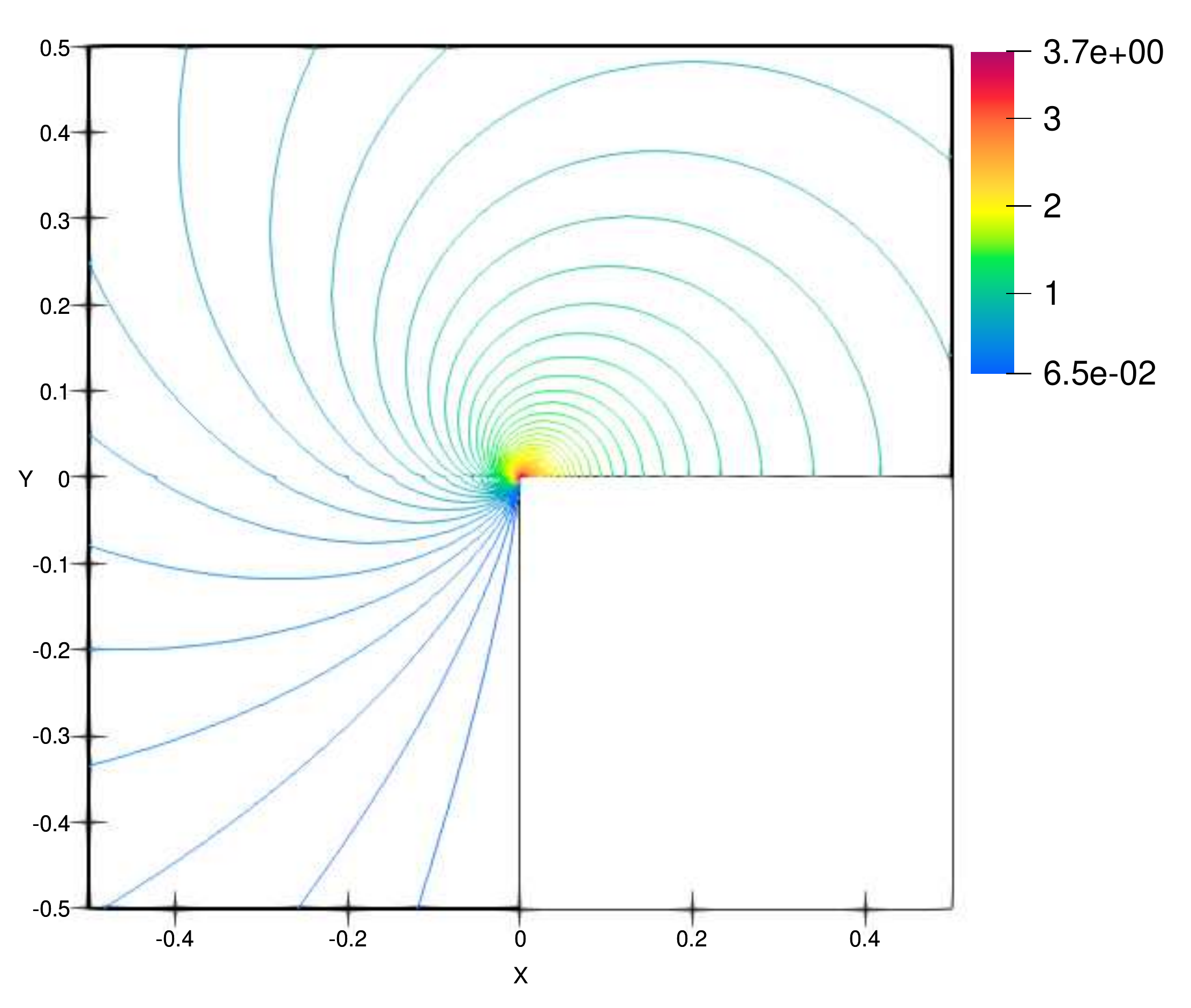}}
\caption{(a) Numerical approximations of $\bm{J}$, (b) contours of $J_1$, (c) contours of $J_2$.}
\end{center}
\end{figure}

\begin{figure}[!htbp]
\begin{center}
\subfigure[]{\includegraphics[width=0.45\textwidth]{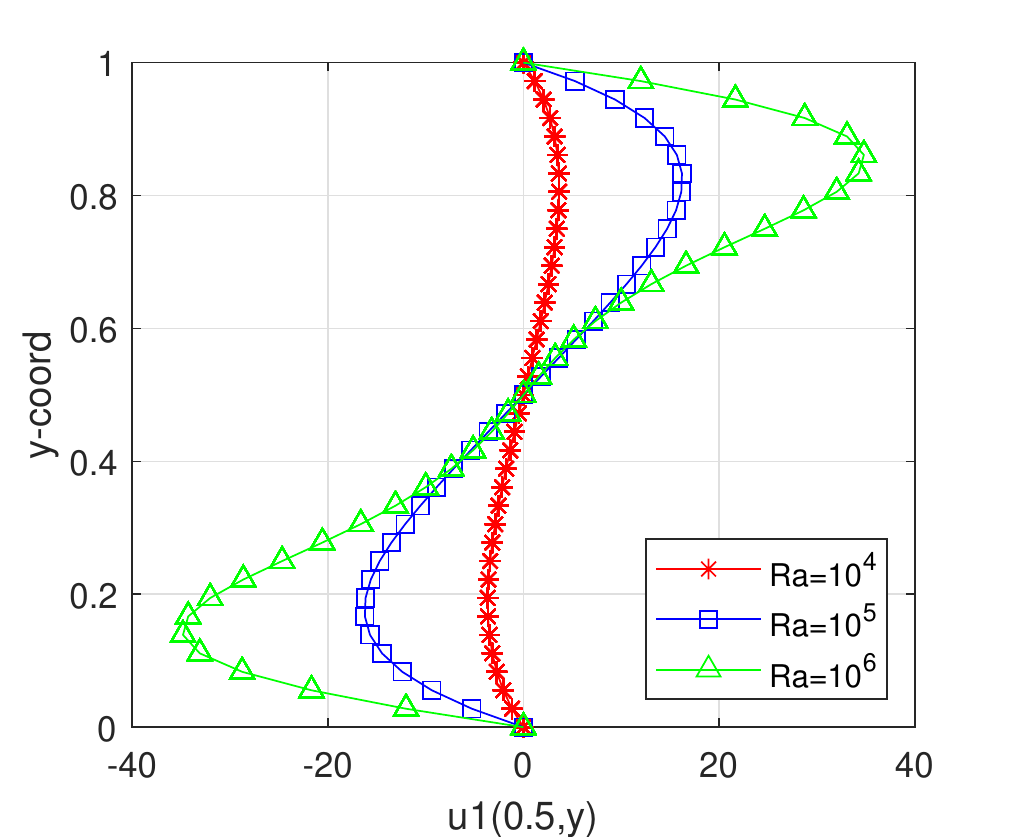}}
\subfigure[]{\includegraphics[width=0.45\textwidth]{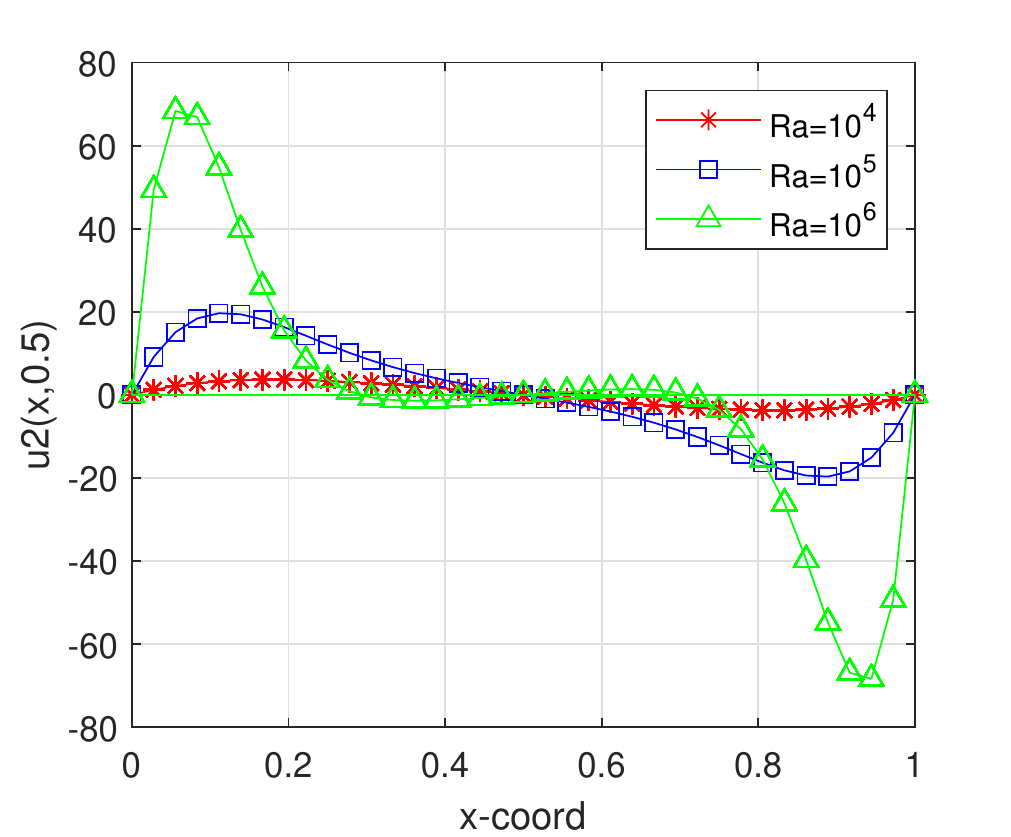}}
\caption{(a) Vertical velocity at mid-height $(x$$=$$0.5)$, (b) horizontal velocity at mid-width $(y$$=$$0.5)$.}
\end{center}
\end{figure}

\begin{figure}[!htbp]
\begin{center}
\subfigure[]{\includegraphics[width=0.32\textwidth]{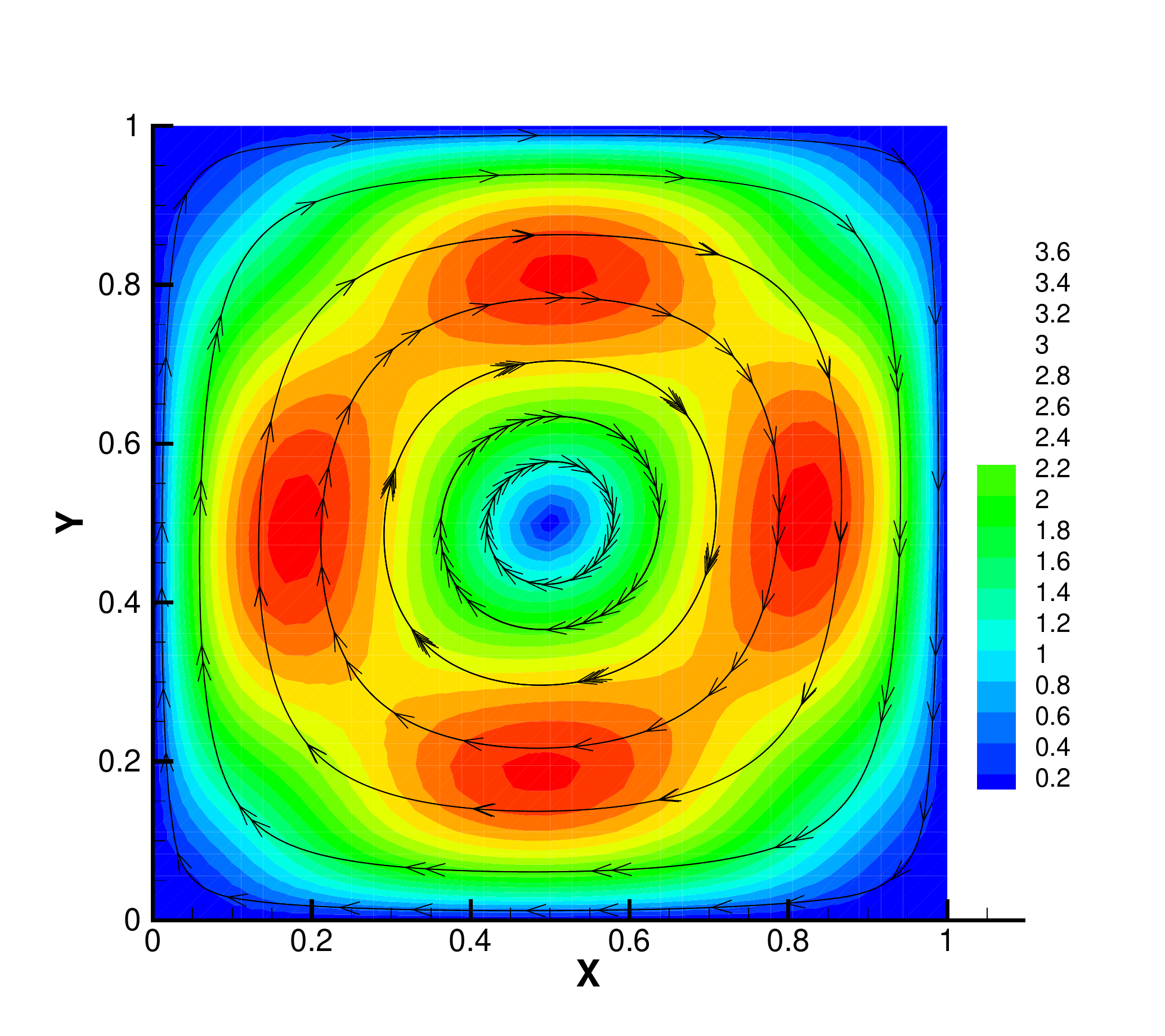}}
\subfigure[]{\includegraphics[width=0.32\textwidth]{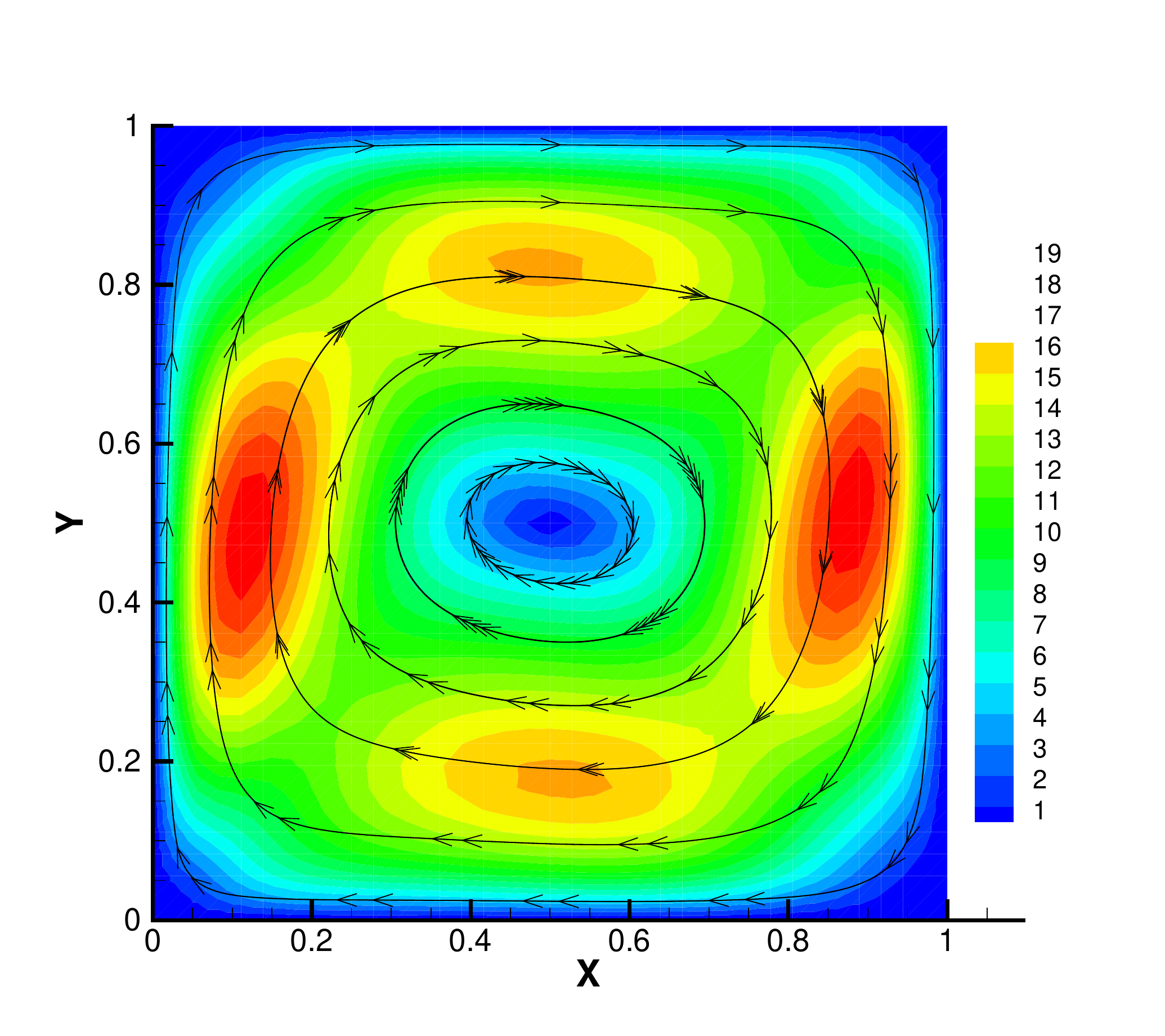}}
\subfigure[]{\includegraphics[width=0.32\textwidth]{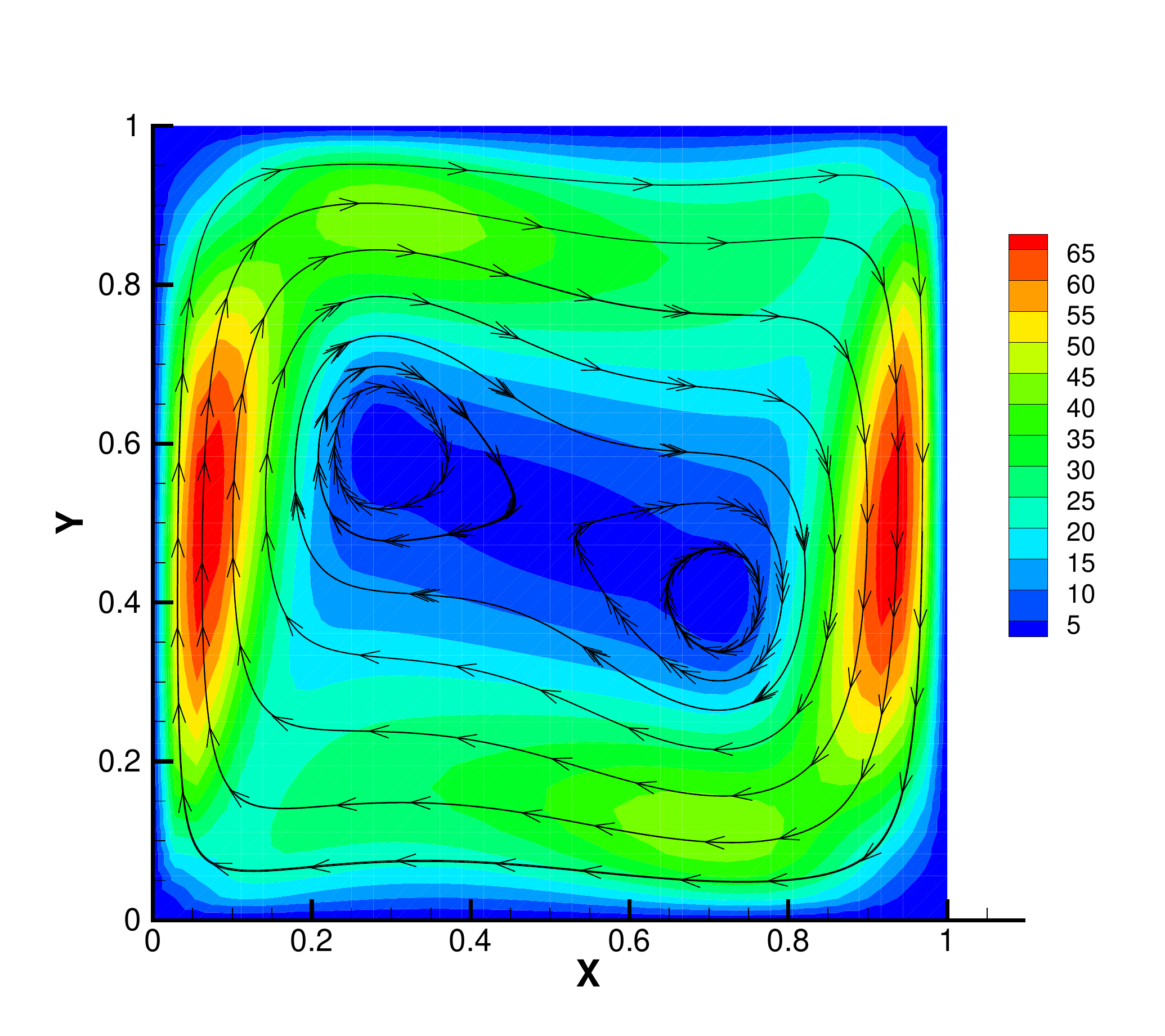}}
\caption{Velocity streamlines: (a) Ra$=10^3$, (b) Ra$=10^4$, (c) Ra$=10^5$.}
\end{center}
\end{figure}

\begin{figure}[!htbp]
\begin{center}
\subfigure[]{\includegraphics[width=0.32\textwidth]{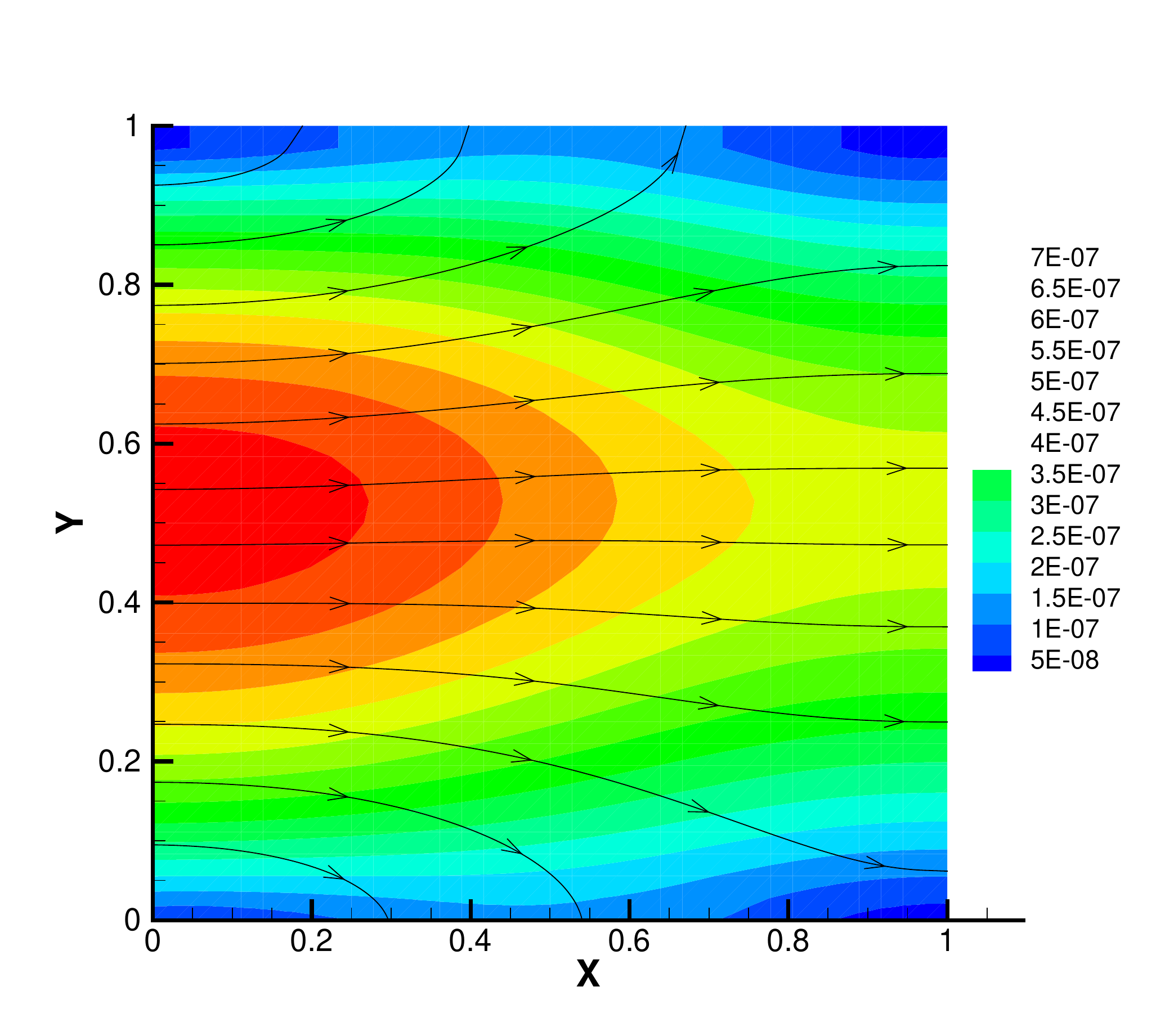}}
\subfigure[]{\includegraphics[width=0.32\textwidth]{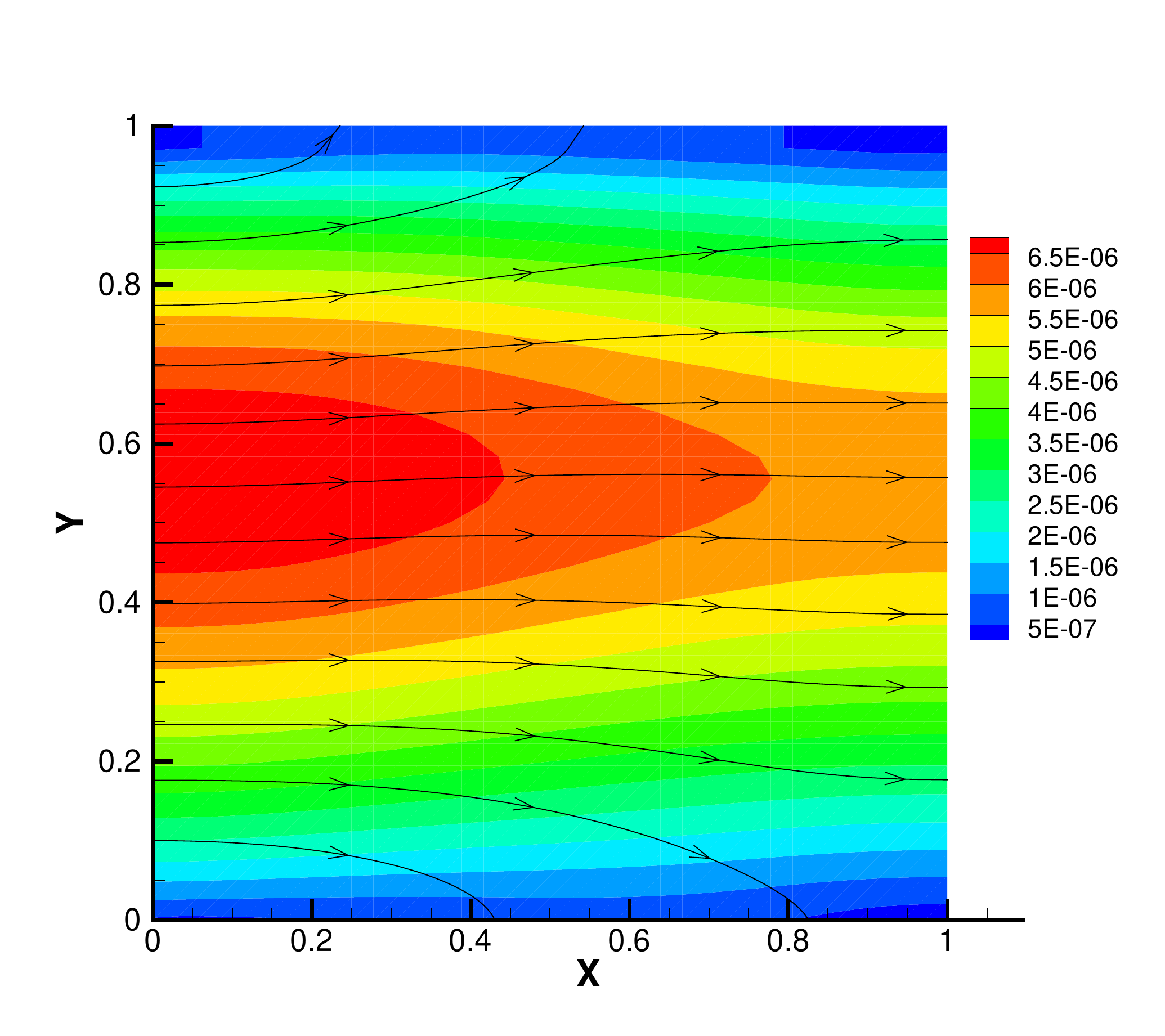}}
\subfigure[]{\includegraphics[width=0.32\textwidth]{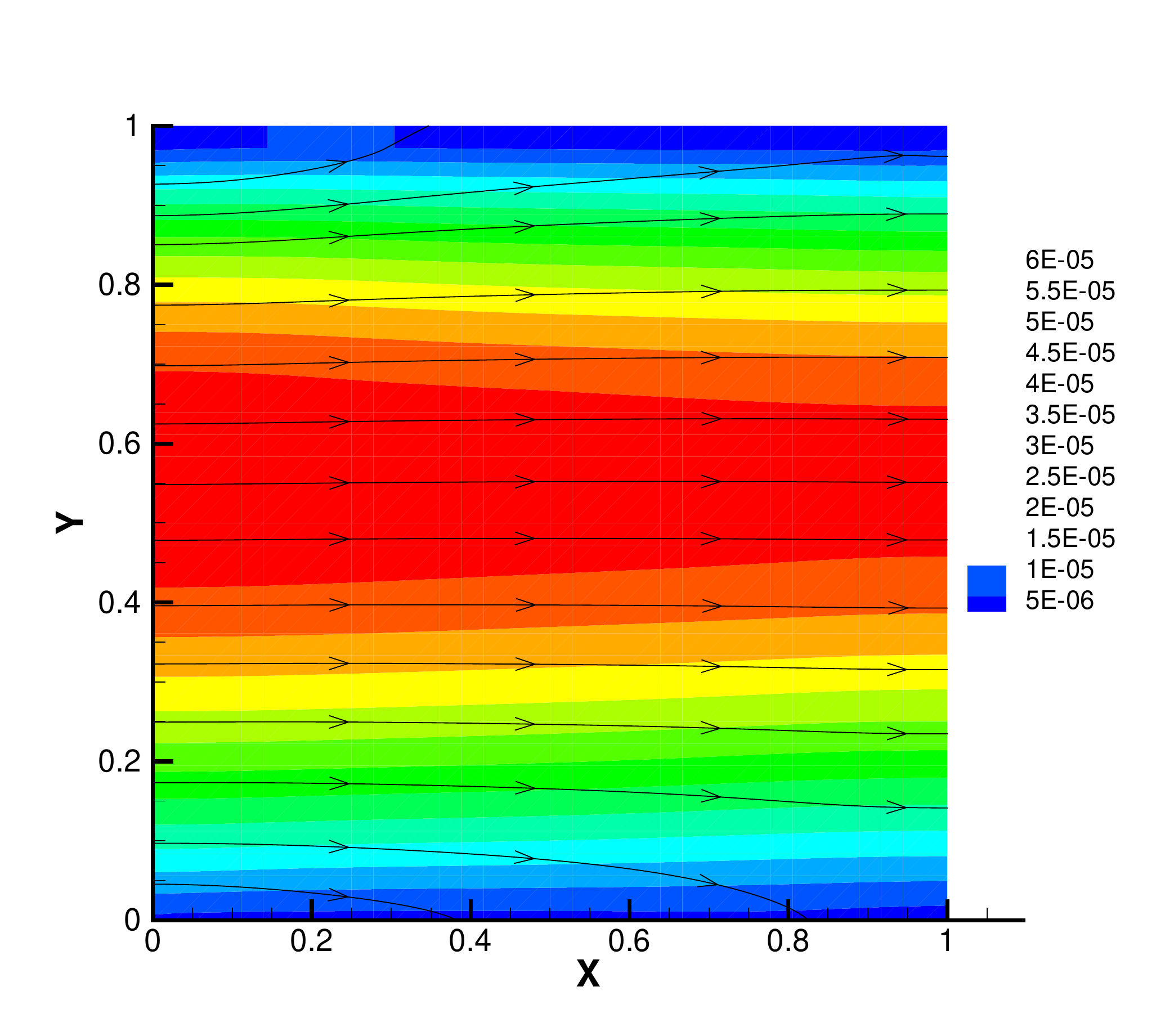}}
\caption{Current density streamline: (a) Ra$=10^3$, (b) Ra$=10^4$, (c) Ra$=10^5$.}
\end{center}
\end{figure}

%\begin{figure}[!htbp]
%\begin{center}
%\subfigure[]{\includegraphics[width=0.32\textwidth]{p_Pr=100,Ra=10^4,ka=1_Newton_P1b.eps}}
%\subfigure[]{\includegraphics[width=0.32\textwidth]{p_Pr=100,Ra=10^5,ka=1_Newton_P1b.eps}}
%\subfigure[]{\includegraphics[width=0.32\textwidth]{p_Pr=100,Ra=10^6,ka=1_Newton_P1b.eps}}
%\caption{Isobars:Ra=1e+4(a),1e+5(b),1e+6(c).}
%\end{center}
%\end{figure}

\begin{figure}[!htbp]
\begin{center}
\subfigure[]{\includegraphics[width=0.32\textwidth]{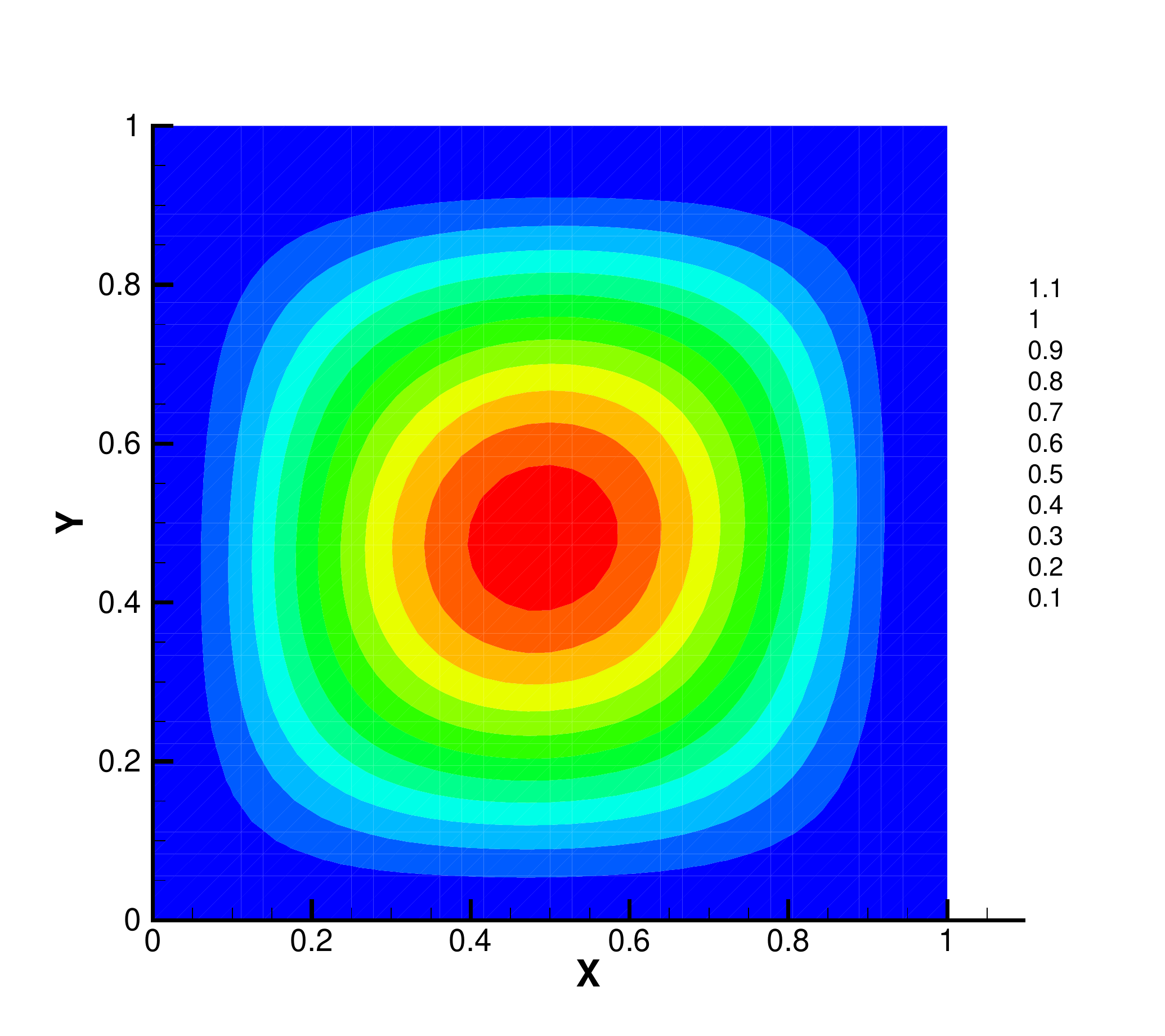}}
\subfigure[]{\includegraphics[width=0.32\textwidth]{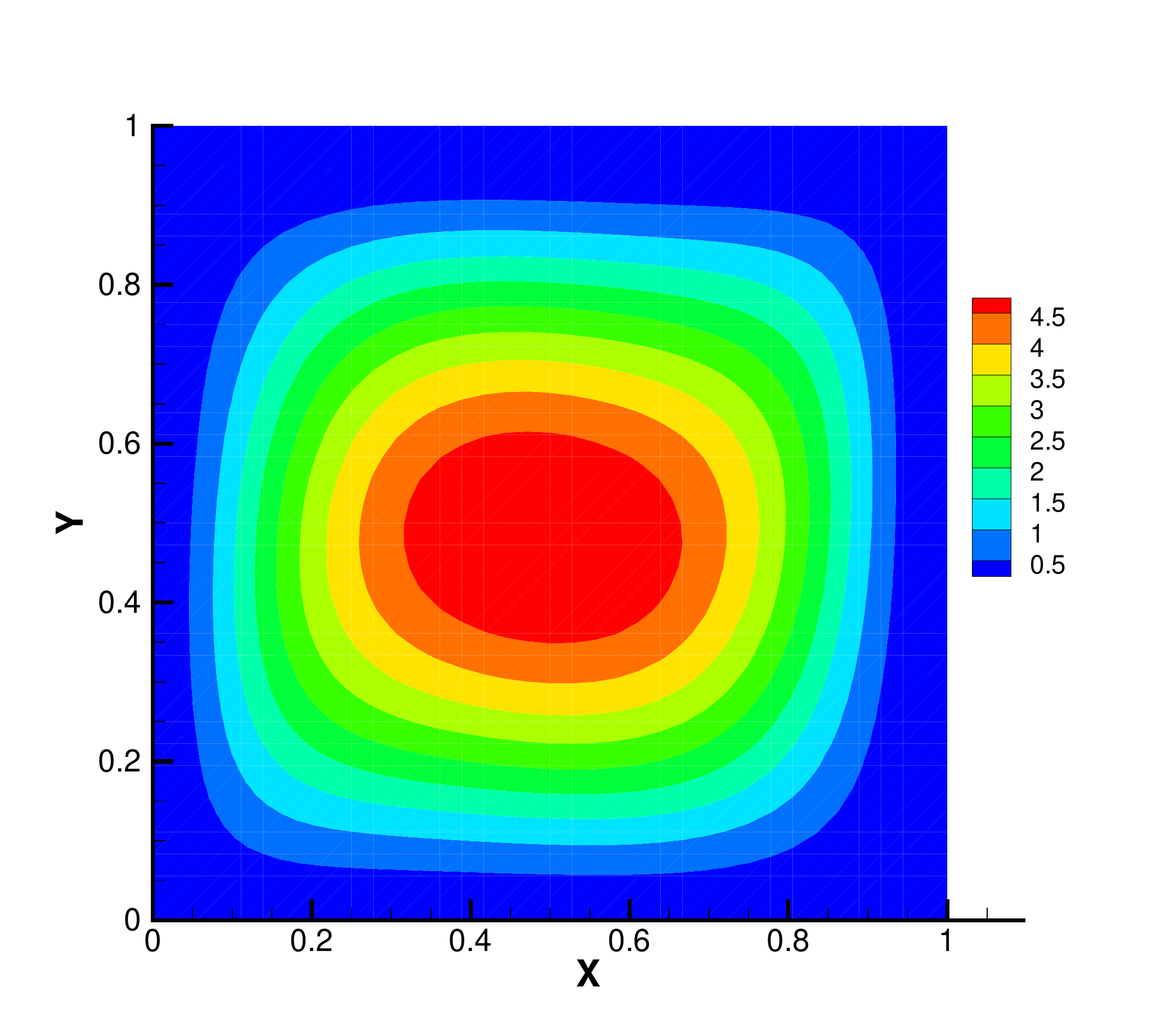}}
\subfigure[]{\includegraphics[width=0.32\textwidth]{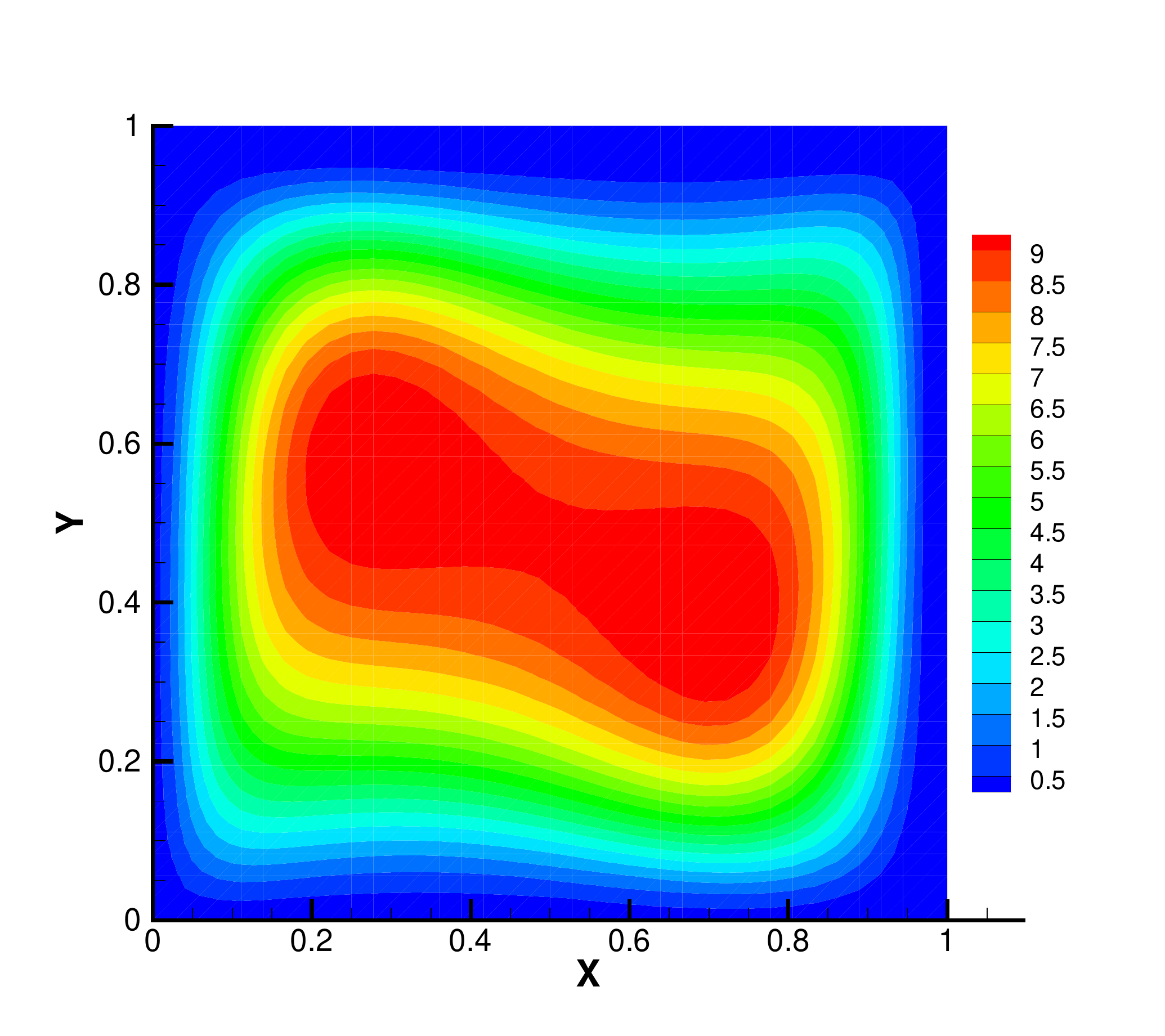}}
\caption{Isoline of electric potential: (a) Ra$=10^3$, (b) Ra$=10^4$, (c) Ra$=10^5.$}
\end{center}
\end{figure}

\begin{figure}[!htbp]
\begin{center}
\subfigure[]{\includegraphics[width=0.32\textwidth]{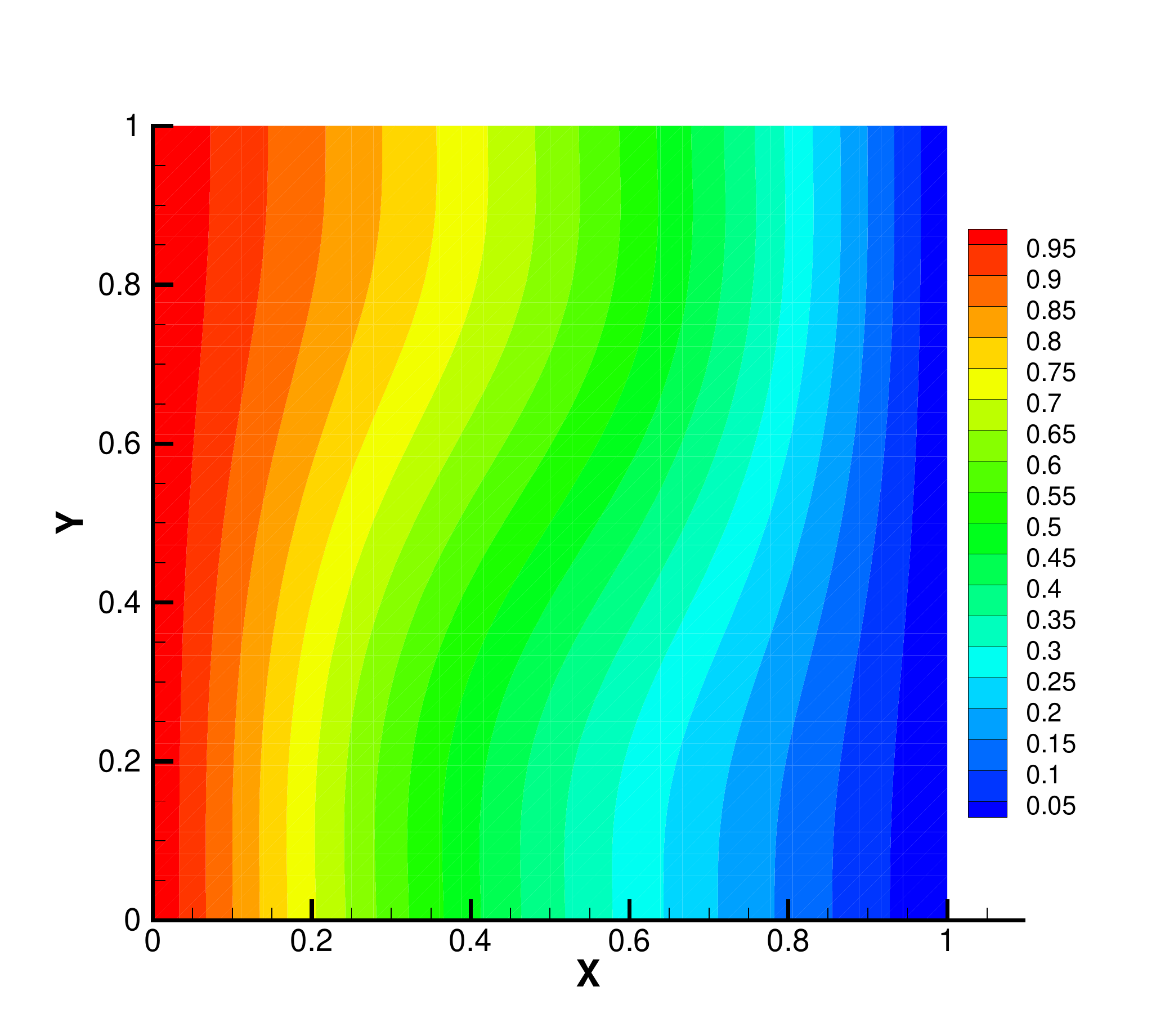}}
\subfigure[]{\includegraphics[width=0.32\textwidth]{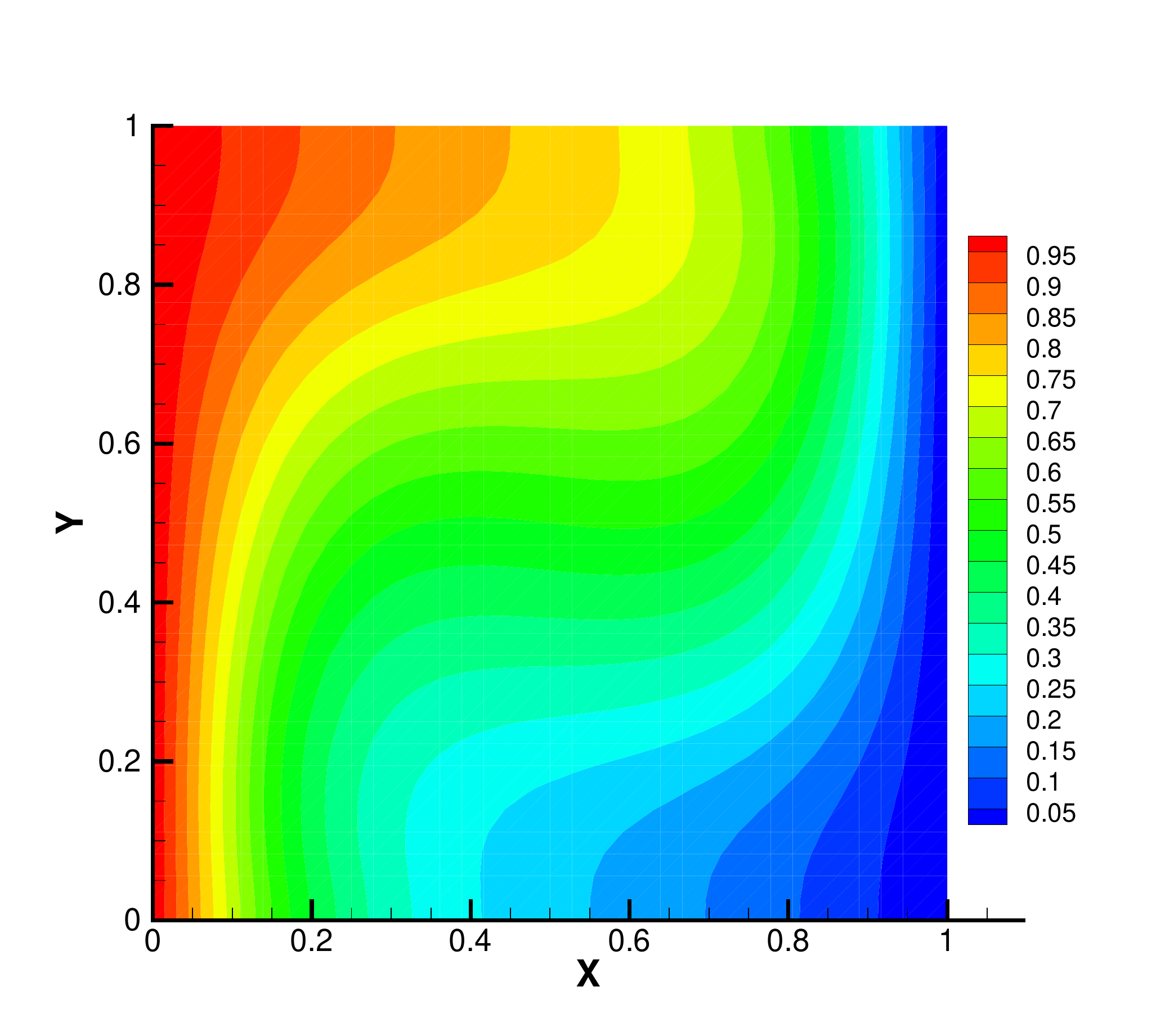}}
\subfigure[]{\includegraphics[width=0.32\textwidth]{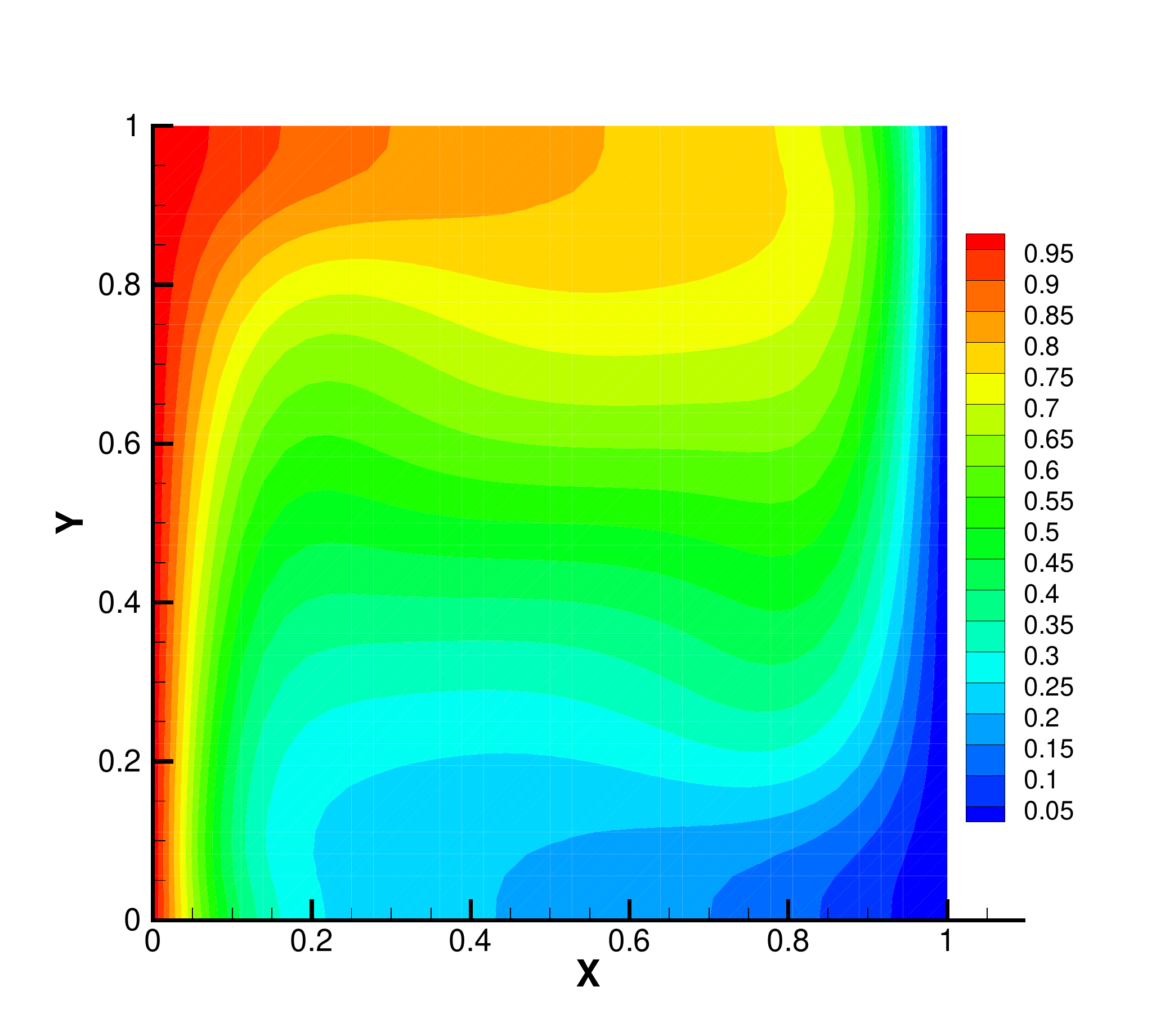}}
\caption{Isotherms: (a) Ra$=10^3$, (b) Ra$=10^4$, (c) Ra$=10^5.$}
\end{center}
\end{figure}

\begin{figure}

\begin{minipage}[t]{1\linewidth}
  \centering
  \includegraphics[width=4.5in,height=1in]{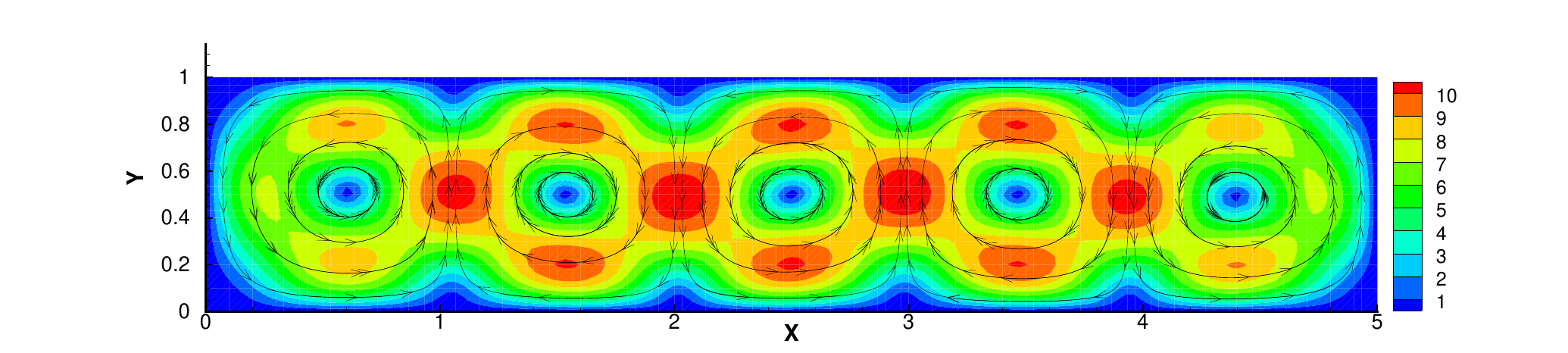}\\
\centering{(a)}\\
 \end{minipage}%

 %\begin{minipage}[t]{1\linewidth}
  %\centering
  %\includegraphics[width=4.5in,height=1in]{u_Ra=10^4_P1b_Oseen.eps}\\
%\centering{(b)}\\
 %\end{minipage}%

 \begin{minipage}[t]{1\linewidth}
  \centering
  \includegraphics[width=4.5in,height=1in]{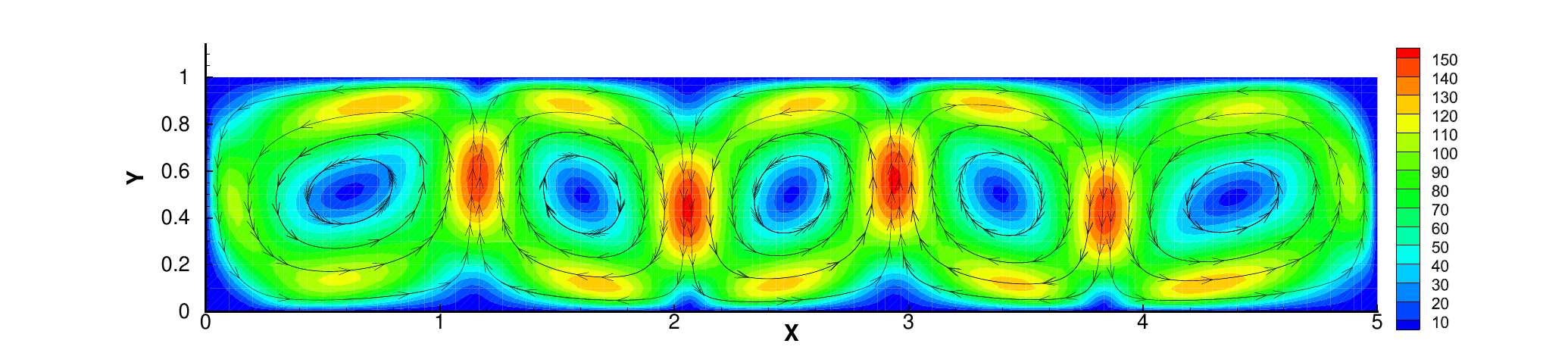}\\
\centering{(c)}\\
 \end{minipage}%
  \centering{\caption{Velocity streamlines with homogeneous heating: (a) Ra$=3\times 10^3$,(b) Ra=$10^5$.}}
\end{figure}

\begin{figure}

\begin{minipage}[t]{1\linewidth}
  \centering
  \includegraphics[width=4.5in,height=1in]{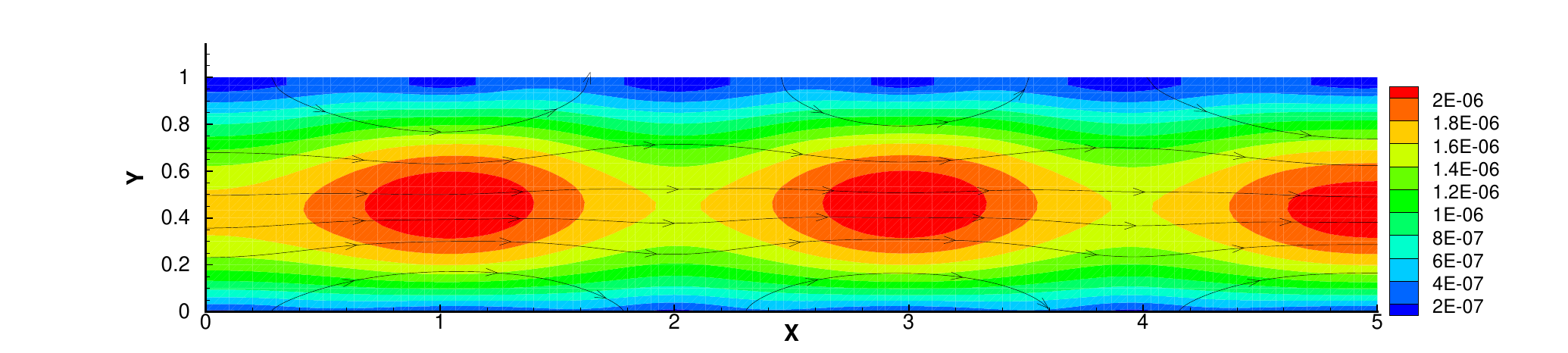}\\
\centering{(a)}\\
 \end{minipage}%

 %\begin{minipage}[t]{1\linewidth}
  %\centering
  %\includegraphics[width=4.5in,height=1in]{J_Ra=10^4_P1b_Oseen.eps}\\
%\centering{(b)}\\
 %\end{minipage}%

 \begin{minipage}[t]{1\linewidth}
  \centering
  \includegraphics[width=4.5in,height=1in]{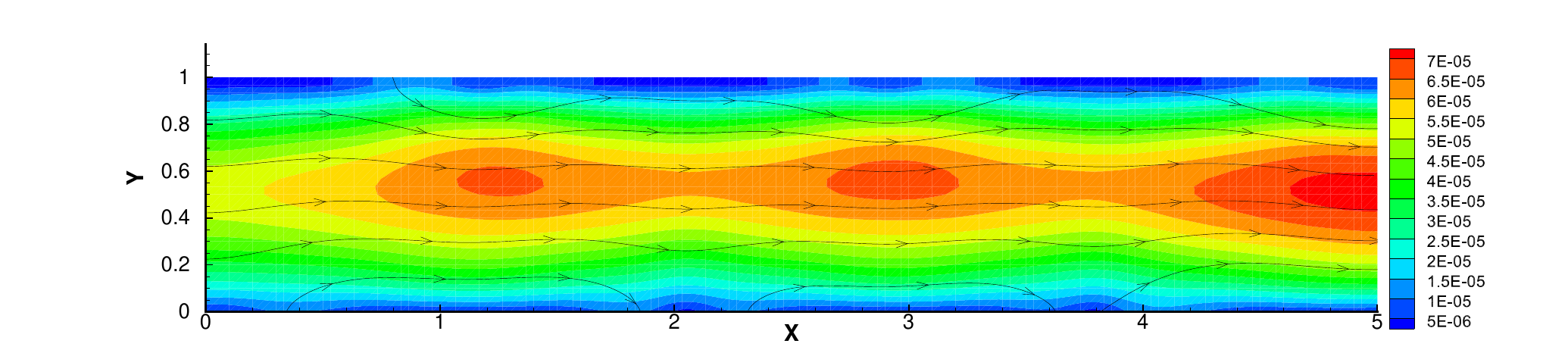}\\
\centering{(c)}\\
 \end{minipage}%
  \centering{\caption{Current density streamline with homogeneous heating: (a) Ra$=3\times 10^3$, (b) Ra=$10^5$.}}
\end{figure}

\begin{figure}

\begin{minipage}[t]{1\linewidth}
  \centering
  \includegraphics[width=4.5in,height=1in]{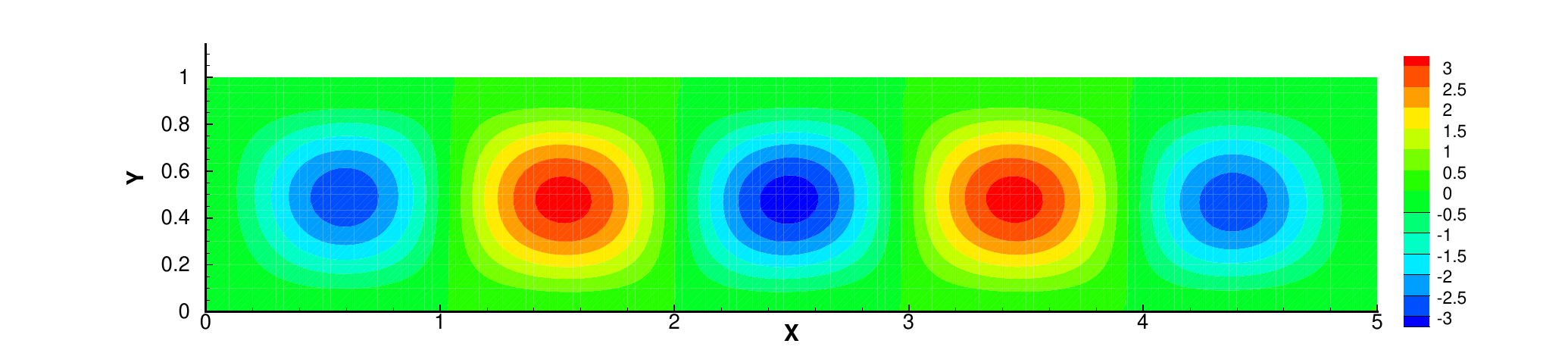}\\
\centering{(a)}\\
 \end{minipage}%

 %\begin{minipage}[t]{1\linewidth}
  %\centering
  %\includegraphics[width=4.5in,height=1in]{phi_Ra=10^4_P1b_Oseen.eps}\\
%\centering{(b)}\\
 %\end{minipage}%

 \begin{minipage}[t]{1\linewidth}
  \centering
  \includegraphics[width=4.5in,height=1in]{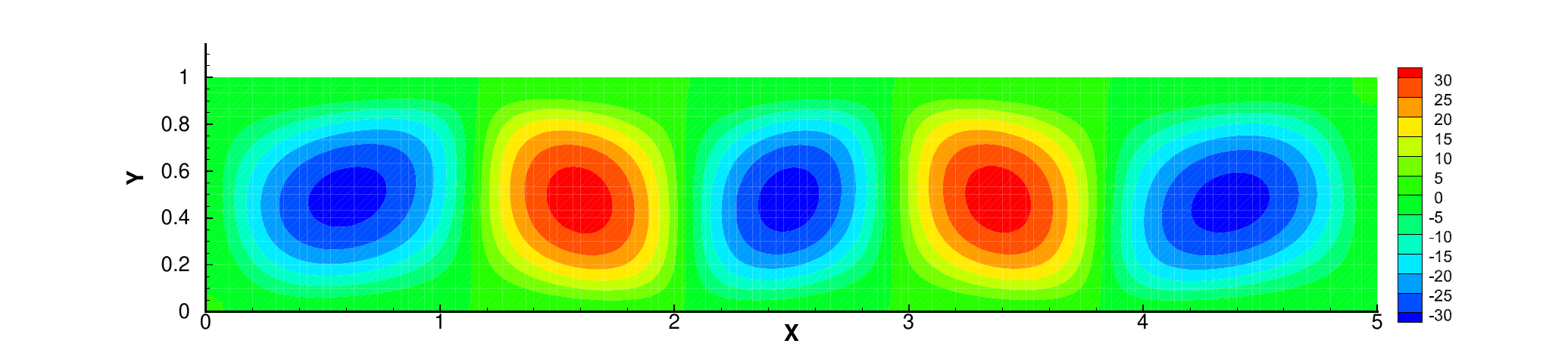}\\
\centering{(c)}\\
 \end{minipage}%
  \centering{\caption{Isoline of electric potential with homogeneous heating: (a) Ra$=3\times 10^3$, (b) Ra=$10^5$.}}
\end{figure}

\begin{figure}

\begin{minipage}[t]{1\linewidth}
  \centering
  \includegraphics[width=4.5in,height=1in]{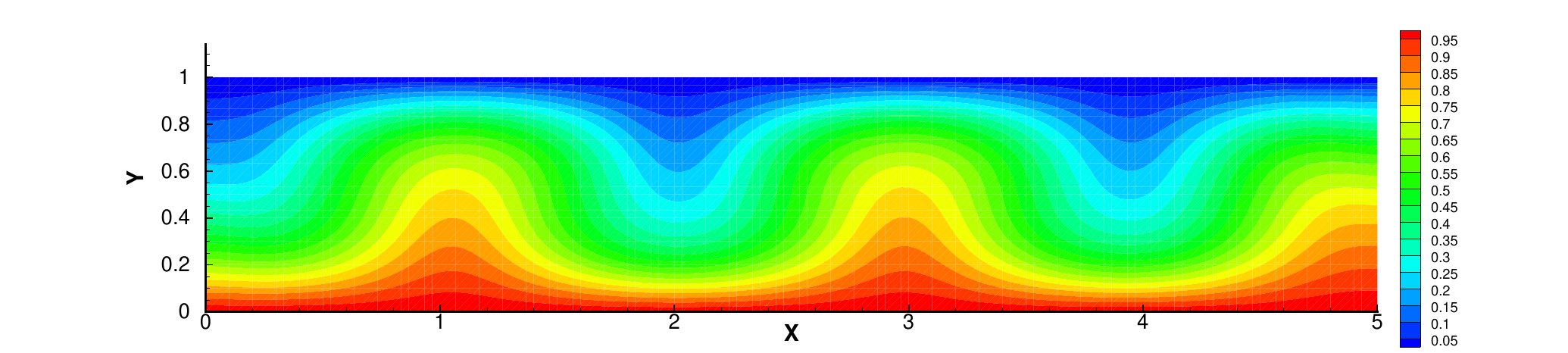}\\
\centering{(a)}\\
 \end{minipage}%

 %\begin{minipage}[t]{1\linewidth}
  %\centering
  %\includegraphics[width=4.5in,height=1in]{theta_Ra=10^4_P1b_Oseen.eps}\\
%\centering{(b)}\\
 %\end{minipage}%

 \begin{minipage}[t]{1\linewidth}
  \centering
  \includegraphics[width=4.5in,height=1in]{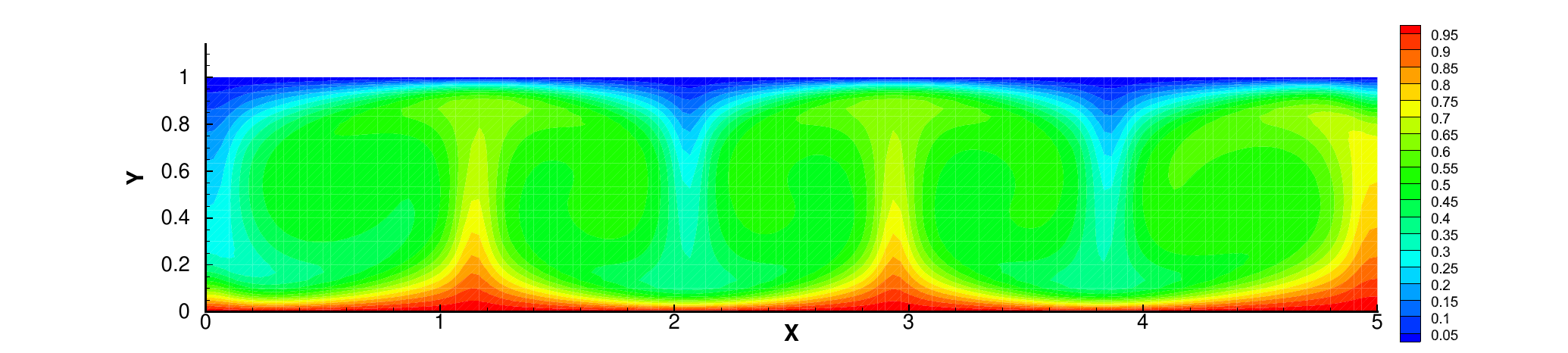}\\
\centering{(c)}\\
 \end{minipage}%
  \centering{\caption{Isotherms with homogeneous heating: (a) Ra$=3\times 10^3$, (b) Ra=$10^5$.}}
\end{figure}

\begin{figure}

\begin{minipage}[t]{1\linewidth}
  \centering
  \includegraphics[width=4.5in,height=1in]{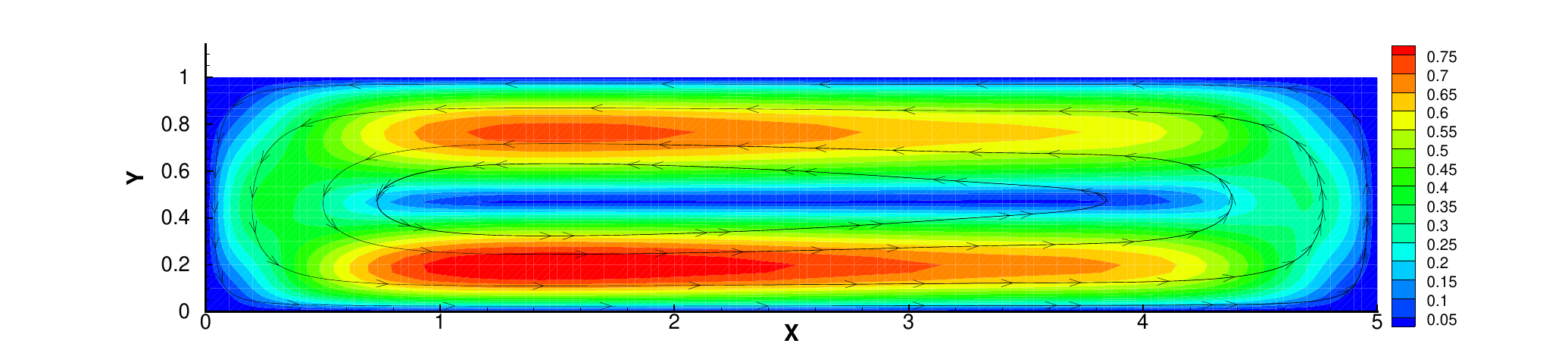}\\
\centering{(a)}\\
 \end{minipage}%

 %\begin{minipage}[t]{1\linewidth}
  %\centering
  %\includegraphics[width=4.5in,height=1in]{u_Ra=5000_P1b_Oseen_T.eps}\\
%\centering{(b)}\\
 %\end{minipage}%

 \begin{minipage}[t]{1\linewidth}
  \centering
  \includegraphics[width=4.5in,height=1in]{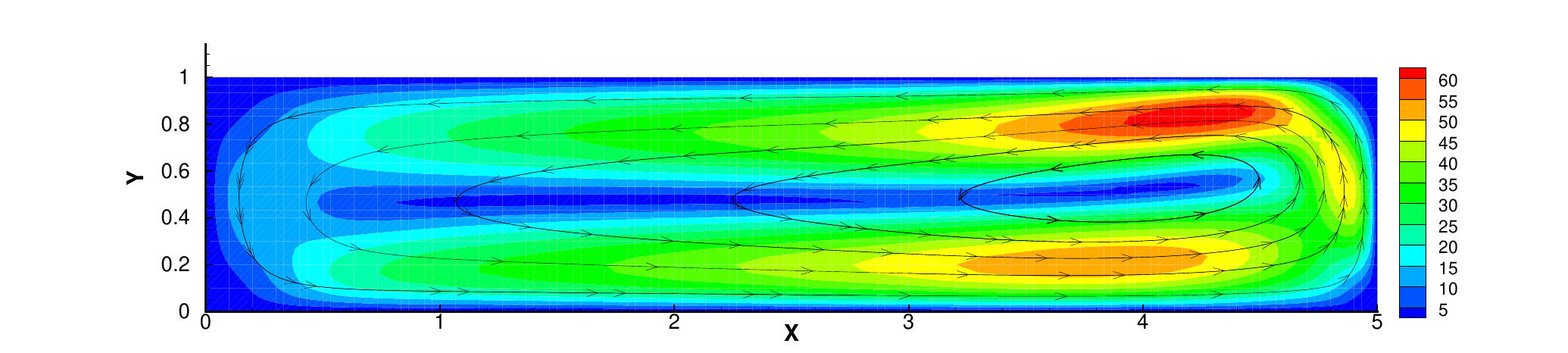}\\
\centering{(c)}\\
 \end{minipage}%
  \centering{\caption{Velocity streamlines with non-uniform heating: (a) Ra$=10^3$, (b) Ra=$5\times10^4$.}}
\end{figure}

\begin{figure}

\begin{minipage}[t]{1\linewidth}
  \centering
  \includegraphics[width=4.5in,height=1in]{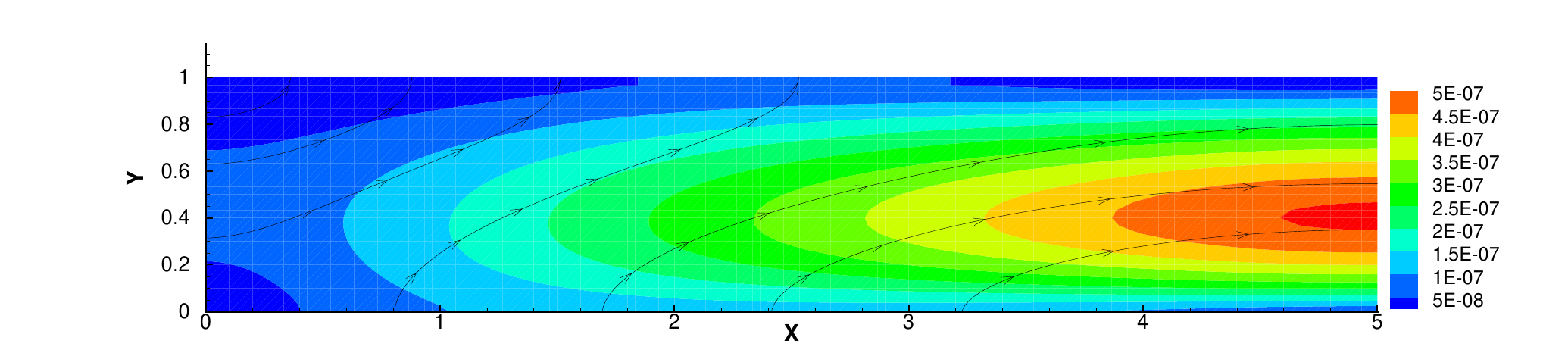}\\
\centering{(a)}\\
 \end{minipage}%

 %\begin{minipage}[t]{1\linewidth}
  %\centering
  %\includegraphics[width=4.5in,height=1in]{J_Ra=5000_P1b_Oseen_T.eps}\\
%\centering{(b)}\\
 %\end{minipage}%

 \begin{minipage}[t]{1\linewidth}
  \centering
  \includegraphics[width=4.5in,height=1in]{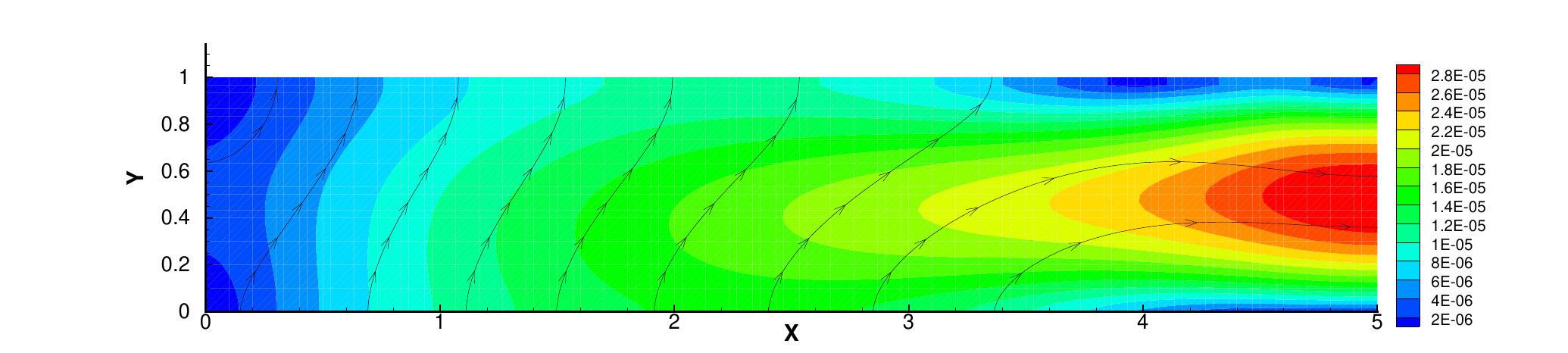}\\
\centering{(c)}\\
 \end{minipage}%
  \centering{\caption{Current density streamline with non-uniform heating: (a) Ra$=10^3$, (b) Ra=$5\times10^4$.}}
\end{figure}

\begin{figure}

\begin{minipage}[t]{1\linewidth}
  \centering
  \includegraphics[width=4.5in,height=1in]{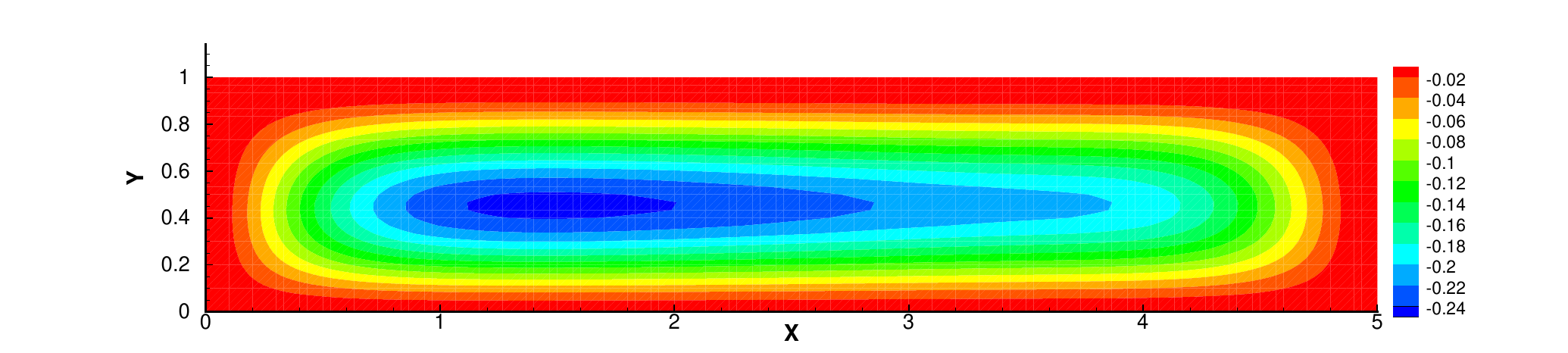}\\
\centering{(a)}\\
 \end{minipage}%

 %\begin{minipage}[t]{1\linewidth}
  %\centering
  %\includegraphics[width=4.5in,height=1in]{phi_Ra=5000_P1b_Oseen_T.eps}\\
%\centering{(b)}\\
 %\end{minipage}%

 \begin{minipage}[t]{1\linewidth}
  \centering
  \includegraphics[width=4.5in,height=1in]{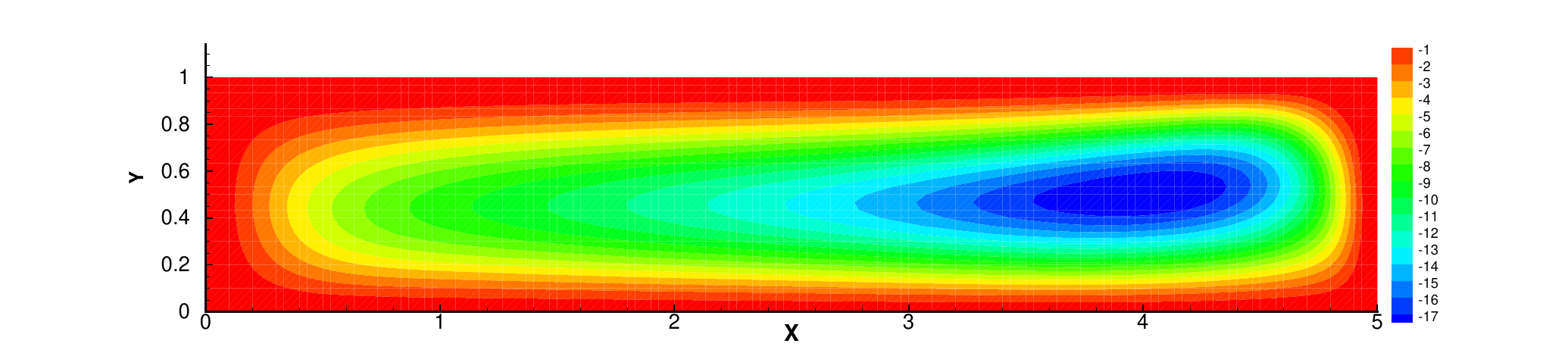}\\
\centering{(c)}\\
 \end{minipage}%
  \centering{\caption{Isoline of electric potential with non-uniform heating: (a) Ra$=10^3$, (b) Ra=$5\times10^4$.}}
\end{figure}

\begin{figure}

\begin{minipage}[t]{1\linewidth}
  \centering
  \includegraphics[width=4.5in,height=1in]{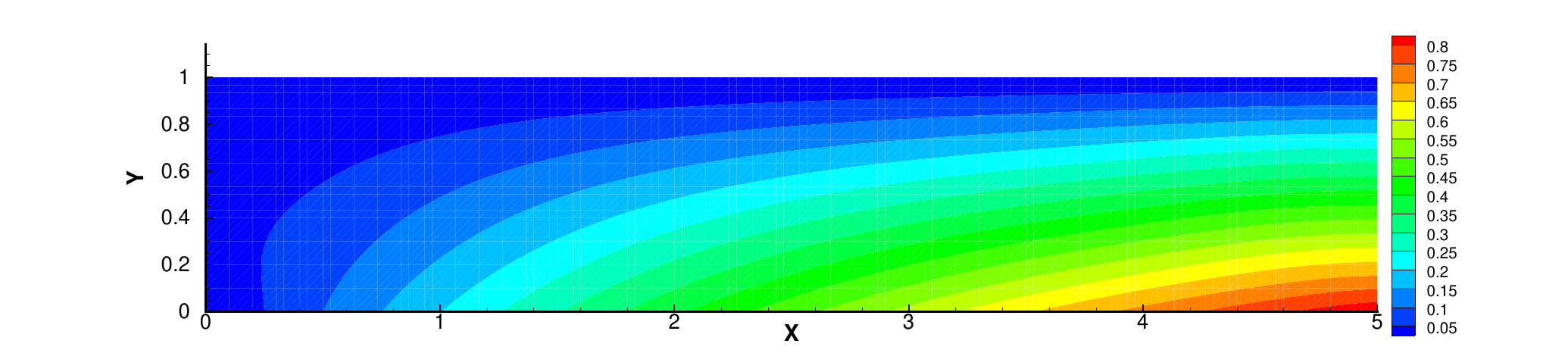}\\
\centering{(a)}\\
 \end{minipage}%

 %\begin{minipage}[t]{1\linewidth}
  %\centering
  %\includegraphics[width=4.5in,height=1in]{theta_Ra=5000_P1b_Oseen_T.eps}\\
%\centering{(b)}\\
 %\end{minipage}%

 \begin{minipage}[t]{1\linewidth}
  \centering
  \includegraphics[width=4.5in,height=1in]{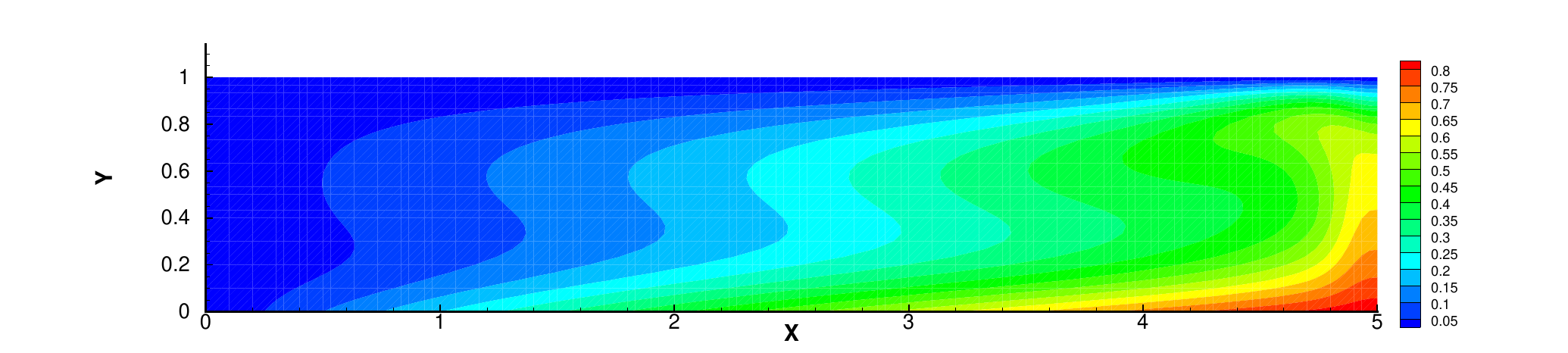}\\
\centering{(c)}\\
 \end{minipage}%
  \centering{\caption{Isotherms with non-uniform heating: (a) Ra$=10^3$, (b) Ra=$5\times10^4$.}}
\end{figure}
%\end{comment}

\section{Conclusion}
In this paper, three iterative methods for steady-state thermally coupled inductionless MHD problems are theoretically analyzed. In the finite element discretization, we solve simultaneously the current density and potential, and use respectively conforming face elements in $\bm{H}(\mbox{div},\Omega)$ along with conforming volume elements in $L^2(\Omega)$ to discretize them. This results in an accurate divergence-free for the discrete current density. Under the assumption of weak regularity, we give further the optimal error estimates, where the error bounds of velocity, pressure, charge density and temperature are independent of potential. After that, we propose and analyze three coupling iterations of the scheme. Finally, through numerical experiments, we verify the results of our theoretical analysis and prove the effectiveness of our proposed iterative method. In our following research, a two-level iterative method for thermally coupled inductionless MHD equations will be considered.

%\bibliographystyle{plain}%plain,unsrt,agsm
%\bibliography{paper}

\begin{thebibliography}{10}

\bibitem{2001On}
M.~A. Abdou, A.~Ying, and N.~B. Morley.
\newblock On the exploration of innovative concepts for fusion chamber
  technology, apex interim report overview.
\newblock {\em Fusion Engineering and Design}, 73:83--93, 2001.

\bibitem{abdou2001exploration}
MA~Abdou, A~Ying, N~Morley, K~Gulec, S~Smolentsev, M~Kotschenreuther, S~Malang,
  S~Zinkle, T~Rognlien, P~Fogarty, et~al.
\newblock On the exploration of innovative concepts for fusion chamber
  technology.
\newblock {\em Fusion Engineering and Design}, 54(2):181--247, 2001.

\bibitem{0a}
L.~Abstract, M.~Li, W.~Ni, and Zheng.
\newblock a charge-conservative finite element method for inductionless mhd
  equations. part ii: a robust solver ast.

\bibitem{2014Block}
Santiago Badia, Alberto F.Martin, and Ramon Planas.
\newblock Block recursive LU preconditioners for the thermally coupled
  incompressible inductionless mhd problem.
\newblock {\em Journal of Computational Physics}, 274:562--591, 2014.

\bibitem{1991Mixed}
F.~Brezzi and M.~Fortin.
\newblock {\em Mixed and hybrid finite element methods}.
\newblock Springer-Verlag,, 1991.

\bibitem{davidson2002introduction}
Peter~Alan Davidson.
\newblock An introduction to magnetohydrodynamics, 2002.

\bibitem{Evans1964Partial}
Evans and C.~Lawrence.
\newblock Partial differential equations.
\newblock {\em Intersxcience Publishers}, 1964.

\bibitem{gerbeau2006mathematical}
Jean-Fr{\'e}d{\'e}ric Gerbeau, Claude Le~Bris, and Tony Leli{\`e}vre.
\newblock {\em Mathematical methods for the magnetohydrodynamics of liquid
  metals}.
\newblock Clarendon Press, 2006.

\bibitem{1986Finite}
V.~Girault and P.~A. Raviart.
\newblock Finite element methods for navier-stokes equations: Theory and
  algorithms.
\newblock {\em NASA STI/Recon Technical Report A}, 87, 1986.

\bibitem{greif2010mixed}
Chen Greif, Dan Li, Dominik Sch{\"o}tzau, and Xiaoxi Wei.
\newblock A mixed finite element method with exactly divergence-free velocities
  for incompressible magnetohydrodynamics.
\newblock {\em Computer Methods in Applied Mechanics and Engineering},
  199(45-48):2840--2855, 2010.

\bibitem{MD1991On}
MD~Gunzburger, A.~J. Meir, and J.~S. Peterson.
\newblock On the existence, uniqueness, and finite element approximation of
  solutions of the equations of stationary, incompressible
  magnetohydrodynamics.
\newblock {\em Mathematics of Computation}, 56(194):523--563, 1991.

\bibitem{2018A}
R.~Hiptmair, L.~Li, S.~Mao, and W.~Zheng.
\newblock A fully divergence-free finite element method for magnetohydrodynamic
  equations.
\newblock {\em Mathematical Models and Methods in Applied ences}, pages 1--37,
  2018.

\bibitem{hu2017stable}
Kaibo Hu, Yicong Ma, and Jinchao Xu.
\newblock Stable finite element methods preserving $\nabla\cdot\bm{B}$ exactly
  for mhd models.
\newblock {\em Numerische Mathematik}, 135(2):371--396, 2017.

\bibitem{hughes1966electromagnetodynamics}
William~F Hughes and Frederick~John Young.
\newblock The electromagnetodynamics of fluids.
\newblock {\em The electromagnetodynamics of fluids}, 1966.

\bibitem{layton1994two}
W~Layton, HWJ Lenferink, and JS~Peterson.
\newblock A two-level newton, finite element algorithm for approximating
  electrically conducting incompressible fluid flows.
\newblock {\em Computers \& Mathematics with Applications}, 28(5):21--31, 1994.

\bibitem{lifshits2012magnetohydrodynamics}
Alexander~E Lifshits.
\newblock {\em Magnetohydrodynamics and spectral theory}, volume~4.
\newblock Springer Science \& Business Media, 2012.

\bibitem{long2022convergence}
Xiaonian Long and Qianqian Ding.
\newblock Convergence analysis of a conservative finite element scheme for the
  thermally coupled incompressible inductionless mhd problem.
\newblock {\em Applied Numerical Mathematics}, 182:176--195, 2022.

\bibitem{meir1995thermally}
AJ~Meir.
\newblock Thermally coupled, stationary, incompressible mhd flow; existence,
  uniqueness, and finite element approximation.
\newblock {\em Numerical Methods for Partial Differential Equations},
  11(4):311--337, 1995.

\bibitem{moreau1990magnetohydrodynamics}
Ren{\'e}~J Moreau.
\newblock {\em Magnetohydrodynamics}, volume~3.
\newblock Springer Science \& Business Media, 1990.

\bibitem{ni2012consistent}
Ming-Jiu Ni and Jun-Feng Li.
\newblock A consistent and conservative scheme for incompressible mhd flows at
  a low magnetic reynolds number. part iii: On a staggered mesh.
\newblock {\em Journal of Computational Physics}, 231(2):281--298, 2012.

\bibitem{ni2007current}
Ming-Jiu Ni, Ramakanth Munipalli, Peter Huang, Neil~B Morley, and Mohamed~A
  Abdou.
\newblock A current density conservative scheme for incompressible mhd flows at
  a low magnetic reynolds number. part ii: On an arbitrary collocated mesh.
\newblock {\em Journal of Computational Physics}, 227(1):205--228, 2007.

\bibitem{ni2007current1}
Ming-Jiu Ni, Ramakanth Munipalli, Neil~B Morley, Peter Huang, and Mohamed~A
  Abdou.
\newblock A current density conservative scheme for incompressible mhd flows at
  a low magnetic reynolds number. part i: On a rectangular collocated grid
  system.
\newblock {\em Journal of Computational Physics}, 227(1):174--204, 2007.

\bibitem{20101On}
Janet~S. Peterson.
\newblock On the finite element approximation of incompressible flows of an
  electrically conducting fluid.
\newblock {\em Numerical Methods for Partial Differential Equations},
  4(1):57--68, 2010.

\bibitem{ravindran2018decoupled}
SS~Ravindran.
\newblock A decoupled crank-nicolson time-stepping scheme for thermally coupled
  magneto-hydrodynamic system.
\newblock {\em An International Journal of Optimization and Control: Theories
  \& Applications (IJOCTA)}, 8(1):43--62, 2018.

\bibitem{2021Coupled}
Xiaodi Zhang and Qianqian Ding.
\newblock Coupled iterative analysis for stationary inductionless
  magnetohydrodynamic system based on charge-conservative finite element
  method.
\newblock 2021.

\end{thebibliography}
%\input{paper.bbl}

\end{document}